\documentclass[a4paper,10pt,leqno, english]{amsart}
        \title[Equivariant principal bundles  and their classifying spaces]
        {Equivariant principal bundles and their classifying spaces}
          \author{Wolfgang L\"uck}
          \address{Rheinische Wilhelms-Universit\"at Bonn\\
               Mathematisches Institut\\
               Endenicher Allee 60, 53115 Bonn, Germany}      
       \email{wolfgang.lueck@him.uni-bonn.de}
      \urladdr{http://www.him.uni-bonn.de/lueck/}
      \author{Bernardo Uribe}
      \address{Departamento de Matem\'{a}ticas, Universidad del Norte, Km.5 V\'ia Puerto Colombia, 
Barranquilla, Colombia}
      \email{buribe@gmail.com}
            \thanks{This work was partially funded by the Deutsche Forschungsgemeinschaft and the Alexander Von Humboldt 
      Foundation}

\date{January 7, 2014}

\usepackage{pdfsync}
\usepackage{hyperref}
\usepackage{calc}
\usepackage{enumerate,amssymb}
\usepackage{color}
\usepackage[arrow,curve,matrix,tips,2cell]{xy}
  \SelectTips{eu}{10} \UseTips
  \UseAllTwocells
\usepackage{tikz}
\usepackage{pdfcolmk}

\DeclareMathAlphabet{\matheurm}{U}{eur}{m}{n}

\DeclareMathOperator{\aut}{aut}

\DeclareMathOperator{\Bundle}{Bundle}

\DeclareMathOperator{\coker}{coker}

\DeclareMathOperator{\colim}{colim}

\DeclareMathOperator{\ev}{ev}
\DeclareMathOperator{\Ext}{Ext}
\DeclareMathOperator{\Fred}{Fred}

\DeclareMathOperator{\id}{id}
\DeclareMathOperator{\Index}{index}
\DeclareMathOperator{\im}{im}

\DeclareMathOperator{\map}{map}

\DeclareMathOperator{\pr}{pr}

\newcommand{\calCOM}{{\mathcal C}{\mathcal O}{\mathcal M}}

  \newcommand{\IR}{\mathbb{R}}

  \newcommand{\IZ}{\mathbb{Z}}

  \newcommand{\cala}{\mathcal{A}}
  \newcommand{\calb}{\mathcal{B}}

  \newcommand{\calf}{\mathcal{F}}

  \newcommand{\calh}{\mathcal{H}}
  \newcommand{\cali}{\mathcal{I}}

  \newcommand{\caln}{\mathcal{N}}
  
  \newcommand{\calp}{\mathcal{P}}
  
  \newcommand{\calr}{\mathcal{R}}
  \newcommand{\cals}{\mathcal{S}}
  
  \newcommand{\caltr}{\mathcal{T}\mathcal{R}}
  \newcommand{\calu}{\mathcal{U}}

\newcommand{\EGF}[2]{E_{#2}(#1)}

\newcounter{commentcounter}

\theoremstyle{plain}
\newtheorem{theorem}{Theorem}[section]

\newtheorem{lemma}[theorem]{Lemma}

\newtheorem{corollary}[theorem]{Corollary}
\newtheorem{proposition}[theorem]{Proposition}

\newtheorem*{theorem*}{Theorem}
\newtheorem*{mtheorem*}{Main Theorem}

\theoremstyle{definition}
\newtheorem{definition}[theorem]{Definition}

\newtheorem{example}[theorem]{Example}

\newtheorem{remark}[theorem]{Remark}
\newtheorem{notation}[theorem]{Notation}

\theoremstyle{remark}
\newtheorem*{summary*}{Summary}

\renewcommand{\labelenumi}{(\roman{enumi})}

\makeatletter\let\c@equation=\c@theorem\makeatother

\newcommand{\pt}{\{\bullet\}}

\newcommand{\version}[1] 
{\begin{center} Last edited on #1\\
    Last compiled on \today\\
file name: \jobname
  \end{center}
}

\begin{document}

\begin{abstract}
  We consider $\Gamma$-equivariant principal $G$-bundles over proper $\Gamma$-$CW$-complexes
  with prescribed family of local
  representations. We construct and analyze their classifying spaces for locally compact, second
  countable topological groups with finite covering dimension $\Gamma$ and $G$, such that
  $G$ is almost connected.
\end{abstract}

\maketitle

\newlength{\origlabelwidth} \setlength\origlabelwidth\labelwidth


\section*{Introduction}

Let $\Gamma$ and $G$ be topological groups. We will introduce the notion of a
\emph{$\Gamma$-equivariant principle $G$-bundle} over a $\Gamma$-$CW$-complex, i.e. a
principal $G$-bundle $p \colon E \to B$ together with left $\Gamma$-actions on $E$ and $B$
commuting with the right $G$-action on $E$ such that $p$ is $\Gamma$-equivariant. For
every $e \in E$ we obtain a \emph{local representation} $\rho_e \colon \Gamma_{p(e)} \to
G$ uniquely determined by $\gamma^{-1} \cdot e = e \cdot \rho_e(\gamma)$ for $\gamma \in
\Gamma_{p(e)}$, where $\Gamma_{p(e)}$ is the isotropy group of $p(e) \in B$. One can
consider such bundles where the family of local representations $\calr$ is prescribed,
e.g. one may demand that $\Gamma_{p(e)}$ is always compact and allow only certain
homomorphisms $\rho_e$.

Our main technical result is Theorem~\ref{the:bundles_and_Gamma_times_G-spaces}, where we
prove that a $\Gamma$-equivariant principal $G$-bundle $p \colon E \to B$ is the same as a
$\Gamma \times G$-$CW$-complex $E$, provided that the family of locally representations
satisfies \emph{Condition (H)} introduced in Definition~\ref{def:property_(H)}.  This implies the
main result of this paper
Theorem~\ref{the:Classifying_space_for_gamma-equivariant_principal_G-bundles} that gives a
\emph{universal $\Gamma$-equivariant principal $G$-bundle with respect to a given family
  of local representations $\calr$}, provided that $\calr$ satisfies Condition (H).

Condition (H) is needed to ensure homotopy invariance for $\Gamma$-equivariant principal
$G$-bundles. It is automatically satisfied, if $\Gamma$ and $G$ are locally compact second
countable groups with finite covering dimensions (e.g. Lie groups),  $G$ is almost
connected (i.e. $G$ modulo its connected component of the identity is compact), and all
base spaces are $\Gamma$-$CW$-complexes with compact isotropy groups,
see Theorem~\ref{the:criterion_for_(H)}.

Equivariant principal bundles have been studied before by several authors (see
\cite{Hambleton-Hausmann(2003), Lashof(1982), Lashof-May(1986), Lashof-May-Segal(1983),
  May(1990), Maruyama-Shimakawa(1995), Dieck(1987)} and references therein) and our
construction generalizes all the previous constructions in the following sense.  We
isolated the conditions that the groups $\Gamma$ and $G$ and the family of local
representations need to satisfy, in order to show the existence of a universal equivariant
bundle associated to the prescribed family of local representation  as base; 
this is the Condition (H) mentioned above.  Moreover, in the literature some conditions about the restriction of the
bundles to $\Gamma$-invariant neighborhoods of points in the base space are demanded; we
show that they automatically follow from our setting in
Theorem~\ref{the:local_objects} if the Condition (H) is satisfied. 

In Section~\ref{Example:The_case_G_is_PU(H)_the_projective_unitary_group} 
we have included
a study of the case $G=\calp\calu(\calh)$, the projective unitary group, endowed with the
norm topology. The main result in
Section~\ref{Example:The_case_G_is_PU(H)_the_projective_unitary_group} is
Theorem~\ref{the:classification_of_equivariant_stable_projective_unitary_bundles} which
produces a universal $\Gamma$-equivariant stable projective unitary bundle for almost free
$\Gamma$-CW-complexes.  This result
generalizes~\cite[Theorem~3.21]{Barcenas-Espinoza-Joachim-Uribe(2012)}, where $\Gamma$ is
assumed to be discrete. This universal bundle is relevant for equivariant twisted
topological $K$-theory since with this bundle it can be defined as a parameterized
equivariant cohomology theory.

Our results carry directly over to the case, where one allows an intertwining between the
$\Gamma$ and the $G$-action, i.e. there exists a group homomorphism 
$\tau \colon \Gamma \to \aut(G)$ 
and the condition that the $\Gamma$ and $G$-action on the total
space commute, is replaced by the weaker condition 
$\gamma \cdot (e \cdot g) = (\gamma \cdot e) \cdot \tau(\gamma)(g)$. 
A typical example for such a non-trivial intertwining is the case $\Gamma = \IZ/2$, $G
= U(n)$ and $\tau \colon \IZ/2 \to \aut(U(n))$ given by complex conjugation, which leads to
real vector bundles in the sense of Atiyah~\cite{Atiyah(1966)}.  For the simplicity of the
exposition we only treat the case where $\tau$ is trivial.

Throughout this paper we will work in the category of compactly generated spaces,
see~\cite{Steenrod(1967)} and Section~\ref{sec:Appendix_A:Compactly_generated_spaces}, and 
subgroups are always understood to be closed subgroups. Most of the 
equivariant $CW$-complexes under consideration are proper, or, equivalently, have compact isotropy groups.

This paper has been financially supported by the Leibniz-Award, granted by the Deutsche
Forschungsgemeinschaft, to the first author, and by the Alexander Von Humboldt Foundation
through a scholarship for Experienced Researchers granted to the second author. Both
authors would like to acknowledge and thank the financial support of the Deutsche
Forschungsgemeinschaft and of the Alexander Von Humboldt Foundation. Moreover, 
the authors thank the referee
for his careful reading of the first version and his useful comments.

The paper is organized as follows:

\tableofcontents


\section{Principal bundles}
\label{sec:Principle_bundles}

We recall some basic facts about principal $G$-bundles over $CW$-complexes for a
topological group $G$.

\begin{definition}[Quasi-regular open set and regular space]
  \label{def:quasi-regular}
  An open subset $U \subseteq B$ is called \emph{quasi-regular} if for any $x \in U$ there
  exists an open neighborhood $V_x$ whose closure in $B$ is contained in $U$. An space is called 
  \emph{regular} if it separates points from closed subsets.
\end{definition}

The main point of this notion is that a quasi-regular open subset equipped 
with its subspace topology is again compactly generated, see 
Lemma~\ref{lem:properties_of_quasi-regular_open_sets}~\ref{lem:properties_of_quasi-regular_open_sets:open_subset},
and preimages of quasi-regular open subsets are again quasi-regular open subsets, see 
Lemma~\ref{lem:properties_of_quasi-regular_open_sets}~\ref{lem:properties_of_quasi-regular_open_sets:preimage}.

\begin{definition}[Principal $G$-bundle]
  \label{def:principal_G-bundle}
  A \emph{principal $G$-bundle} $p \colon E \to B$ consists of a space $E$ with right
  $G$-action, a space $B$ with trivial $G$-action and a $G$-map $p$ such that $p$ is
  locally trivial, i.e., for any $b \in B$ there exists a quasi-regular open neighborhood
  $U$ of $b$ in $B$ and a $G$-homeomorphism $\phi \colon G \times U \to p^{-1}(U)$
  satisfying $p \circ \phi = \pr$ for the projection $\pr \colon U \times G \to U$.
\end{definition}
 
In the previous definition we have added the condition that the local trivialization can be done on a quasi-regular open set. 
On the standard definition of principal bundles the quasi-regularity is not required. Nevertheless, since in this article
we will work in the category of equivariant $CW$-complexes, and any invariant open subset of a equivariant $CW$-complex
is automatically quasi-regular, see Lemma \ref{lem:properties_of_quasi-regular_open_sets:Gamma-CW}, this extra condition in the definition is innocuous.

\begin{lemma} \label{lem:homotopy_invariance_for_trivial_Gamma} Let $B$ be a $CW$-complex
  and let $p \colon E \to B \times [0,1]$ be a principal $G$-bundle. Let $i_0 \colon B = B
  \times \{0\} \to B \times [0,1]$ be the inclusion.

  Then $i_0^*E \times [0,1] \xrightarrow{i_0^*p \times \id_{[0,1]}} B \times [0,1]$ is a
  principal $G$-bundle and there exists an isomorphism of principal $G$-bundles
  \[
  f \colon i_0^*E \times [0,1] \to E
  \]
  over $B \times [0,1]$ whose restriction to $B \times \{0\}$ is the identity.
\end{lemma}
\begin{proof}
  A $CW$-complex $B$ is paracompact by~\cite{Miyazaki(1952)}.  Now the proof is analogous
  to the one of~\cite[Theorem~9.8 in Chapter~3 on page~51]{Husemoeller(1966)} taking into
  account that any open subset of $B$ is quasi-regular by 
  Lemma~\ref{lem:properties_of_quasi-regular_open_sets}~\ref{lem:properties_of_quasi-regular_open_sets:regular}
  and~\ref{lem:properties_of_quasi-regular_open_sets:equivalence:locally_compact_or_metrizable} and that
  in~\cite[Theorem~9.8 in Chapter~3 on page~51]{Husemoeller(1966)} the symbol $\times$
  stands for the classical product space, where in our setting $\times$ stands for the
  product within the category of compactly generated spaces.
\end{proof}

  \begin{notation}\label{not:X_l:_and_Y_r}
    Given a $\Gamma \times G$-space $X$, let $X_r$ be the space $X$ but now equipped with
    the left $\Gamma$ action given by $\gamma \cdot x = (\gamma,1) \cdot x$ and the right
    $G$-action given by $x \cdot g := (1,g^{-1}) \cdot x$.

    Given a space $Y$ with (commuting) left $\Gamma$- and right $G$-action, let $Y_l$ be
    the same space but now equipped with the left $\Gamma \times G$-action given by
    $(\gamma,g) \cdot y := \gamma \cdot y \cdot g^{-1}$.
  \end{notation}

  \begin{example}[Free $G$-$CW$-complexes]
    \label{free_G-CW-complex}
    Let $X$ be a free $G$-$CW$-complex. Then $p \colon X_r \to X_r/G$ is a principal
    $G$-bundle.  Conversely, if $p \colon E \to B$ is a principal $G$-bundle over a
    $CW$-complex $B$, then $E_l$ carries the structure of a free $G$-$CW$-complex coming
    from the filtration given by the preimages of the skeletons of $B$.

    These claims are proved in~\cite[1.24 and 1.25 on page~18]{Lueck(1989)}, and they will
    also follow from Theorem~\ref{the:bundles_and_Gamma_times_G-spaces} applied to the
    special case $\Gamma = \{1\}$.
  \end{example}

  \begin{example}[Free proper smooth $G$-actions on smooth manifolds]
    \label{exa:Free_proper_smooth_G-actions_on_smooth_manifolds} Consider a Lie group $G$
    with a free proper smooth left $G$-action on a smooth manifold $M$. Then $M$ is a
    proper $G$-$CW$-complex by~\cite{Illman(2000)} and the projection $M_r \to M_r/G$ is a principal $G$-bundle.
  \end{example}


  \section{Equivariant principal bundles}
  \label{sec:Equivariant_principle_bundles}

  We first fix the notions for the objects we want to study.

  \begin{definition}[$\Gamma$-equivariant principle $G$-bundle]
    \label{def:Gamma-equivariant_principle_G-bundle}
    A \emph{$\Gamma$-equivariant principal $G$- bundle} $p \colon E \to B$ consists of a
    principal $G$-bundle together with left $\Gamma$-actions on $E$ and $B$ (commuting
    with the right $G$-actions) such that $p \colon E \to B$ is $\Gamma$-equivariant.
  \end{definition}

  Note that since the $\Gamma$ and the $G$ actions commute, then $\Gamma$ acts
  on  $p \colon E \to B$ through $G$-bundle maps, cf. \cite[\S 1]{Lashof(1982)}. See also~\cite[Chapter~I, Section~8]{Dieck(1987)} for this notion and its main
  properties including universal objects for a compact Lie group $\Gamma$ and a
  topological group $G$, where also a twisting of the left $\Gamma$- and right $G$-actions
  by a homomorphism $\Gamma \to \aut(G)$ is allowed. Sometimes in the literature some
  conditions about the restriction of the bundles to $\Gamma$-invariant neighborhoods of
  points in the base space are demanded or can only be proved in the case that $G$ is a compact Lie group, see for 
  instance~\cite[Proposition~8.10 on page~58]{Dieck(1987)}. We will show that they automatically follow from
  our setting in Theorem~\ref{the:local_objects}.

  We mention some basic properties of $\Gamma$-equivariant principal $G$-bundles.

  If $f \colon X \to B$ is a $\Gamma$-map of $\Gamma$-$CW$-complexes and $p \colon E \to
  B$ is a $\Gamma$-invariant principal $G$-bundle, then the map $f^*p \colon f^*E \to X$
  obtained from the pullback of $p$ with $f$
  \[
  \xymatrix{ f^*E \ar[r]^{\overline{f}} \ar[d]^{f^*p} & E \ar[d]^p
    \\
    X \ar[r]^{f} & B }
  \]
  is a $\Gamma$-equivariant principal $G$-bundle again.  An \emph{isomorphism of two
    $\Gamma$-invariant principal $G$-bundles} $p_0 \colon E_0 \to X$ and $p_1 \colon E_1
  \to B$ over the $\Gamma$-$CW$-complex $B$ is a homeomorphism $f \colon E_0 \to E_1$
  which is compatible with both the left $\Gamma$-action and the right $G$-action and
  satisfies $p_1 \circ f = p_0$.

  \begin{lemma} \label{lem_maps_are_isos} Let $p_0 \colon E_0 \to B$ and $p_1 \colon E_1
    \to B$ be $\Gamma$-equivariant principle $G$-bundles over the $\Gamma$-$CW$-complex
    $B$. Let $f \colon E_0 \to E_1 $ be a map which is compatible with both the left
    $\Gamma$-actions and the right $G$-actions and satisfies $p_1 \circ f = p_0$.

    Then $f$ is an isomorphism of $\Gamma$-equivariant principal $G$-bundles.
  \end{lemma}
  \begin{proof}
    The map $f$ is a homeomorphism because of the local triviality of the principal
    $G$-bundles $p_0$ and $p_1$.
  \end{proof}


  \section{Families of local representations}
  \label{sec:Families_of_local_representations}

  In this section we introduce the local representations coming from the left action of
  the $\Gamma$-isotropy group of a point $b$ in the base space and the free right
  $G$-action on the fiber over $b$ (after choosing a lift of $b$ to the total space).

  \begin{definition}[Local representations]
    \label{def:local_representations}
    Let $p \colon E \to B$ be a $\Gamma$-equivariant principal $G$-bundle.  Consider $e
    \in E$. Then we obtain a (continuous) group homomorphism
    \begin{eqnarray}
      \rho_e \colon \Gamma_{p(e)} \to G
     \label{rho_e}
    \end{eqnarray}
    uniquely determined by $\gamma \cdot e = e \cdot \rho_e(\gamma)$ for $\gamma \in
    \Gamma_{p(e)}$, where $\Gamma_{p(e)}$ is the isotropy group of $p(e) \in B$.
  \end{definition}

  $\rho_e$ is indeed a homomorphism by the following calculation
  \begin{eqnarray*}
    e \cdot \rho_e(\gamma_1 \cdot \gamma_2)
    & = & 
    (\gamma_1 \cdot \gamma_2) \cdot e
    \\
    & = &
    \gamma_1 \cdot (\gamma_2 \cdot e)
    \\
    & = &
    \gamma_1 \cdot \bigl(e \cdot \rho_e(\gamma_2)\bigr)
    \\
    & = &
    \bigl(\gamma_1 \cdot e\bigr)  \cdot \rho_e(\gamma_2)
    \\
    & = &
    \bigl(e \cdot \rho_e(\gamma_1)\bigr) \cdot \rho_e(\gamma_2)
    \\
    & = &
    e \cdot \bigl(\rho_e(\gamma_1) \cdot \rho_e(\gamma_2)\bigr).
  \end{eqnarray*}
  It is continuous since the map $G \to p^{-1}(e),\; g \mapsto e \cdot g$ is a
  homeomorphism because of the local triviality of the principal $G$-bundle $p$
  and the map $\Gamma_{p(e)} \to p^{-1}(e), \;\gamma \mapsto \gamma \cdot e$ is continuous.

  \begin{remark}[Basic properties of the local representations]
    \label{rem:local_representations}
    If we replace $e$ by $eg$ for some $g \in G$, then $\rho_{eg} = c_{g^{-1}} \circ \rho_e$ for
    $c_g \colon G \to G$ the conjugation homomorphism sending $g'$ to $gg'g^{-1}$.

    If we replace $e$ by $\gamma e$ for some $\gamma \in \Gamma$, then 
    $\Gamma_{p(e)} = \gamma^{-1} \Gamma_{p(\gamma e)} \gamma$ and $\rho_{\gamma e} = \rho_e \circ
    c_{\gamma^{-1}}$.

    If
    \[\xymatrix{E_0 \ar[r]^{\overline{f}} \ar[d]^{p_0} & E_1 \ar[d]^{p_1}
      \\
      B_0 \ar[r]^{f} & B_1 }
    \]
    is a morphism of $\Gamma$-equivariant principal $G$-bundles, then 
    \[\rho^{p_0}_{e} = \rho^{p_1}_{\overline{f}(e)} \circ i_e
    \] 
    holds for all $e \in E_0$, where $i_e \colon \Gamma_{p_0(e)} \to \Gamma_{p_1\circ \overline{f}(e)}$ is the inclusion.
  \end{remark}

  \begin{definition}[Family of local representations]
    \label{def:family_of_local_representations}
    A \emph{family $\calr$ of local representations for $(\Gamma,G)$} is a set of pairs
    $(H,\alpha)$, where $H$ is a subgroup of $\Gamma$ and $\alpha \colon H \to G$ is a
   continuous group homomorphism such that the following conditions are satisfied:

    \begin{itemize}
    \item \emph{Finite intersections}\\
      Suppose that $(H_0,\alpha_0)$ and $(H_1,\alpha_1)$ belong to $\calr$. Define $H :=
      \{h \in H_0 \cap H_1 \mid \alpha_0(h) = \alpha_1(h)\}$ and $\alpha \colon H \to G$
      by $\alpha = \alpha_0|_H = \alpha_1|_H$. Then $(H,\alpha) \in \calr$;

    \item \emph{Conjugation in G}\\
      If $(H,\alpha)$ belongs to $\calr$ and $g \in G$, then $(H,c_{g^{-1}} \circ \alpha)$
      belongs to $\calr$;

    \item \emph{Conjugation in $\Gamma$}\\
      If $(H,\alpha)$ belongs to $\calr$ and $\gamma \in \Gamma$, then $(\gamma
      H\gamma^{-1},\alpha \circ c_{\gamma^{-1}})$ belongs to $\calr$.

    \end{itemize}
  \end{definition}

  \begin{definition}[(Pre)family of local representations associated to a
    $\Gamma$-equivariant principal $G$-bundle]
    \label{exa_calr_of_an-equivariant_principal_bundle}
    Let $p \colon E \to B$ be a $\Gamma$-equivariant principal $G$-bundle. Define the
    \emph{prefamily of local representations of $p$} to be
    \[
    \calr'(p) := \{(\Gamma_{p(e)},\rho_e) \mid e \in E\}.
    \]
    Let $\calr(p)$ be the smallest system of local representations containing $\calr'(p)$.
    We call $\calr(p)$ the \emph{family of local representations associated to $p$}.
  \end{definition}

  One easily checks using Remark~\ref{rem:local_representations} that $\calr'(p)$ is
  closed under conjugation, but not necessary under finite intersections so that
  $\calr'(p)$ itself is not a family of local representations.

  We want to deal with families of local representations to ensure that the following
  lemma is true.

  A \emph{family of subgroups of $G$} is a set of subgroups of $G$ closed under
  conjugation and taking finite intersections.

  \begin{lemma} \label{lem:family_associated_to_calr} 
   Let $\calr$ be a family of local
    representations for $\Gamma$ and $G$.  For $(H,\alpha)$ in $\calr$ let $K(H,\alpha)$
    be the subgroup of $\Gamma \times G$ given by
    \[
    K(H,\alpha) := \{(\gamma,\alpha(\gamma)) \mid \gamma \in H\}.
    \]
    Put
    \[
    \calf(\calr) = \{K(H,\alpha) \mid (H,\alpha) \in \calr\}.
    \]
    Then $\calf(\calr)$ is a family of subgroups of $\Gamma \times G$.
  \end{lemma}
  \begin{proof}
    We have to check that $\calf(\calr)$ is closed under conjugation and finite
    intersections.

    Consider $K \in \calf(\calr)$ and $(\gamma,g) \in \Gamma \times G$.  Choose
    $(H,\alpha) \in \calr$ with $K = K(H,\alpha)$. Then
    \begin{eqnarray*}
      (\gamma,g)^{-1}  \cdot K \cdot (\gamma,g) 
      & = & 
      (\gamma,g)^{-1}  \cdot K(H,\alpha) \cdot (\gamma,g) 
      \\
      & = & 
      \{(\gamma,g)^{-1} \cdot  (h,\alpha(h)) \cdot  (\gamma,g) \mid h \in H\}
      \\
      & = & 
      \bigl\{(\gamma^{-1}h\gamma,g^{-1}\alpha(h)g) \mid h \in H\bigr\}
      \\
      & = & 
      \bigl\{(c_{\gamma}(h), c_{g^{-1}} \circ \alpha \circ c_{\gamma}(c_{\gamma^{-1}}(h)) \mid h \in H\bigr\}
      \\
      & = & 
      \bigl\{h',  c_{g^{-1}} \circ \alpha \circ c_{\gamma}(h') \mid h' \in \gamma^{-1} H\gamma\bigr\}
      \\
      & = & 
      K( \gamma^{-1} H\gamma,c_{g^{-1}} \circ \alpha \circ c_{\gamma}).
    \end{eqnarray*}
    Since $(\gamma^{-1} H\gamma,c_{g^{-1}} \circ \alpha \circ c_{\gamma})$ belongs to
    $\calr$, we conclude that $(\gamma,g)^{-1} \cdot K \cdot (\gamma,g) $ belongs to
    $\calf(\calr)$.

    Consider $K_0,K_1 \in \calf(\calr)$. Choose $(H_i,\alpha_i)$ in $\calr$ with $K_i =
    K(H_i,\alpha_i)$ for $i = 0,1$. Define $H := \{h \in H_0 \cap H_1 \mid \alpha_0(h) =
    \alpha_1(h)\}$ and $\alpha \colon H \to G$ by $\alpha = \alpha_0|_H =
    \alpha_1|_H$. Then $K_0 \cap K_1 = K(H,\alpha)$ and $(H,\alpha) \in \calr$. This
    implies $K_0 \cap K_1 \in \calf(\calr)$.
  \end{proof}

  \begin{remark}[Families] \label{rem:families} We later will consider the classifying
    space $\EGF{\Gamma} {\calf}$ of a family of subgroups of $\Gamma$.  It can be defined
    without the condition that $\calf$ is closed under finite intersections, being closed
    under conjugation is enough, but this extra condition is usually required to ensure
    that for two $\Gamma$-spaces $X$ and $Y$ whose isotropy groups belong to $\calf$ also
    the isotropy groups of $X \times Y$ with the diagonal $\Gamma$-action belong to
    $\calf$.

    It is actually more convenient to require instead of the condition that $\calf$ is
    closed under finite intersections that it is closed under subgroups. However, this
    rules out one important case, namely, the case of the family of open compact
    subgroups, which naturally occurs in the context of locally compact second countable
    totally disconnected groups $\Gamma$.

    If $\calr$ is closed under subgroups, i.e., for $(H,\alpha) \in \calr$ and any
    subgroup $K \subseteq H$ we have $(K,\alpha|_K) \in \calr$, then $\calf(\calr)$ is
    closed under taking subgroups.
  \end{remark}

\begin{remark}[Local representations and pullbacks]\label{rem:calr_and_pullbacks}
  Let $p \colon E \to B$ be a $\Gamma$-equivariant principal $G$-bundle and let $f \colon
  A \to B$ be a $\Gamma$-map.  Let $\calr$ be a family of local representations. Suppose
  that we have $\calr(p) \subseteq \calr$.  Then we get $\calr(f^*p) \subseteq \calr$ for
  the pullback $f^*p$, provided for any $(H,\alpha) \in \calr$ and any subgroup $K
  \subseteq H$ which occurs as isotropy group in $A$, we have $(K,\alpha|_K) \in \calr$.
  This follows from Remark~\ref{rem:local_representations}.

  If we make the assumption that $\calr$ is closed under subgroups, then $f^*p$
  automatically satisfies $\calr(f^*p) \subseteq \calr$ if $\calr(p) \subseteq \calr$
  holds.
\end{remark}


\section{The Condition (S)}
\label{sec:The_condition_(S)}

\begin{definition}[Condition (S)] \label{ref:condition_(S)}
  Given a topological group $\Gamma$ and a (closed) subgroup $H \subseteq \Gamma$, we say
  that the pair $(\Gamma,H)$ \emph{satisfies Condition (S)} if the projection $\pr\colon
  \Gamma \to \Gamma/H$ has a local cross section, i.e., there is a quasi-regular open 
  neighborhood $U$ of $1H \in \Gamma/H$ together with a map $\sigma \colon U \to \Gamma$ 
  such that $\pr \circ \sigma = \id_U$.

  A topological group $\Gamma$ \emph{satisfies Condition (S)} if for any subgroup $H
  \subseteq \Gamma$ the pair $(\Gamma,H)$ satisfies Condition (S).
\end{definition}

The role of the Condition (S) is to ensure the following lemma.

\begin{lemma} \label{lem:condition_(S)_and_homeomorphisms} 
  Let $f \colon E \to \Gamma/H$   be a $\Gamma$-map for some subgroup $H \subseteq \Gamma$.  
  Suppose that the pair  $(\Gamma,H)$ satisfies Condition (S).

  Then the $\Gamma$-map
  \[
  u \colon \Gamma \times_H f^{-1}(1H) \to E, \quad (\gamma,e) \mapsto \gamma \cdot e
  \]
  is a homeomorphism.
\end{lemma}
\begin{proof}
  The map $u$ is clearly a bijective map of sets. Now, the Condition (S) ensures that there is an open neighborhood $U \subseteq \Gamma/H$ of
  $1H$ and a map $s \colon U \to \Gamma$ whose composite with the projection $\pr \colon
  \Gamma \to \Gamma/H$ is the identity on $U$.  Moreover, the open subsets
   $U \subseteq \Gamma/H$,  $\pr^{-1}(U) \subseteq \Gamma$  and $f^{-1}(\pr^{-1}(U)) \subseteq E$
  equipped   with the subspace topologies are compactly generated, see 
Lemma~\ref{lem:properties_of_quasi-regular_open_sets}~\ref{lem:properties_of_quasi-regular_open_sets:open_subset} 
and~\ref{lem:properties_of_quasi-regular_open_sets:preimage}.
 Define a map
  \[
  v \colon f^{-1}(\pr^{-1}(U)) \to \pr^{-1}(U) \times_H f^{-1}(1H), \quad e \mapsto (s
  \circ \pi(e),s \circ \pi(e)^{-1} \cdot e),
  \]
  where $\pi \colon f^{-1}(\pr^{-1}(U)) \to U$ is the map induced by $\pr \circ f$.  Let
  \[
  u|_{ \pr^{-1}(U) \times_H f^{-1}(1H)} \colon \pr^{-1}(U) \times_H f^{-1}(1H) \to
  f^{-1}(\pr^{-1}(U))
  \]
  be obtained by restricting $u$. Then $u|_{ \pr^{-1}(U) \times_H f^{-1}(1H)} \circ v =
  \id_{f^{-1}(\pr^{-1}(U))}$ and $v \circ u|_{ \pr^{-1}(U) \times_H f^{-1}(1H)} =
  \id_{\pr^{-1}(U) \times_H f^{-1}(1H)}$.  Hence $u|_{ \pr^{-1}(U) \times_H f^{-1}(1H)}$
  is a homeomorphism. Since $u$ is $\Gamma$-equivariant,
  $f^{-1}(\pr^{-1}(U)) \subseteq E$ and $\pr^{-1}(U) \times_H f^{-1}(1H) \subseteq \Gamma \times_H f^{-1}(1H)$ are 
  open subsets,  $\{\gamma \cdot U \mid \gamma
  \in \Gamma\}$ is an open covering of $\Gamma/H$ and $u$ is bijective, then the map $u$ is a homeomorphism.
\end{proof}

\begin{remark}[Condition (S) and principal bundle structure] \label{rem:(S)_and_principal}
  If the projection $\pr \colon \Gamma \to \Gamma/H$ is a principal $H$-bundle, then the
  local triviality implies that the pair $(\Gamma,H)$ satisfies Condition (S). The
  converse is also true, namely, apply
  Lemma~\ref{lem:equivariant_principal_G-bundles_over_Gamma/H_times_Z}~%
\ref{lem:equivariant_principal_G-bundles_over_Gamma/H_times_Z:well-defined} in the
  special case, where the role of $\Gamma$, $H$ and $G$ is played by $\Gamma$, $H$, and
  $H,$ and $Z = \pt$, and use the canonical $\Gamma$-homeomorphism $\Gamma \times_H H
  \xrightarrow{\cong} \Gamma, (\gamma,h) \mapsto \gamma \cdot h$.
\end{remark}

The Condition (S) is satisfied in many cases.

\begin{lemma} \label{lem:condition_(S)}
  \begin{enumerate}

  \item \label{lem:condition_(S):H_compact_Lie} Suppose that $\Gamma$ is completely
    regular, i.e., for any $x \in \Gamma$ and any neighborhood $U$ of $x$ in $\Gamma$,
    there exists a continuous function $f \colon \Gamma \to [0,1]$ with $f(x) = 0$ and
    $f(X\setminus U) = 1$.  Then for any subgroup $H \subseteq \Gamma$ which is a compact
    Lie group, the pair $(\Gamma,H)$ satisfies Condition (S);

  \item \label{lem:condition_(S):Gamma} A topological group $\Gamma$ satisfies Condition
    (S) if $\Gamma$ is discrete, if $\Gamma$ is a Lie group, or more generally, if
    $\Gamma$ is locally compact and second countable and has finite covering dimension.
  \end{enumerate}
\end{lemma}
\begin{proof}~\ref{lem:condition_(S):H_compact_Lie} This follows from~\cite{Palais(1961)},
  see also~\cite[Theorem~1.38 on page~27]{Lueck(1989)}. Notice that the conditions that
  $\Gamma$ is completely regular (and hence regular) and $H$ is compact imply that $\Gamma/H$ is regular.
  Hence every open subset of $\Gamma/H$ is quasi-regular, see
  Lemma~\ref{lem:properties_of_quasi-regular_open_sets}~\ref{lem:properties_of_quasi-regular_open_sets:regular}.
  \\[2mm]~\ref{lem:condition_(S):Gamma} This follows from~\cite{Mostert(1953)}. The metric
  needed in~\cite{Mostert(1953)} follows under our assumptions from
  Theorem~\ref{the:Birkhoff-Kakutani}.
 Notice that the condition that
  $\Gamma$ is locally compact implies that $\Gamma/H$ is locally compact and hence
  that any open subset of $\Gamma/H$ is quasi-regular, see
  Lemma~\ref{lem:properties_of_quasi-regular_open_sets}~\ref{lem:properties_of_quasi-regular_open_sets:open_subset}~%
\ref{lem:properties_of_quasi-regular_open_sets:regular},
  and~\ref{lem:properties_of_quasi-regular_open_sets:equivalence:locally_compact_or_metrizable}.  
\end{proof}

\begin{example}[Kac-Moody groups]\label{exa:Kac-Moody_groups}
Kac-Moody groups are not Lie groups but all their compact subgroups are Lie groups, see~\cite[Theorem~2.4]{Kitchloo(2009)}. 
Since they are completely regular, Lemma~\ref{lem:condition_(S)}~\ref{lem:condition_(S):H_compact_Lie} applies to them
and hence they satisfy Condition (S).
\end{example}


\section{Equivariant principal bundles over equivariant cells}
\label{sec:Equivariant_principal_bundles_over_equivariant_cells}

In this section we want to analyze $\Gamma$-equivariant principal $G$-bundles over spaces
of the type $\Gamma/H \times Z$ for some subgroup $H \subset \Gamma$ and space $Z$ with
trivial $H$-action.  Later we will be mainly interested in the case $\Gamma/H \times D^n$.

Let $H$ and $G$ be two  topological groups. Equip $\hom(H,G)$
with the subspace topology with respect to the inclusion $\hom(H,G) \subseteq \map(H,G)$
see Subsection~\ref{subsec:Space_of_homomorphisms}.

Consider a space $Z$ and a map $\sigma \colon Z
\to \hom(H,G)$.  We have the obvious  right $H$-action on $\Gamma$ and the left $H$-action on
$Z \times G$ given by $h \cdot (z,g) := \bigl(z, \sigma(z)(h) \cdot g\bigr)$.  Let
\[
p_{\sigma} \colon \Gamma \times_H (Z \times G) \to \Gamma/H \times Z
\]
be the map induced by the projection $Z \times G \to Z$. It is compatible with the left
$\Gamma$-action on $\Gamma \times_H (Z \times G)$ given by $\gamma_0 \cdot
\bigl(\gamma,(z,g)\bigr) = \bigl(\gamma_0 \gamma,(z,g)\bigr)$ and the left $\Gamma$-action
on $\Gamma/H \times Z$ given by 
$\gamma_0 \cdot (\gamma \cdot H ,z) = (\gamma_0\gamma \cdot H,z)$.  
It is also compatible with the right $G$-action on $\Gamma \times_H (Z
\times G)$ given by $\bigl(\gamma,(z,g)\bigr) \cdot g_0 = \bigl(\gamma,(z,gg_0)\bigr)$ and
the trivial right $G$-action on $\Gamma \times_H Z$.  The left $\Gamma$-action and the
right $G$-action on $\Gamma \times_H (Z \times G)$ commute.

\begin{lemma} \label{lem:equivariant_principal_G-bundles_over_Gamma/H_times_Z} Suppose
  that the pair $(\Gamma,H)$ satisfies Condition (S), see
  Definition~\ref{ref:condition_(S)}, and let $Z$ be a space. Then
  \begin{enumerate}

  \item \label{lem:equivariant_principal_G-bundles_over_Gamma/H_times_Z:well-defined} The
    map $p_{\sigma} \colon \Gamma \times_H (Z \times G) \to \Gamma/H \times Z$ is a
    $\Gamma$-equivariant principal $G$-bundle;

  \item \label{lem:equivariant_principal_G-bundles_over_Gamma/H_times_Z:given_by_rho} A
    $\Gamma$-equivariant principal $G$-bundle $E \to \Gamma/H \times Z$ is isomorphic as
    $\Gamma$-equivariant principal $G$-bundle to $p_{\sigma}$ for an appropriate map
    $\sigma \colon Z \to \hom(H,G)$, provided that the restriction of $p$ to $\{1H\} \times Z$
    is (after forgetting the $H$-action) a trivial principal $G$-bundle;

  \item \label{lem:equivariant_principal_G-bundles_over_Gamma/H_times_Z:uniqueness} Given
    two maps $\sigma_0 \colon Z \to \hom(H,G)$ and $\sigma_1 \colon Z \to \hom(H,G)$, the
    $\Gamma$-equivariant principal $G$-bundles $p_{\sigma_0}$ and $p_{\sigma_1}$ are
    isomorphic, if and only if there is a map $\omega \colon Z \to G$ such that
    \[
    \sigma_1(z)(g) = \omega(z) \cdot \sigma_0(z)(h) \cdot \omega(z)^{-1}
    \]
    holds for all $h\in H$ and $z \in Z$;

  \item \label{lem:equivariant_principal_G-bundles_over_Gamma/H_times_Z:local_system}
    Given a map $\sigma \colon Z \to \hom(H,G)$, the homomorphism $\rho_{(\gamma,(z,g))}
    \colon \Gamma_{(\gamma H,z)} \to G$ associated to $p_{\sigma}$ in~\eqref{rho_e} for
    $\bigl(\gamma,(z,g)\bigr) \in \Gamma \times_H (Z \times G)$ is given by
    \begin{eqnarray*}
      \Gamma_{(\gamma H,z)} 
      & = & 
      \gamma H \gamma^{-1};
      \\ 
      \rho_{(\gamma,(z,g))}(\gamma \cdot h \cdot \gamma^{-1}) 
      & = & 
      g^{-1}\sigma(z)(h) \cdot g.
    \end{eqnarray*}
  \end{enumerate}
\end{lemma}

\begin{proof}~\ref{lem:equivariant_principal_G-bundles_over_Gamma/H_times_Z:well-defined}
  It remains to show that $p_{\sigma} \colon \Gamma \times_H (Z \times G) \to \Gamma/H
  \times Z$ is a principal $G$-bundle after forgetting the $\Gamma$-action.  The Condition
  (S) ensures that there is a quasi-regular open neighborhood $U \subseteq \Gamma/H$ of $1H$ and a map
  $s \colon U \to \Gamma$ whose composite with the projection 
  $\pr \colon \Gamma \to \Gamma/H$ is the identity.   Notice that the open subsets $U \subseteq \Gamma/H$ and
  $\pr^{-1}(U) \subseteq \Gamma$ equipped with the subspace topology are compactly generated by
  Lemma~\ref{lem:properties_of_quasi-regular_open_sets}~\ref{lem:properties_of_quasi-regular_open_sets:open_subset} 
  and~\ref{lem:properties_of_quasi-regular_open_sets:preimage}.

  Define $\overline{s} \colon \pr^{-1}(U) \to H$ by
  $\overline{s}(\gamma) = s \circ \pr(\gamma)^{-1} \cdot\gamma$.  It has the property
  $\overline{s}(\gamma \cdot h) = \overline{s}(\gamma) \cdot h$ for all $\gamma \in   p^{-1}(U)$ and $h \in H$.
  Define maps
  \begin{eqnarray*}
    \alpha \colon (U \times Z) \times G 
    & \to &
    \pr^{-1}(U) \times_H(Z \times G);
    \\
    \beta \colon \pr^{-1}(U) \times_H(Z \times G);
    & \to &
    (U \times  Z) \times G, 
  \end{eqnarray*}
  by
  \begin{eqnarray*}
    \alpha \bigl((\gamma H,z),g\bigr)  
    & := &
    \bigl(\gamma, (z, \sigma(z)(\overline{s}(\gamma))^{-1} \cdot g)\bigr);
    \\
    \beta \bigl(\gamma,(z,g)\bigr)
    & := &
    \bigl((\gamma H,z), \sigma(z)(\overline{s}(\gamma)) \cdot g\bigr).
  \end{eqnarray*}
  Then $\alpha$ and $\beta$ are to one another inverse $G$-homeomorphisms. They induce
  isomorphisms of principal $G$-bundles from $p_{\sigma}$ restricted to $U  \times Z$ 
  to the trivial principal $G$-bundle over $U \times Z$.   One easily checks using
  Lemma~\ref{lem:properties_of_quasi-regular_open_sets}~\ref{lem:properties_of_quasi-regular_open_sets:preimage} 
  applied to the standard map $\Gamma/H \times Z  = k(\Gamma/H \times_p Z) \to \Gamma/H \times_p Z$, see
  Subsections~\ref{subsec:The_retraction_functor} and~\ref{subsec:mapping_spaces_product_spaces_and_subspaces},
  and Lemma~\ref{lem:properties_of_quasi-regular_open_sets}~\ref{lem:properties_of_quasi-regular_open_sets:open_subset} 
  that $U  \times  Z$ is a quasi-regular open subset of $\Gamma/H \times Z$. Since
  $\gamma H \in \Gamma/H$ is contained in the open subset $\gamma \cdot U$ and
  $p_{\sigma}$ is a $\Gamma$-equivariant map, we conclude that $p_{\sigma}$ is locally
  trivial and hence a principal $G$-bundle.
  \\[2mm]~\ref{lem:equivariant_principal_G-bundles_over_Gamma/H_times_Z:given_by_rho} In
  the sequel we identify the subspace $\{1\} \times Z$ of $\Gamma/H \times Z$ with $Z$.
  Define the $\Gamma$-map
  \[
  u \colon \Gamma \times _H p^{-1}(Z) \to E, \quad (\gamma,e) \to \gamma \cdot e.
  \]
  It is a homeomorphism by Lemma~\ref{lem:condition_(S)_and_homeomorphisms}.  It is
  compatible with the natural right $G$-actions and commutes with the left $\Gamma$-actions. The restriction $p|_{p^{-1}(Z)} \colon
  p^{-1}(Z) \to Z$ of $p$ to $Z = \{1\} \times Z$ is a principal $G$-bundle over $Z$.  By assumption
  $p|_{p^{-1}(Z)}$ is isomorphic to the trivial principal $G$-bundle over $Z$. Hence we can
  choose a $G$-homeomorphism
  \[
  f \colon Z \times G \xrightarrow{\cong} p^{-1}(Z)
  \]
  such that the composite $p|_{p^{-1}(Z)} \circ f$ is the canonical projection $Z \times G
  \to Z$. Let $\sigma \colon Z \to \hom(H,G)$ be defined by the map $\rho_{f(z,1)}$, see~\eqref{rho_e}, i.e., $\sigma$ is uniquely determined by 
$h \cdot f(z,1) = f(z,1)   \cdot \sigma(z)(h)$.  (It is continuous because
  of the properties of the category of compactly generated spaces listed in 
  Subsection~\ref{subsec:Basic_feature_of_the_category_of_compactly_generated_spaces}.)
We define a left $H$-action on $Z \times G$ by $h \cdot (z,g) =
  \bigl(z,\sigma(z)(h) \cdot g)$.  Then $f$ is compatible with the left $H$-actions by the
  following calculation
   \begin{eqnarray*}
    f\bigl(h \cdot (z,g)\bigr) 
    & = & 
    f\bigl(z, \sigma(z)(h)\cdot g\bigr)  
    \\
    & = & 
    f(z,1) \cdot  \sigma(z)(h)\cdot g
    \\
    & = & 
    h \cdot f(z,1) \cdot   g
    \\
    & = & 
    h \cdot f(z,g).
  \end{eqnarray*}
    We obtain a homeomorphism compatible with the obvious left $\Gamma$-actions and with the
  obvious right $G$-actions
  \[
  \Gamma \times_H (Z \times G) \xrightarrow{\id_{\Gamma} \times_H f} \Gamma \times_H
  p^{-1}(Z) \xrightarrow{u} E.
  \]
  This is an isomorphism of $\Gamma$-equivariant principal $G$-bundles from $p_{\sigma}
  \colon \Gamma \times_H (Z \times G) \to \Gamma/H \times Z $ to $p \colon E \to \Gamma/H \times Z$.
  \\[2mm]~\ref{lem:equivariant_principal_G-bundles_over_Gamma/H_times_Z:uniqueness} Let
  \[
  f \colon \Gamma \times _H (Z \times G) \xrightarrow{\cong} \Gamma \times_H (Z \times G)
  \]
  be an isomorphism of $\Gamma$-equivariant principal $G$-bundles from $p_{\sigma_0}$ to
  $p_{\sigma_1}$.  Notice that the left $H$-actions on $(Z \times G)$ in the source and
  the target are different, the first one depends on $\sigma_0$, the second on $\sigma_1$.
  There is precisely one map $\omega \colon Z \to G$ satisfying $f(1,(z,1)) =
  \bigl(1,(z,\omega(z))\bigr)$.  (It is continuous because of
  the properties of the category of compactly generated spaces listed in 
  Subsection~\ref{subsec:Basic_feature_of_the_category_of_compactly_generated_spaces}.)
  Since $f$ is compatible with the natural left
  $\Gamma$-actions and with the natural right $G$-actions, we get for $h \in H$ and $z \in  Z$
  \begin{eqnarray*}
    \bigl(1,(z,\sigma_1(z)(h) \cdot \omega(z))\bigr)
    & = &
    \bigl(1,h \cdot (z,\omega(z))\bigr)
    \\
    & = &
    \bigl(h, (z,\omega(z))\bigr)
    \\
    & = & 
    h \cdot \bigl(1,(z,\omega(z))\bigr)
    \\
    & = & 
    h \cdot f\bigl(1,(z,1)\bigr)
    \\
    & = & 
    f\bigl(h \cdot (1,(z,1))\bigr)
    \\
    & = & 
    f\bigl(h,(1,z)\bigr)
    \\
    & = & 
    f\bigl(1, h \cdot (z,1)\bigr)
    \\
    & = &
    f\bigl(1,(z,\sigma_0(z)(h)\bigr)
    \\
    & = &
    f\bigl(1,(z,1)\bigr) \cdot \sigma_0(z)(h)
    \\
    & = &
    (1,z,\omega(z))\cdot \sigma_0(z)(h)
    \\
    & = &
    \bigl(1,(z,\omega(z)\cdot \sigma_0(z)(h))\big).
  \end{eqnarray*}
  This implies $\sigma_1(z)(h) \cdot \omega(z) = \omega(z)\cdot \sigma_0(z)(h)$ for all 
  $h \in H$ and $z \in Z$.
  
  For the converse, the map
  \begin{align*}
  \Gamma \times_H(Z \times G)  & \to  \Gamma \times_H(Z \times G)\\
  (\gamma,(z,g)) & \mapsto (\gamma,(z,w(z)g)))
  \end{align*}
  satisfies the desired properties.
  \\[2mm]~\ref{lem:equivariant_principal_G-bundles_over_Gamma/H_times_Z:local_system}
  Consider $(\gamma,(z,g)) \in \Gamma \times_H (Z \times G)$. One easily checks
  \[
  \Gamma_{(\gamma H ,z)} = \gamma \cdot \Gamma_{(1H,z)} \cdot \gamma^{-1} = \gamma \cdot H
  \cdot \gamma^{-1}.
  \]
  We compute for $h \in H$ in $\Gamma \times_H (Z \times G)$
  \begin{eqnarray*}
    (\gamma \cdot h \cdot \gamma^{-1}) \cdot \bigl(\gamma ,(z,g)\bigr)
    & = & 
    \bigl(\gamma \cdot h,(z,g)\bigr)
    \\
    & = & 
    \bigl(\gamma, h \cdot (z,g)\bigr)
    \\
    & = & 
    \bigl(\gamma, (z,\sigma(z)(h) \cdot g)\bigr)
    \\
    & = & 
    \bigl(\gamma,(z,g \cdot (g^{-1}\cdot \sigma(z)(h) \cdot g))\bigr)
    \\
    & = & 
    \bigl(\gamma ,(z,g)\bigr) \cdot (g^{-1}\cdot \sigma(z)(h) \cdot g).
  \end{eqnarray*}
  This finishes the proof of
  Lemma~\ref{lem:equivariant_principal_G-bundles_over_Gamma/H_times_Z}.
\end{proof}


\section{Discussion of homotopy invariance and Condition (H)}
\label{sec:Discussion_of_homotopy_invariance_and_property_(H)}

Since we want to have a bundle theory for which there exists a universal bundle, we have to
ensure homotopy invariance, i.e., if $f_0,f_1 \colon B_0 \to B_1$ are $\Gamma$-maps and $p
\colon E \to B_1$ is a $\Gamma$-equivariant principal $G$-bundle, we want to arrange that
the pullbacks $f_0^*p$ and $f_1^*p$ are isomorphic as $\Gamma$-equivariant principal
$G$-bundles.

Let $Z$ be a contractible $CW$-complex and $H \subseteq \Gamma$ be a subgroup. Equip $Z$ with
the trivial $\Gamma$-action. Then the projection $\pr \colon \Gamma/H \times Z \to
\Gamma/H$ is a $\Gamma$-homotopy equivalence. Hence in order to guarantee homotopy
invariance, we must ensure that every $\Gamma$-equivariant principal $G$-bundle $p \colon
E \to \Gamma/H \times Z$ over $\Gamma/H \times Z$ is isomorphic to $\pr^* E'$ for some
$\Gamma$-equivariant principal $G$-bundle on $p' \colon E' \to \Gamma/H$ over $\Gamma/H$.

For $\alpha \in \hom(H,G)$ the  \emph{centralizer of $\alpha$ in $G$} is defined to be the subgroup of $G$ given by
\[
C_G(\alpha) := \{g \in G \mid g\alpha(h)g^{-1} =\alpha(h) \; \text{for all}\; h \in H\}.
\]

\begin{definition}[Condition (H)]\label{def:property_(H)}
  A family $\calr$ of local representations in the sense of
  Definition~\ref{def:local_representations} satisfies Condition (H) if the following
  holds for every $(H,\alpha) \in \calr$:

  \begin{enumerate}

  \item \label{def:property_(H):component} The path component of $\alpha$ in $\hom(H,G)$
    is contained in $\{c_g \circ \alpha \mid g \in G\}$;

  \item \label{def:property_(H):S_for_G} The pair $(G,C_G(\alpha))$ satisfies Condition
    (S) introduced in Definition~\ref{ref:condition_(S)};

  \item \label{def:property_(H):S_for_Gamma} The pair $(\Gamma,H)$ satisfies Condition (S)
    introduced in Definition~\ref{ref:condition_(S)};

  \item \label{def:property_(H):homeo} The canonical map
    \[\iota_{\alpha} \colon G/C_G(\alpha) \to \hom(H,G), \quad gC_G(\alpha) \mapsto c_g
    \circ \alpha
    \]
    is a homeomorphism onto its image.

  \end{enumerate}
\end{definition}

\begin{lemma}\label{lem_property_(H)_and_bundles_over_Gamma/H_times_Z}
  Suppose that the family $\calr$ of local representations satisfies Condition (H). Let
  $Z$ be a (non-equivariant) contractible $CW$-complex and let $p \colon E \to \Gamma/H \times Z$ be a
  $\Gamma$-equivariant principal $G$-bundle with $\calr(p) \subseteq \calr$.

  Then $p$ is isomorphic to $\pr^*E'$ for a $\Gamma$-equivariant principal $G$-bundle $p'
  \colon E'\to \Gamma/H$ for the projection $\pr \colon \Gamma/H \times Z \to \Gamma/H$,
  or, equivalently, there exists an element $(H,\alpha)$ in $\calr$ such that $p$ is
  isomorphic to the $\Gamma$-equivariant principal bundle 
  \[
   p_{\alpha} \colon ( \Gamma \times_H G) \times Z \to \Gamma/H \times Z, 
   \quad \bigl((\gamma,g),z\bigr) \mapsto (\gamma H, z).
  \]
\end{lemma}
\begin{proof}
   Since $Z$ is contractible, by Lemma~\ref{lem:homotopy_invariance_for_trivial_Gamma}, the restriction bundle
  $p|_{\{1H\} \times Z} : E|_{\{1H\} \times Z} \to \{1H\} \times Z$ is trivializable. We can then apply
  Lemma~\ref{lem:equivariant_principal_G-bundles_over_Gamma/H_times_Z}
~\ref{lem:equivariant_principal_G-bundles_over_Gamma/H_times_Z:given_by_rho} and we obtain that $p$ is
  isomorphic to $p_{\sigma}$ for an appropriate map  $\sigma \colon Z  \to \hom(H,G)$.
  Since $Z$ is path connected, the image of $\sigma$ is contained in a path component of 
  $\alpha \in \hom(H,G)$ if we take $\alpha = \sigma(z)$ for some $z \in Z$. 
  Lemma~\ref{lem:equivariant_principal_G-bundles_over_Gamma/H_times_Z}
~\ref{lem:equivariant_principal_G-bundles_over_Gamma/H_times_Z:local_system} implies
  that $(H,\alpha)$ belongs to $\calr(p)$ and hence to $\calr$.  Because of Condition (H)
  the image of $\sigma \colon Z \to \hom(H,G)$ is contained in the image of
  $\iota_{\alpha} \colon G/C_G(\alpha) \to \hom(H,G)$.  Since $\iota_{\alpha}$ is a
  homeomorphism onto its image by Condition (H), we can find a map $\overline{\omega}
  \colon Z \to G/C_G(\alpha)$ with $\iota_{\alpha} \circ \overline{\omega} = \sigma$.
  Because of Condition (H) and Remark~\ref{rem:(S)_and_principal}, the projection $G \to
  G/C_G(\alpha)$ is a principal $C_G(\alpha)$-bundle.  Hence its pullback with
  $\overline{\omega}$ is a principal $C_G(\alpha)$-bundle over the contractible
  $CW$-complex $Z$ and hence has a section by
  Lemma~\ref{lem:homotopy_invariance_for_trivial_Gamma}.  Thus we can find a map $\omega
  \colon Z \to G$ whose composite with the projection $G \to G/C_G(\alpha)$ is
  $\overline{\omega}$. This implies $\sigma = c_{\omega} \circ \alpha$. Now apply
  Lemma~\ref{lem:equivariant_principal_G-bundles_over_Gamma/H_times_Z}%
~\ref{lem:equivariant_principal_G-bundles_over_Gamma/H_times_Z:uniqueness}.
\end{proof}

\begin{theorem} \label{the:criterion_for_(H)} Let $\calr$ be a family of local
  representations. Then it satisfies Condition (H) if the following conditions are
  satisfied:

  \begin{enumerate}

  \item The group $\Gamma$ is locally compact, second countable and has finite
    covering dimension, e.g., is a  Lie groups, or $\Gamma$ is completely regular and all compact subgroups of $\Gamma$ are 
   Lie groups, e.g.,  is a Kac-Moody group;

  \item The group $G$ is  locally compact, second countable and has finite
    covering dimension, e.g., is a  Lie group;
 
  \item The group $G$ is almost connected (see Definition~\ref{def:almost_connected}); 
  \item For every element $(H,\alpha)$ the group $H$ is a compact group.
  \end{enumerate}
\end{theorem}
\begin{proof}
  Condition~\ref{def:property_(H):component} appearing in Definition~\ref{def:property_(H)} is proved
  in Theorem~\ref{the:connected_components_and_conjugation}.

  Conditions~\ref{def:property_(H):S_for_G} and~\ref{def:property_(H):S_for_Gamma}
  appearing in Definition~\ref{def:property_(H)} follow from
  Lemma~\ref{lem:condition_(S)}~\ref{lem:condition_(S):Gamma}.

  Condition~\ref{def:property_(H):homeo} appearing in
  Definition~\ref{def:property_(H)} is proved in 
  Theorem~\ref{the:G/C_G(alpha)_homeo_G.alpha}.
\end{proof}


  \section{The Slice Theorem for equivariant $CW$-complexes}
  \label{subsection:The_Slice_Theorem_for_equivariant_CW-complexes}

In this section we prove the following Slice Theorem for $\Gamma$-$CW$-complexes,
generalizing~\cite[Theorem~1.37]{Lueck(1989)}.

\begin{theorem}[Slice Theorem]
\label{the:Slice_Theorem}
Let $G$ be a (compactly generated) topological group and let $X$ be a $G$-$CW$-complex. 
Consider $x \in X$ together with a $G_x$-invariant neighborhood $V_x$. Suppose
that  the pair $(G,G_x)$ satisfies Condition (S), see Definition~\ref{ref:condition_(S)}.

Then there exists a  $G_x$-invariant subset $S_x$ of $x$ with the following properties:

\begin{enumerate}

\item \label{the:Slice_Theorem:compactly_generated}
$S_x$ inherited with the subspace topology is compactly generated;

\item \label{the:Slice_Theorem:compactly_generated:closure_of_S_x_subset_V_x}
The closure of $S_x$ is contained in $V_x$;

\item \label{the:Slice_Theorem:G_x-homotopy_equivalence}
The inclusion $\{x\} \to S_x$ is a $G_x$-homotopy equivalence;

\item \label{the:Slice_Theorem:G-homeomorphism}
The set $U := G \cdot S_x$ is a quasi-regular $G$-invariant open subset of $X$ and the map
\[
G \times_{G_x} S_x \xrightarrow{\cong} U, \quad (g,s) \mapsto g \cdot s
\]
is a $G$-homeomorphism.
\end{enumerate}
\end{theorem}
\begin{proof}
Let $n_x \ge 0$ be the integer for which $x \in X_{n_x}$ and $x \notin X_{n_x-1}$. 
We construct inductively for $n = n_x , n_x +1, \ldots$
open $G$-invariant subsets $U[n] \subseteq X$ such that the following
conditions hold if we put $V = G \cdot V_x$:

\renewcommand{\labelenumi}{(\arabic{enumi})}
\begin{enumerate}

\item \label{the:Slice_Theorem:inclusions}
We have $U[n] \subseteq X_n$, $U[n] \subseteq U[n+1]$, $\overline{U[n]} \subseteq V$ and $x \in U[n]$
for all $n \ge n_x$;

\item \label{the:Slice_Theorem:retractions}
For each $n \ge n_x$ there is a $G$-map 
\[
r[n+1] \colon U[n+1]\to U[n]
\]
satisfying $r[n+1] \circ i[n+1] = \id_{U[n]}$,
where $i[n+1] \colon U[n] \to U[n+1]$ is the inclusion;

\item \label{the:Slice_Theorem:(homotopies)}
For each $n \ge n_x$ there is a $G$-homotopy 
\[
h[n+1] \colon U[n+1] \times [0,1] \to U[n+1]
\]
satisfying
\[\begin{array}{rcll}
h[n+1]_t                  & = & \id_{U[n+1]}                  & \text{for}\; 0 \le t \le (n+3)^{-1};
\\
h[n+1]_{t}                & = & i[n+1] \circ r[n+1]      &  \text{for}\;  (n+2)^{-1} \le t \le 1;
\\
h[n+1](z,t)               & = & z                                & \text{for}\;  z \in U[n], t \in [0,1];
\\
r[n+1] \circ h[n+1] & = & r[n+1] \circ \pr_{n+1}, &
\end{array}
\]
where $\pr_{n+1} \colon U[n+1] \times [0,1] \to U[n+1]$ is the projection;

\item \label{the:Slice_Theorem:beginning}
There is a $G$-map 
\[
r[n_x] \colon U[n_x] \to Gx
\]
such that $r[n_x] \circ i[n_x] = \id_{Gx}$ holds, where
$Gx := \{gx\mid g \in G\}$ and $i[n_x] \colon Gx \to U[n_x]$ is the inclusion;

\item There is a $G$-homotopy 
\[
h[n_x] \colon U[n_x] \times [0,1] \to U[n_x]
\]
satisfying 
\[
\begin{array}{rcll} 
h[n_x]_t               & = & \id_{U_{n_x}}                 & \text{for}\; t \le (n_x+2)^{-1};
\\
h[n_x]_t                & = & i[n_x] \circ r[n_x]       &  \text{for}\; t \ge (n_x+1)^{-1};
\\
h[n_x](gx,t)          & = & gx                             &   \text{for}\;  gx \in Gx, t \in [0,1];
\\
r[n_x] \circ h[n_x] & = & r[n_x] \circ \pr_{n_x}, &
\end{array}
\]
where $\pr_{n_x} \colon U[n_x] \times [0,1] \to U[n_x]$ is the projection.
\end{enumerate}

Recall that we have for each $n\ge 0$ a $G$-pushout of the form
\[
\xymatrix@!C= 8em{\coprod_{i \in I_{n+1}} G/H_i \times S^n \ar[r]^-{\coprod_{i \in I_n} q_i^{n+1}} \ar[d] & X_n \ar[d]
\\
\coprod_{i \in I_{n+1}} G/H_i \times D^{n+1} \ar[r]^-{\coprod_{i \in I_n} Q_i^{n+1}} \ar[r] & X_{n+1} 
}
\]
Next we explain the beginning of the induction. Since $x$ belongs to $X_{n_x}$ but not to $X_{n_x-1}$,
we can find $i \in I_{n_x}$ and $(\gamma H_i,y) \in G/H_i \times \bigl(D^{n_x} \setminus S^{n_x -1}\bigr)$
satisfying $Q_i^{n_x}(g H_i,y) = x$. Choose $\delta_i > 0$
such that $G/H_i \times \overline{B_{\delta}(y)}$ is contained in both 
$G/H_i \times D^{n_x} \setminus S^{n_x -1}$ and $(Q_i^{n_x})^{-1}(V_x \cap X_{n_x})$,
where $B_{\delta}(y)$ is the ball of radius $\delta$ around $y$. 
Define a $G$-map 
\[
r[n_x]' \colon G/H_i \times B_{\delta}(y) \to Gy, \quad (gH_i,z) \mapsto gy,
\]
and a $G$-homotopy
\[
h[n_x]' \colon G/H_i \times B_{\delta}(y) \times [0,1] \to G/H_i \times B_{\delta}(y)
\]
by sending $(gH_i,z,t)$ to
\[
\begin{array}{ll}
(gH_i,z)  &  \text{for}\; t \le (n_x+2)^{-1};
\\
\left(gH_i,\frac{t- (n_x+ 2)^{-1}}{(n_x+1)^{-1} - (n_x +2)^{-1}}   
\cdot y +   \bigl(1 - \frac{t- (n_x+2)^{-1}}{(n_x+1)^{-1} - (n_x +2)^{-1}}\bigr) \cdot z\right)
& \text{for}   (n_x+2)^{-1}  \le t \le (n_x+1)^{-1};
\\
(gH_i,y)  &  \text{for}\;  t \ge (n_x+1)^{-1}.
\end{array}
\]
Now define 
\[
U[n_x] = Q_i^{n_x}\bigl(G/H_i \times B_{\delta}(y)\bigr).
\]
Let 
\[
h[n_x] \colon U[n_x] \times [0,1] \to U[n_x]
\]
be the $G$-homotopy  uniquely determined by the property that for every element 
$(gH_i,z,t) \in G/H_i \times B_{\delta}(y) \times  [0,1]$
we have $Q_i^{n_x} \circ h[n_x]'(gH_i,z,t) = h[n_x](Q_i^{n_x}(gH_i,z),t)$ holds. Define the $G$-map 
\[
r[n_x] \colon U[n_x] \to Gx
\]
analogously using $r[n_x]'$. One easily checks that $U[n_x]$, $r[n_x]$  and $h[n_x]$ have the desired properties.

Next we explain the induction step from $n \ge n_x$ to $n+1$. 
For a real number $\epsilon \in (0,1)$ define the  subspace
\[S^n[\epsilon] := \{t \cdot z \mid t \in (1-\epsilon ,1], z \in S^n\} \subseteq  D^{n+1},
\]
a map 
\[
p[\epsilon] \colon S^n[\epsilon] \to S^n, \quad t \cdot z \mapsto z,
\]
and a homotopy
\[
l[\epsilon] \colon S^n[\epsilon] \times [0,1] \to S^n[\epsilon], \quad (t \cdot z,s) \mapsto (t \cdot (1-s) + s) \cdot z.
\]
If $j[\epsilon] \colon S^n \to S^n[\epsilon]$ is the inclusion, then $l[\epsilon]_0 = \id_{S^n[\epsilon]}$ 
and $l[\epsilon]_1 = i[\epsilon] \circ p[\epsilon]$. Roughly speaking, the homotopy
$l[\epsilon]$ pushes $S^n[\epsilon]$ radially to  $S^n$.

Roughly speaking, we will obtain $U[n+1]$ from $U[n]$ by a thickening into the interior
of the various equivariant $(n+1)$-cells such  that the thickening is  small enough to ensure
that the closure of $U[n+1]$ stays within $V$. In detail:
Consider $i \in I_{n+1}$. Then we get inclusions of $G$-invariants subsets of 
$\Gamma/H_i \times S^n$.
\[
(q^{n+1}_i)^{-1}(U[n]) \subseteq (q^{n+1}_i)^{-1}(\overline{U[n]}) \subseteq (q_i^{n+1})^{-1}(V)
\]
because of the  induction hypothesis.  Notice that $(Q_i^{n+1})^{-1}(V)$ is a 
$G$-invariant open subset of $G/H_i \times D^{n+1}$
containing the closed $G$-invariant subset $(q^{n+1}_i)^{-1}(\overline{U[n]})$.
Let $\pr \colon G/H_i \times D^{n+1} \to D^{n+1}$ be the projection.
Then $\pr\bigl((Q_i^{n+1})^{-1}(V)\bigr)$ is an open subset of $D^{n+1}$
containing the closed  subset $\pr\bigl((q^{n+1}_i)^{-1}(\overline{U[n]})\bigr)$.
Since $S^n$ is compact, we can choose $\epsilon_i > 0$ such that
\[
p[2 \cdot \epsilon_i]^{-1}\left(\pr\bigl((q^{n+1}_i)^{-1}(\overline{U[n]})\bigr)\right) 
\subseteq \pr\left((Q_i^{n+1})^{-1}(V)\right)
\]
holds. Define
\[
U[n+1] := U[n] \cup \bigcup_{i\in I_{n+1}} 
Q_i^{n+1}\left(\pr^{-1}\left(p[\epsilon_i]^{-1}\left(\pr\bigl((q^{n+1}_i)^{-1}(\overline{U[n]})\bigr)\right)\right)\right).
\]
One easily checks that $U[n+1]$ is an open
$G$-invariant subset of $X_{n+1}$ such that $\overline{U[n+1]} \subseteq V$ holds.  
The
various $G$-maps
$\id_{G/H_i} \times p[\epsilon_i] \colon G/H_i \times S^n[\epsilon] \to G/H_i \times S^n$
fit together to a $G$-map
\[
r[n+1] \colon U[n+1] \times [0,1] \to U[n+1]
\]
such that $r[n+1] \circ i[n+1] = \id_{U[n]}$ holds for the inclusion $i[n+1] \colon U[n] \to U[n+1|$.
The various $G$-homotopies 
$\id_{G/H_i} \times l[\epsilon_i] \colon G/H_i \times S^n[\epsilon] \to G/H_i \times S^n$
fit together to a $G$-homotopy 
\[
h[n+1]' \colon U[n+1] \times [0,1] \to U[n+1]
\]
such that $h[n+1]'_0 = \id_{U[n+1]}$, the image of $h[n+1]_1$ lies in $U[n+1] \cap X_n = U[n]$, 
the restriction of  $h[n+1]'_t \colon U[n+1] \to U[n+1]$ to $U[n]$ is the inclusion $U[n] \to U[n+1]$
for all $t \in [0,1]$ and $r[n+1] \circ h[n+1] =  r[n+1] \circ \pr_{n+1}$ for the projection
$\pr_{n+1} \colon U[n+1] \times [0,1] \to U[n+1]$. Define a map
\[
\tau \colon [0,1] \to [0,1], \quad t \mapsto 
\begin{cases}
0 
& 
0 \le t \le (n +3)^{-1};
\\
\frac{t - (n +3)^{-1}}{(n+2)^{-1} - (n +3)^{-1}} 
& 
(n+3) ^{-1} \le t \le (n+2)^{-1};
\\
1 & 
(n+2)^{-1} \le t.
\end{cases}
\]
One easily checks that the $G$-homotopy
\[
h[n+1] := h[n+1]' \circ \bigl(\id_{U[n+1]} \times \tau\bigr) \colon U[n+1]  \times [0,1] \to U[n+1]
\] 
has the desired properties. This finishes the induction step.

Now define
\[
U := \bigcup_{n \ge n_x} U[n]
\]
Since each $U[n] \subseteq X_n$ is an open $G$-invariant subset of $X_n$ and
$X$ has the weak topology with respect to the filtration by its skeletons $\{X_n \mid n \ge -1\}$,
the subset $U$ of $X$ is open and $G$-invariant. Since 
$\overline{U_n} \subseteq V$ holds for all $n \ge 0$, we conclude
$\overline{U} \subseteq V$. Define the $G$-map
\[
\rho[n] := U[n] \to Gx
\]
to be the composite $U[n] \xrightarrow{r[n]} U[n-1] \xrightarrow{r[n-1]} \cdots  
\xrightarrow{r[n_x+1]} U[n_x] \xrightarrow{r[n_x]} Gx$.
Since $\rho[n+1]$ restricted to $U[n]$ is $\rho[n]$, we obtain a $G$-map
\[
\rho \colon U \to Gx.
\]
Let $\iota \colon Gx \to U$ be the inclusion. Obviously $\rho \circ \iota = \id_{Gx}$.
Next we define inductively for $n = n_x, n_x +1, \ldots $ $G$-homotopies
\[
k[n] \colon U[n] \times [0,1] \to U[n]
\]
such that $k[n]_t = \id_{U[n]}$ holds for $t \in [0,(n+2)^{-1}]$, we have $k[n]_1 = \rho[n]$, the composite
$\rho[n] \circ k[n] \colon U[n] \times [0,1] \to Gx$ factorizes over the projection $U[n] \times [0,1]  \to U[n]$
to $\rho[n] \colon U[n] \to Gx$,  and the following diagram commutes
\[
\xymatrix@!C= 7em{
U[n] \times [0,1] \ar[r]^-{k[n]} \ar[d]^{i[n+1] \times \id_{[0,1]}}
&
U[n] \ar[d]^{i[n+1]}
\\
U[n+1] \times [0,1] \ar[r]^-{k[n+1]} 
&
U[n+1]
}
\]
In the induction beginning put $k[n_x] = h[n_x]$. 
In the induction step from $n \ge n_x$ to $n+1$, we define
$k[n+1]|_{U[n+1] \times [0,(n+2)^{-1}]}$ to be $h[n+1]|_{U[n+1] \times [0,(n+2)^{-1}]}$,
and define $k[n+1]|_{U[n+1] \times [(n+2)^{-1},1]}$ to be the composite
\begin{multline*}
U[n+1] \times [(n+2)^{-1},1] \xrightarrow{r[n+1] \times \id_{[(n+2)^{-1},1]}} U[n] \times [(n+2)^{-1},1] 
\\
\xrightarrow{k[n]|_{U[n] \times [(n+2)^{-1},1]}} 
U[n] \xrightarrow{i[n+1]} U[n+1].
\end{multline*}
The homotopies $k[n]$ for $n = n_x, n_x+1, \ldots$ fit together to a $G$-homotopy
\[
k \colon U \times [0,1] \to U
\]
with  $k_0 = \id_U$ and $k_1 = \iota \circ \rho$. To summarize, we have 
$\rho \circ \iota = \id_{Gx}$ and $k \colon \iota \circ \rho \simeq_G \id_U$, and 
the composite $\rho \circ k \colon U \times [0,1] \to Gx$ factorizes
over the projection $U \times [0,1] \to U[n]$ to $\rho$.

Let $S_x$ be the preimage of $x$ under the map $\rho \colon U \to Gx$. 
The $G$-map $G/G_x \to Gx, \; g \mapsto gx$ is obviously a 
$G$-homeomorphism.  Since the pair $(G,G_x)$ satisfies Condition (S) by assumption, 
Lemma~\ref{lem:condition_(S)_and_homeomorphisms}  implies that the map 
\[
u \colon G \times_{G_x} S_x \xrightarrow{\cong} U, \quad (g,s) \mapsto gs
\]
is a $G$-homeomorphism. The restriction of $\iota$ 
to $\{x\}$ is the inclusion $i \colon \{x\} \to S_x$.
The restriction of $\rho$ to $S_x$ is a $G_x$-homotopy inverse of $i$
since the composite $\rho \circ k \colon U \times [0,1] \to Gx$ factorizes
over the projection $U \times [0,1] \to U[n]$ to $\rho$.

Since $V$ is by definition $G \cdot V_x$ and the closure
of $U$ is contained in $V$, the closure of $S_x$ is contained in $V_x$. 

We conclude from  
Lemma~\ref{lem:properties_of_quasi-regular_open_sets}~\ref{lem:properties_of_quasi-regular_open_sets:Gamma-CW} 
that $U = \pr^{-1}(\pr(U))$ is quasi-regular and hence $U$ equipped with the subspace topology
is compactly generated. Since $S_x \subseteq U$ is closed, also $S_x$ is compactly generated.
This finishes the proof of Theorem~\ref{the:Slice_Theorem}.
\end{proof}

\renewcommand{\labelenumi}{(\roman{enumi})}


\section{Equivariant principal bundles versus equivariant $CW$-complexes}
\label{sec:Equivariant_principal_bundles_versus_equivariant_CW-complexes}

In this section we prove one of our main technical results saying that a
$\Gamma$-equivariant principal $G$-bundle $p \colon E \to B$ is the same as a $\Gamma
\times G$-$CW$-complex $E$ with a special structure of its $\Gamma \times G$-isotropy
groups provided that $\calr(p) \subseteq \calr$ holds for a given family of local
representations $\calr$ satisfying Condition (H). 

\begin{theorem} \label{the:bundles_and_Gamma_times_G-spaces} Let $\calr$ be a family of
  local representations for $(\Gamma,G)$ satisfying Condition (H) introduced in
  Definition~\ref{def:property_(H)}.

  \begin{enumerate}

  \item \label{the:bundles_and_Gamma_times_G-spaces:b_to_s} Let $p \colon E \to B$ be a
    $\Gamma$-equivariant principal $G$-bundle with $\calr(p) \subseteq \calr$ over a
    $\Gamma$-$CW$-complex $B$.  Then $E_l$ is a $\Gamma \times G$-$CW$-complex whose
    isotropy groups belong to the family $\calf(\calr)$ introduced in Lemma~\ref{lem:family_associated_to_calr};

  \item \label{the:bundles_and_Gamma_times_G-spaces:s_to_b} Let $E$ be a left $\Gamma
    \times G$-$CW$-complex whose isotropy groups belong to $\calf(\calr)$.  Then $p \colon
    E_r\to E_r/G$ is a $\Gamma$-equivariant principal $G$-bundle with $\calr(p) \subseteq \calr$.

  \end{enumerate}
\end{theorem}
\begin{proof}~\ref{the:bundles_and_Gamma_times_G-spaces:b_to_s} Let $B_n$ be the
  $n$-skeleton of the $\Gamma$-$CW$-complex structure on $B$.  Put $E_n := p^{-1}(B_n)$.
  Let $p_n \colon E_n \to B_n$ be the $\Gamma$-equivariant principal $G$-bundle obtained
  by restricting $p$ to $B_n$. Next we will show that the filtration of $E$ by the $E_n$-s
  induces the structure of a $\Gamma\times G$-$CW$-structure on $E_l$ whose isotropy
  groups belong to $\calf(\calr)$.

  Fix $n \ge 0$. Since $B$ is a $\Gamma$-$CW$-complex, there exists a $\Gamma$-pushout
  \[
  \xymatrix@!C=7em{ \coprod_{i \in I} \Gamma/H_i \times S^{n-1} \ar[r]^-{\coprod_{i \in I} q_i}
    \ar[d]_j & B_{n-1} \ar[d]^{J}
    \\
    \coprod_{i \in I} \Gamma/H_i \times D^n \ar[r]_-{\coprod_{i \in I} Q_i} & B_n }
  \]
  where $j$ and $J$ are the inclusions. Consider the square obtained by the pullback
  construction applied to $p_n \colon E_n \to B_n$.
  \[
  \xymatrix@!C=8em{\coprod_{i \in I} q_i^*J^*E_n = j^*Q_i^*E_n \ar[r]^-{\coprod_{i \in I}
      \overline{q_i}} \ar[d]_{\overline{j}} & J^*E_n \ar[d]^{\overline{J}}
    \\
    \coprod_{i \in I} Q_i^*E_n \ar[r]_-{\coprod_{i \in I} \overline{Q_i}} & E_n }
  \]
  This is a diagram of $\Gamma \times G$-spaces, the various left $\Gamma \times
  G$-actions come from the left $\Gamma \times G$-action on $E_l$.  It is a  $\Gamma \times G$-pushout
  by Lemma~\ref{lem_pushouts-pullbacks}~\ref{lem_pushouts-pullbacks:equivariant_versus_non-equivariant}
  and~\ref{lem_pushouts-pullbacks:pullbacks}. 

  We conclude from Remark~\ref{rem:local_representations} that $\calr(Q_i^*E_n) \subseteq
  \calr$ holds.  Lemma~\ref{lem_property_(H)_and_bundles_over_Gamma/H_times_Z} implies
  that the $\Gamma$-equivariant principal $G$-bundle $Q_i^*E_n$ is isomorphic to the
  $\Gamma$-equivariant principal $G$-bundle 
   $p_{\alpha} \colon (\Gamma \times_H  G) \times D^n \to \Gamma/H \times D^n$ for some element 
   $(H,\alpha) \in \calr$. Hence there exists
  a $\Gamma \times G$-homeomorphism of left $\Gamma \times G$-pairs
  \begin{eqnarray}
    (\Gamma \times G)/H_i' \times (D^n,S^{n-1}) 
    &\xrightarrow{\cong} & 
    \bigl((Q_i^*E_n)_l,(q_i^*J^*E_n)_l\bigr)
    \label{desired_Gamma_times_G-homeo}
  \end{eqnarray}
  for an appropriate subgroup $H_i'\subseteq \Gamma \times G$ belonging to $\calf(\calr)$.

  It remains to show that $E$ has the weak topology with respect to the filtration given
  by the $E_n$-s. This follows from
  Lemma~\ref{lem:filtrations}~\ref{lem:filtrations:equivariant_versu_non_equivariant} 
  and~\ref{lem:filtrations:preimages}. This finishes the
  proof of assertion~\ref{the:bundles_and_Gamma_times_G-spaces:b_to_s}.
  \\[2mm]~\ref{the:bundles_and_Gamma_times_G-spaces:s_to_b} 
   Let $E$ be a left $\Gamma
  \times G$-$CW$-complex whose isotropy groups belong to $\calf(\calr)$.  Firstly, we show
  that $E/G$ is a $\Gamma$-$CW$-complex.  Consider $e \in E$.
  There exists $(H,\alpha) \in \calr$ such that 
  $(\Gamma \times G)_e = K(H,\alpha) = \{(h,\alpha(h) \mid h \in H\}$.
  Hence the image of $(\Gamma \times G)_e$ under the projection $\Gamma \times G \to \Gamma$ is $H$ and hence closed.
  Lemma~\ref{lem:quotient_of_CW-complexes} implies that $E/G$ is a $\Gamma$-$CW$-complex.

  Next we show that $p \colon E_r \to E/G$ is a principal $G$-bundle. Consider $b \in E/G$.  Choose $e
  \in E$ with $p(e) = b$.  From the Slice Theorem~\ref{the:Slice_Theorem}  we obtain a 
  $(\Gamma \times G)_e$-invariant (compactly generated) subspace $S$ of $E$ containing $e$ such that the map
  \[
  f\colon (\Gamma \times G) \times_{(\Gamma \times G)_e} S \to (\Gamma \times G) \cdot S,
  \quad \bigl((\gamma,g),s\bigr) \mapsto (\gamma,g) \cdot s
  \]
  is a $(\Gamma \times G)$-homeomorphism and $(\Gamma \times G) \cdot S$ is an open
  $\Gamma \times G$-invariant (compactly generated) subset of $E$. Put $V := p((\Gamma \times G) \cdot S)$. 
  Choose $(H,\alpha) \in \calr$ such that
  \[
  (\Gamma \times G)_e = K(H,\alpha) = \{(h,\alpha(h) \mid h \in H\}.
  \]
  We equip $S$ with the left $H$-action given by $h \cdot s := (h,\alpha(h)) \cdot s$.  We
  obtain a $\Gamma$-map
  \[
  u \colon \Gamma \times_H S \to V, \quad (\gamma, s) \to p \circ
  f\bigl((\gamma,e),s\bigr).
  \]
  Define the $\Gamma \times G$-map
  \[q \colon (\Gamma \times G) \times_{(\Gamma \times G)_e} S \to \Gamma \times_H S, \quad
  \bigl((\gamma,g),s\bigr) \mapsto (\gamma, s),
  \]
  where $G$ acts trivially on the target.  Then $V$ is an open $\Gamma$-invariant
  neighborhood of $b$ in $E/G$ which is quasi-regular by 
   Lemma~\ref{lem:properties_of_quasi-regular_open_sets}~\ref{lem:properties_of_quasi-regular_open_sets:Gamma-CW}.
   Moreover, the following diagram of right $G$-spaces
   commutes, where $G$ acts trivial on $\Gamma \times _H S$ and $V$, has a
  $G$-homeomorphism as upper horizontal map and identifications as vertical maps
  \begin{eqnarray}
    & 
    \xymatrix{
      \left((\Gamma \times G) \times_{(\Gamma \times G)_e} S\right)_r  \ar[r]^-{f}_-{\cong} \ar[d]_{q}
      &
      p^{-1}(V)   \ar[d]^{p|_{p^{-1}(V)}}
      \\
      \Gamma \times _H S \ar[r]^-{u}_-{\cong} 
      &
      V 
    }
    &
    \label{diagram_for_local_object_for_Gamma_times_G}
  \end{eqnarray}
  Since $u$ is a bijective identification, it is a homeomorphisms. Hence it suffices to
  show that $q$ is a principal $G$-bundle.
  
  Equip $S \times G$ with the left $H$-action given by $h \cdot (s,g) := \bigl(h \cdot s,
  \alpha(h) \cdot g\bigr)$.  Then
  \[
  p_{\alpha} \colon \Gamma \times_H (S \times G) \to \Gamma \times _H S, \quad
  \bigl(\gamma, (s,g)\bigr) \mapsto (\gamma,s)
  \]
  is a well-defined map which is compatible with the obvious left $\Gamma$-actions on the
  source and the target and with the obvious right $G$-action on the source and the
  trivial right $G$-action on the target.  Define a map
  \[
  \xi \colon \Gamma \times _H (S \times G) \xrightarrow{\cong} \left((\Gamma \times G)
    \times_{(\Gamma \times G)_e} S\right)_r, \quad \bigl(\gamma, (s,g)\bigr) \mapsto
  \bigl((\gamma,g^{-1}),s\bigr).
  \]
  It is well-defined by the following calculation for $\gamma \in \Gamma$, $s \in S$ and
  $g \in G$
  \begin{eqnarray*}
    \bigl((\gamma h^{-1} , (\alpha(h) \cdot g)^{-1}), h \cdot s\bigr)
    & = & 
    \bigl((\gamma \cdot h^{-1} , g^{-1} \alpha(h)^{-1}),  h\cdot s\bigr)
    \\
    & = & 
    \bigl((\gamma, g^{-1}) \cdot (h,\alpha(h))^{-1}, (h,\alpha(h)) \cdot s\bigr)
    \\
    & = & 
    \bigl((\gamma,g^{-1}),s\bigr).
  \end{eqnarray*}
  The map $\xi$ is a homeomorphisms, an inverse is given by $\bigl((\gamma,g),s\bigr)
  \mapsto \bigl(\gamma, (s,g^{-1})\bigr)$. It is compatible with the left $\Gamma$-actions
  and the right $G$-actions. We obtain a commutative diagram
  of spaces with left $\Gamma$-actions and right $G$-actions, where the spaces in the lower
  left and lower right corner carry trivial $G$-actions
  \begin{eqnarray}
    & 
    \xymatrix{
      \Gamma \times _H (S \times G)  \ar[r]^-{\xi}_-{\cong} \ar[d]_{p_{\alpha}}
      &
      \left((\Gamma \times G)  \times_{(\Gamma \times G)_e} S\right)_r   \ar[d]^{u \circ q}
      \\
      \Gamma \times _H S \ar[r]^-{u}_-{\cong} 
      &
      V 
    }
    &
    \label{diagram_for_local_object}
  \end{eqnarray}
  Hence it remains to show that 
  $p_{\alpha} \colon \Gamma \times_H (S \times G) \to \Gamma/H \times S$ 
  is a principal $G$-bundle after forgetting the $\Gamma$-action.

  The Condition (S) ensures that there is an open quasi-regular 
neighborhood $U \subseteq \Gamma/H$ of
  $1H$ and a map $s \colon U \to \Gamma$ whose composite with the projection $\pr \colon
  \Gamma \to \Gamma/H$ is the identity.  Define $\overline{s} \colon \pr^{-1}(U) \to H$ by
  $\overline{s}(\gamma) = s \circ \pr(\gamma)^{-1} \cdot\gamma$.  It has the property
  $\overline{s}(\gamma \cdot h) = \overline{s}(\gamma) \cdot h$ for all $\gamma \in
  p^{-1}(U)$.  Define maps
  \begin{eqnarray*}
    \mu \colon (\pr^{-1}(U) \times_H S) \times G 
    & \to &
    \pr^{-1}(U) \times_H(S \times G);
    \\
    \nu \colon \pr^{-1}(U) \times_H(S \times G)
    & \to &
    (\pr^{-1}(U) \times_H S) \times G, 
  \end{eqnarray*}
  by
  \begin{eqnarray*}
    \mu \bigl((\gamma,s),g\bigr)  
    & := &
    \bigl(\gamma, (s, \alpha(\gamma)^{-1} \cdot g)\bigr);
    \\
    \nu \bigl(\gamma,(s,g)\bigr)
    & := &
    \bigl((\gamma,s), \alpha(\gamma) \cdot g\bigr).
  \end{eqnarray*}
  Then $\mu$ and $\nu$ are to one another inverse $G$-homeomorphism. They induce a
  trivialization of $p_{\alpha}$ restricted to $\pr^{-1}(U) \times_H S$ to the trivial
  principal $G$-bundle over $\pr^{-1}(U) \times_H S$.

  Since $\gamma H \in \Gamma/H$ is contained in the open subset $\gamma \cdot U$ and
  $p_{\sigma}$ is a $\Gamma$-equivariant map, we conclude that $p_{\alpha}$ is locally
  trivial and hence a principal $G$-bundle.

  This finishes the proof that $p \colon E_r \to E/G$ is a principal $G$-bundle.  Since
  $p$ is obviously $\Gamma$-equivariant, $p$ is a $\Gamma$-equivariant principal
  $G$-bundle.  It remains to prove $\calr'(p) \subseteq \calr$ since then $\calr(p)
  \subseteq \calr$ holds.

  Consider $e \in E$. Its isotropy subgroup $(\Gamma \times G)_e$ belongs to $\calf(\calr)$ by
  assumption.  Hence there exists $(H,\alpha) \in \calr$ with 
  $(\Gamma \times G)_e =   K(H,\alpha) := \{(h,\alpha(h) \mid h \in H\}$.  
  This implies that $\Gamma_{p(e)} =  H$. We conclude that $\rho_e \colon H \to G$ is given by 
  $\alpha$ since  for $h \in H$ we have 
  $e   = (h,\alpha(h))^{-1}  \cdot e  = (h^{-1},1) \cdot (1,\alpha(h^{-1})) \cdot e$ in the 
  $\Gamma \times G$-space $E$ and hence we get $h \cdot e = e \cdot \alpha(h)$ in $E_r$.  This
  finishes the proof of Theorem~\ref{the:bundles_and_Gamma_times_G-spaces}.

\end{proof}


\section{Local structure}
\label{sec:Local_structure}

Next we deal with the structure of a $\Gamma$-equivariant principal $G$-bundles on small
open $\Gamma_b$-invariant neighborhoods of points $b$ in the base space.

\begin{theorem}[Local structure] \label{the:local_objects} Let $\calr$ be a family of
  local representations for $(\Gamma,G)$ satisfying Condition (H) introduced in
  Definition~\ref{def:property_(H)}.  Let $p \colon E \to B$ be a $\Gamma$-equivariant
  principal $G$-bundle with $\calr(p) \subseteq \calr$.  Consider any point $b \in B$ and
  any $\Gamma_b$-invariant open neighborhood $W$ of $b \in B$.

  Then for every $b$ in $B$ the exists a commutative diagram
  \[
  \xymatrix{ \Gamma \times_{\Gamma_b} (T \times G)\ar[r]^-{f}_-{\cong} \ar[d]_{q} &
    p^{-1}(V) \ar[d]^{p|_{p^{-1}(V)}}
    \\
    \Gamma \times _{\Gamma_b} T \ar[r]^-{u}_-{\cong} & V }
  \]
  with the following properties:
  \begin{enumerate}

  \item The subset $T \subseteq B$ contains $b$, satisfies $T \subseteq W$, is
    $\Gamma_b$-invariant and $\Gamma_b$-contractible;

  \item The subset $V \subseteq B$ is an open $\Gamma$-invariant neighborhood of $b$;
     
  \item The group $\Gamma_b$ acts from the right on $T \times G$ by
    \[
    \gamma \cdot (t,g) := \bigl(\gamma \cdot t,\rho_e(\gamma) \cdot g\bigr),
    \]
    where $(\Gamma_e,\rho_e)$ is the local representation of $p$ associated to a fixed
    element $e \in E$ with $p(e) = b$, see~\eqref{rho_e};
  \item The upper vertical map is a homeomorphism compatible with the left
    $\Gamma$-actions and the right $G$-actions, which at the source is given by
    \[
    \gamma' \cdot \bigl( (\gamma,t),g\bigr) \cdot g' = \bigl( (\gamma'\gamma,t),gg'\bigr);
    \]

  \item The lower horizontal arrow is a homeomorphism compatible with the left
    $\Gamma$-actions, and $q$ sends $\bigl( (\gamma,t),g\bigr)$ to $(\gamma,t)$.
  \end{enumerate}
\end{theorem}
\begin{proof}
  Because of
  Theorem~\ref{the:bundles_and_Gamma_times_G-spaces}~\ref{the:bundles_and_Gamma_times_G-spaces:b_to_s}
  we can interpret $E_l$ as a $\Gamma \times G$-$CW$-complex.  Then the claim follows
  from the Slice Theorem~\ref{the:Slice_Theorem} and the proof of
   Theorem~\ref{the:bundles_and_Gamma_times_G-spaces}~\ref{the:bundles_and_Gamma_times_G-spaces:s_to_b}.
   More precisely, the desired diagram is the
  diagram of left $\Gamma \times G$-spaces coming from combining
  the diagrams~\eqref{diagram_for_local_object_for_Gamma_times_G} and~\eqref{diagram_for_local_object},
  if we replace $S$ by its image $T$ under $p \colon E \to B$. This is possible since the
  projection $\Gamma \times G \to \Gamma$ induces an isomorphism $\psi_e \colon (\Gamma
  \times G)_{e} \xrightarrow{\cong} \Gamma_b$ and $p|_S \colon S \to T$ is a
  $\psi_e$-homeomorphism.  
\end{proof}


\section{Homotopy invariance}
\label{sec:Homotopy_invariance}

Next we show that the pullback of a $\Gamma$-equivariant principal $G$-bundle with
$\Gamma$-homotopic maps yields isomorphic $\Gamma$-equivariant principal $G$-bundles.

\begin{theorem}[Homotopy invariance] \label{the:homotopy_invariance} Let $\calr$ be a
  family of local representations for $(\Gamma,G)$ satisfying Condition (H) introduced in
  Definition~\ref{def:property_(H)}. Let $B$ be a $\Gamma$-$CW$-complex and let $p \colon
  E \to B \times [0,1]$ be a $\Gamma$-equivariant principal $G$-bundle with $\calr(p)
  \subseteq \calr$.  Let $i_0 \colon B = B \times \{0\} \to B \times [0,1]$ be the
  inclusion.

  Then $i_0^*E \times [0,1] \xrightarrow{i_0^*p \times \id_{[0,1]}} B \times [0,1]$ is a
  $\Gamma$-equivariant principal $G$-bundle and there exists an isomorphism of
  $\Gamma$-equivariant principal $G$-bundles
  \[
  f \colon i_0^*E \times [0,1] \to E
  \]
  over $B \times [0,1]$ whose restriction to $B \times \{0\}$ is the identity.
\end{theorem}
\begin{proof}
  Let $p_n \colon E_n \to B_n$ be the restriction of $i_0^*E$ to the $n$-skeleton $B_n$ of
  $B$.  We will construct inductively over $n$ an isomorphism of $\Gamma$-equivariant
  principal $G$-bundles
  \[f_n \colon E_n \times [0,1] \xrightarrow{\cong} E|_{B_n \times [0,1]}
  \]
  such that the restriction of $f_n$ to $B_n \times \{0\}$ is the identity and the
  restriction of $f_n$ to $B_{n-1} \times [0,1]$ is $f_{n-1}$. Then we can define the
  desired isomorphism $f$ by requiring that $f|_{B_n \times [0,1]} = f_n$, since $i_0^*E$
  has the weak topology with respect to the filtration by the $E_n$-s
  by Lemma~\ref{lem:filtrations}~\ref{lem:filtrations:equivariant_versu_non_equivariant}  
  and~\ref{lem:filtrations:preimages}. 

  The induction beginning $n = -1$ is trivial, the induction step from $(n-1)$ to $n$ done
  as follows.  Choose a $\Gamma$-pushout
  \[
  \xymatrix{\coprod_{i \in I} \Gamma/H_i \times S^{n-1} \ar[r]^-{\coprod_{i \in I} q_i}
    \ar[d] & B_{n-1} \ar[d]
    \\
    \coprod_{i \in I} \Gamma/H_i \times D^n \ar[r]^-{\coprod_{i \in I} Q_i} & B_n }
  \]
  From Lemma~\ref{lem_pushouts-pullbacks}~\ref{lem_pushouts-pullbacks:equivariant_versus_non-equivariant}
  and~\ref{lem_pushouts-pullbacks:pullbacks}  we obtain a $\Gamma \times G$-pushout
  \[
  \xymatrix{\coprod_{i \in I} q_i^*E_{n-1} \ar[r]^-{\coprod_{i \in I} \overline{q_i}}
    \ar[d] & E_{n-1} \ar[d]
    \\
    \coprod_{i \in I} Q_i^*E_n \ar[r]^-{\coprod_{i \in I} \overline{Q_i}} & E_n }
  \]
  and thus a $\Gamma \times G$-pushout
  \[
  \xymatrix@!C= 11em{\coprod_{i \in I} q_i^*E_{n-1} \times [0,1] \ar[r]^-{\coprod_{i \in
        I} \overline{q_i} \times \id_{[0,1]}} \ar[d] & E_{n-1} \times [0,1] \ar[d]
    \\
    \coprod_{i \in I} Q_i^*E_n \times [0,1] \ar[r]^-{\coprod_{i \in I} \overline{Q_i}
      \times \id_{[0,1]}} & E_n \times [0,1] }
  \]
  Hence it suffices to extend for each $i \in I$ the map of $\Gamma$-equivariant principal
  $G$-bundles over $\Gamma/H_i \times S^{n-1} \times [0,1]$
  \[
  x_i \colon q_i^*E_{n-1} \times [0,1] \xrightarrow{\overline{q_i} \times [0,1]} E_{n-1}
  \times [0,1] \xrightarrow{f_{n-1}} E|_{B_{n-1} \times [0,1]}
  \]
  covering $q_i \times \id_{[0,1]} \colon \Gamma/H_i \times S^{n-1} \times [0,1] \to
  B_{n-1} \times [0,1]$ to a map of $\Gamma$-equivariant principal $G$-bundles over $\Gamma/H_i \times
  D^n \times [0,1]$
  \[
  y_i \colon Q_i^*E_n \times [0,1] \to E|_{B_n \times [0,1]}
  \]
  covering $Q_i \times \id_{[0,1]} \colon \Gamma/H_i \times D^n \times [0,1] \to B_n
  \times [0,1]$ such that the restriction of $y_i$ to $\Gamma /H_i \times D^n \times
  \{0\}$ is fiberwise the identity.  We obtain from $x_i$ a map of $\Gamma$-equivariant
  principal $G$-bundles over $\Gamma/H_i \times S^{n-1} \times [0,1]$
  \[
  x_i' \colon q_i^*E_{n-1} \times [0,1] \to (Q_i \times \id_{[0,1]})^*E|_{\Gamma/H_i
    \times S^{n-1} \times [0,1]}
  \]
  covering the identity $\id \colon \Gamma/H_i \times S^{n-1} \times [0,1] \to \Gamma/H_i
  \times S^{n-1} \times [0,1]$ such that the restriction of $x_i'$ to $\Gamma/H_i \times
  S^{n-1} \times \{0\}$ is the identity. It remains to extend $x_i'$ to map of
  $\Gamma$-equivariant principal $G$-bundles over $\Gamma/H_i \times D^n\times [0,1]$
  \[
  y_i' \colon Q_i^*E_n \times [0,1] \to (Q_i \times \id_{[0,1]})^*E
  \]
  covering the identity $\id \colon \Gamma/H_i \times D^n \times [0,1] \to \Gamma/H_i
  \times D^n \times [0,1]$ such that the restriction of $y_i'$ to $\Gamma/H_i \times D^n
  \times \{0\}$ is the identity.  We obtain from~\eqref{desired_Gamma_times_G-homeo}, now
  applied to $D^n \times [0,1]$ instead of $D^n$, isomorphisms of $\Gamma$-equivariant
  principal $G$-bundles over $\Gamma/H_i \times D^n \times [0,1] $
  \begin{eqnarray*}
    a \colon (\Gamma \times G)/H_i \times D^n \times [0,1] & \xrightarrow{\cong} & Q_i^*E_n \times [0,1];
    \\
    b \colon (\Gamma \times G)/H_i \times D^n \times [0,1]  & \xrightarrow{\cong} & (Q_i \times \id_{[0,1]})^*E,
  \end{eqnarray*}
  for an appropriate subgroup $H_i \subseteq \Gamma \times G$ if we convert the left
  $G$-action into a right $G$-action in the usual way. By conjugation with the
  restrictions of $a$ and $b^{-1}$ to $\Gamma/H_i \times S^{n-1} \times [0,1]$, we obtain
  from $x_i'$ an isomorphism of $\Gamma$-equivariant principal $G$-bundles over
  $\Gamma/H_i \times S^{n-1} \times [0,1]$
  \[
  x_i'' \colon (\Gamma \times G)/H_i \times S^{n-1} \times [0,1] \xrightarrow{\cong}
  (\Gamma \times G)/H_i \times S^{n-1} \times [0,1].
  \]
  The remaining problem is to extend $x_i''$ to an isomorphism of $\Gamma$-equivariant
  principal $G$-bundles over $\Gamma/H_i \times D^n \times [0,1]$
  \[
  y_i'' \colon (\Gamma \times G)/H_i \times D^n \times [0,1] 
  \xrightarrow{\cong} (\Gamma \times G)/H_i \times D^n \times [0,1]
  \]
  whose restriction to $D^n \times \{0\}$ is the restriction of $b^{-1} \circ a$ to $D^n
  \times \{0\}$.  Notice that $x_i''$ is the same as a map $S^{n-1} \times [0,1]\to
  \map_{\Gamma \times G}\bigl((\Gamma \times G)/H_i,(\Gamma \times G)/H_i\bigr)$ and
  $y_i''$ is the same as map $D^n\times[0,1] \to \map_{\Gamma \times G}\bigl((\Gamma
  \times G)/H_i,(\Gamma \times G)/H_i\bigr)$.  Hence the remaining problem is to extend a
  given map
  \[
  S^{n-1} \times [0,1] \cup D^n \times \{0\} \to \map_{\Gamma \times G}\bigl((\Gamma
  \times G)/H_i,(\Gamma \times G)/H_i\bigr)
  \]
  to a map
  \[
  D^n \times [0,1] \to \map_{\Gamma \times G}\bigl((\Gamma \times G)/H_i,(\Gamma \times
  G)/H_i\bigr).
  \]
  This is possible since there is a retraction $D^n \times [0,1] \to S^{n-1} \times [0,1]
  \cup D^n \times \{0\}$.
\end{proof}


\section{Universal equivariant principal bundles}
\label{sec:Universal_equivariant_principal_bundles}

Fix a family of local representations $\calr$ satisfying Condition (H) introduced in
Definition~\ref{def:property_(H)}.  In this section we construct the universal
$\Gamma$-equivariant principal $G$-bundle and in particular the classifying space for
$\Gamma$-equivariant principal $G$-bundles with respect to $\calr$.

\begin{definition}[Compatibility] \label{def:compatibility} We call $\calr$
  \emph{compatible with the $\Gamma$-$CW$-complex $X$} if for any $x \in X$ and
  $(H,\alpha) \in \calr$ with $\Gamma_x \subseteq H$ the pair
  $(\Gamma_x,\alpha|_{\Gamma_x})$ belongs to $\calr$ again.
\end{definition}

\begin{remark}\label{rem:automatic}
  Notice that this condition is automatically satisfied for every $\Gamma$-$CW$ if $\calr$
  is closed under taking subgroups, i.e., for $(H,\alpha)$ in $\calr$ and $K \subseteq H$
  we have $(K,\alpha|_K) \in \calr$.
\end{remark}

Consider $\Gamma$-$CW$-complexes $X$ and $B$ and a $\Gamma$-equivariant principal
$G$-bundle $p \colon E \to B$ over the $\Gamma$-$CW$-complex $B$ with $\calr(p) \subseteq
\calr$.  Let $\Bundle_{\Gamma,G,\calr}(X)$ be the set of isomorphism classes of
$\Gamma$-equivariant principal $G$-bundles $q \colon E \to X$ with $\calr(q) \subseteq
\calr$.  Suppose that $\calr$ is compatible with $X$. Then for any $\Gamma$-map $f \colon
X \to B$ the pull back $f^*p$ is a $\Gamma$-equivariant principal $G$-bundle with
$\calr(f^*p) \subseteq \calr$. Because of Theorem~\ref{the:homotopy_invariance} the
pullback construction yields a well-defined map
\begin{eqnarray}
  c \colon [X,B]^{\Gamma} \to \Bundle_{\Gamma,G,\calr}(X), \quad [f] \mapsto [f^*p].
  \label{c_[X,B]_to_Bundles}
\end{eqnarray}

\begin{notation}[Classifying space of a family of subgroups]
  \label{not:Classifying_space_of_a_family_of_subgroups}
  Given group $G$ and a family $\calf$ of subgroups of  $G$, denote by
  $\EGF{G}{\calf}$ the \emph{classifying space of the family} $\calf$.
\end{notation}

Recall that $\EGF{G}{\calf}$ is a $G$-$CW$-complex whose isotropy groups belong to $\calf$
and for which the $H$-fixed point set $\EGF{G}{\calf}^H$ is non-empty and weakly
contractible for every $H \in \calf$.  A model always exists.  For any $G$-$CW$-complex
$X$ whose isotropy groups belong to $\calf$, there is up to $G$-homotopy precisely one
$G$-map $X \to \EGF{G}{\calf}$. In particular two models for $\EGF{G}{\calf}$ 
are $G$-homotopy equivalent. For
more information about classifying spaces of a family we refer for instance
to~\cite{Lueck(2005s)}.

\begin{theorem}[Classifying space for $\Gamma$-equivariant principal $G$-bundles with
  family of local representatives contained in $\calr$ satisfying Condition (H)]
  \label{the:Classifying_space_for_gamma-equivariant_principal_G-bundles}
  Let $\calr$ be a family of local representations for $(\Gamma,G)$ satisfying Condition
  (H) introduced in Definition~\ref{def:property_(H)}. Define
  \begin{eqnarray*}
    E(\Gamma,G,\calr) & := & E_{\calf(\calr)}(\Gamma \times G)_r;
    \\
    B(\Gamma,G,\calr) & := & E_{\calf(\calr)}(\Gamma \times G)_r/G,
  \end{eqnarray*}
  where $\calf(\calr)$ is the family of subgroups of $\Gamma \times G$ introduced in
  Lemma~\ref{lem:family_associated_to_calr}. Let 
  $p \colon E(\Gamma,G,\calr) \to B(\Gamma,G,\calr)$ be the projection.  Let $X$
  be a $\Gamma$-$CW$-complex such that $\calr$ is compatible with $X$ in the sense of
  Definition~\ref{def:compatibility}.

  Then $p \colon E(\Gamma,G,\calr) \to B(\Gamma,G,\calr)$ is a $\Gamma$-equivariant
  principal $G$-bundle and the map $c$ defined in~\eqref{c_[X,B]_to_Bundles}
  \begin{eqnarray*}
    c \colon [X,B]^{\Gamma} \to \Bundle_{\Gamma,G,\calr}(X), \quad [f] \mapsto [f^*p].
  \end{eqnarray*}
  is bijective.

\end{theorem}
\begin{proof}
  We conclude from Theorem~\ref{the:bundles_and_Gamma_times_G-spaces}~%
\ref{the:bundles_and_Gamma_times_G-spaces:s_to_b} that $p \colon E(\Gamma,G,\calr) \to
  B(\Gamma,G,\calr)$ is a $\Gamma$-equivariant principal $G$-bundle.

  We construct an inverse map
  \begin{eqnarray*}
    d \colon \Bundle_{\Gamma,G,\calr}(X)  \to [X,B]^{\Gamma} 
  \end{eqnarray*}
  as follows.

  Let $q \colon E \to B$ be a $\Gamma$-equivariant principal $G$-bundle whose family of
  local representations is contained in $\calr$. We conclude from
  Theorem~\ref{the:bundles_and_Gamma_times_G-spaces}~\ref{the:bundles_and_Gamma_times_G-spaces:b_to_s}
  that $E_l$ is a left $\Gamma \times G$-$CW$-complex whose isotropy groups belong to
  $\calf(\calr)$.  Hence there is up to $\Gamma \times G$-homotopy precisely one
  $\Gamma\times G$-map $\overline{f} \colon E_l \to E_{\calf(\calr)}(\Gamma \times G)$
  which is the same a map $\overline{f} \colon E \to E(\Gamma,G,\calr)$ compatible with
  the left $\Gamma$-actions and right $G$-actions. Taking the $G$-quotient yields a map $f
  \colon B \to B(\Gamma,G,\calr)$ which is unique up to $\Gamma$-homotopy and for which
  the following diagram commutes
  \[
  \xymatrix{E \ar[r]^-{\overline{f}} \ar[d]_q & E(\Gamma,G,\calr) \ar[d]^p
    \\
    B \ar[r]_-{f} & B(\Gamma,G,\calr) }
  \]
  Hence $q$ is up to isomorphism the pullback of the universal bundle $p$ with the
  $\Gamma$-map $f \colon B \to B(\Gamma,G,\calr)$ and the construction of this $\Gamma$-map 
  $f$ is unique up to $\Gamma$-homotopy.  Define $d([q])$ by $[f]$. This is well-defined since
  obviously $[f]$ depends only on the isomorphism class $[q]$ of $q$.  Since $f^*p$ is
  isomorphic to $q$, we get $c \circ d = \id_{\Bundle_{\Gamma,G,\calr}(X)}$.  One easily
  checks $d \circ c = \id_{ [X,B]^{\Gamma}}$.

  This finishes the proof of
  Theorem~\ref{the:Classifying_space_for_gamma-equivariant_principal_G-bundles}.
\end{proof}

We will call $p \colon E(\Gamma,G,\calr) \to B(\Gamma,G,\calr)$ appearing in
Theorem~\ref{the:Classifying_space_for_gamma-equivariant_principal_G-bundles} the
\emph{universal $\Gamma$-equivariant principal $G$-bundle with respect to the family of
  local representations $\calr$}.

If $p' \colon E(\gamma,G,\calr)' \to B(\gamma,G,\calr)'$ is another such universal bundle,
then there is a commutative diagram of such bundles covering a $\Gamma$-homotopy
equivalence $f$
\[
\xymatrix{E(\Gamma,G,\calr) \ar[r]^{\overline{f}} \ar[d]^{p} & E(\Gamma,G,\calr)'
  \ar[d]^{p'}
  \\
  B(\Gamma,G,\calr) \ar[r]^f & B(\Gamma,G,\calr)' }
\]
i.e., $\overline{f}$ is compatible with the left $\Gamma$ and the right $G$-actions and is a
homeomorphism.  Such a diagram is unique up to $\Gamma \times G$-homotopy, in particular
the $\Gamma$-map $f$ is unique up to $\Gamma$-homotopy.


\section{Reduction of the structure group to a maximal compact subgroup}
\label{sec:Reduction_of_the_structure_group_to_a_maximal_compact_subgroup}

If $K$ is a subgroup of $G$ and $p_0 \colon E_0 \to B$ is a $\Gamma$-equivariant principal $K$-bundle, then
\[
\overline{p_0} \colon E_0 \times_K G \to B, \quad (e,g) \mapsto p_0(e)
\]
inherits the structure of a $\Gamma$-equivariant principal $G$-bundle in the obvious
way. One has just to verify local triviality of the underlying principal $G$-bundle
$\overline{p_0}$.

\begin{definition}[Reduction]
  Let $p \colon E \to B$ be a $\Gamma$-equivariant principal $G$-bundle. Given a subgroup
  $K \subseteq G$, a \emph{reduction} to a $\Gamma$-equivariant principal $K$-bundle is a
  $\Gamma$-equivariant principal $K$-bundle $p_0 \colon E_0 \to B$ together with an
  isomorphism of $\Gamma$-equivariant principal $G$-bundles from 
  $\overline{p_0}$ to $p$.
\end{definition}

\begin{lemma} \label{lem:criterion_for_reduction} 
   Let $p \colon E \to B$ be a
  $\Gamma$-equivariant principal $G$-bundle and $K \subseteq G$ be a subgroup.  Let 
  $f   \colon E \to G/K$ be a $\Gamma \times G$-map, where we equip the target with the 
  $\Gamma \times G$-action given by $(\gamma,g) \cdot g'K = gg'K$. Suppose that the pair $(G,K)$
  satisfies Condition (S), see Definition~\ref{ref:condition_(S)}. Put $E_f =  f^{-1}(1K)$. 
   Let $p_f \colon E_f \to B$ be the restriction of $p$ to $E_f$. Then:

  \begin{enumerate} 
  \item  \label{lem:criterion_for_reduction:is_reduction} 
  The map $p_f \colon E_f \to B$ is a $\Gamma$-equivariant principal $K$-bundle which is a
 $K$-reduction of $p$;

\item \label{lem:criterion_for_reduction:every_reduction_arises_this_way} Every
  $K$-reduction of $p \colon E \to B$ is up to isomorphism of $\Gamma$-equivariant
  principal $K$-bundles of the form $p_f \colon E_f \to B$ for appropriate $f$;

\item  \label{lem:criterion_for_reduction:homotopy_invariance} 
If $f_0,f_1 \colon E \to G/K$ are
$\Gamma \times G$-maps which are $\Gamma \times G$-homotopic, then
$p_{f_0} \colon E_{f_0} \to B$ and  $p_{f_1} \colon E_{f_1} \to B$ are isomorphic as  
$\Gamma$-equivariant principal $K$-bundles.
\end{enumerate}
\end{lemma}
\begin{proof}~\ref{lem:criterion_for_reduction:is_reduction}  
From Lemma~\ref{lem:condition_(S)} we obtain a $\Gamma \times G$-homeomorphism
\[
u \colon E_0 \times_K G \xrightarrow{\cong} E, \quad (g,e) \mapsto e \cdot g.
\]
It remains to show that $p_0 \colon E_0 \to B$ is a principal $K$-bundle, i.e., to show
local triviality. Since $p$ is locally trivial, it suffices to treat the case, where $E = G \times B$ 
and  $p \colon G \times B \to B$ is the projection. Choose a quasi-regular open subset 
$U \subseteq G/K$ with $1K \in U$ and a map $s \colon U \to G$ satisfying 
$\pr \circ s = \id_U$, where $\pr \colon G \to G/K$ is the projection.
Let $V$ be the  open subset of $B$ given by  $p(f^{-1}(U))$. 
Since $V$ is $\Gamma$-invariant, it is quasi-regular by 
Lemma~\ref{lem:properties_of_quasi-regular_open_sets}~\ref{lem:properties_of_quasi-regular_open_sets:Gamma-CW}.
Let $\alpha \colon V \to G$ be the map sending $v$ to $s \circ f(v,1)$. 
Then we obtain an automorphism of the trivial principal $G$-bundle $G \times B \to G \times B$
\[
\overline{\alpha} \colon G \times V \xrightarrow{\cong} G \times V, 
\quad (g,v) \mapsto (g \cdot \alpha(v),v)
\]
If $\pi \colon G \times V \to G/K$ sends $(g,v)$ to $gK$,
then $\pi \circ \overline{\alpha} = f$. Hence $\overline{\alpha}$ induces a
commutative diagram of $K$-spaces
\[
\xymatrix{p_f^{-1}(V)  \ar[rd]_{p_f|_{p_f^{-1}(V)}} \ar[rr]^{\cong} 
& &
V \times K \ar[ld]
\\
& V & 
}
\]
such that the horizontal arrow is a $K$-homeomorphism and the right vertical arrow is the
projection.  
\\[1mm]~\ref{lem:criterion_for_reduction:every_reduction_arises_this_way} 
Let $p_0 \colon E_0 \to B$ be a $K$-reduction of $p$. Choose an isomorphism of
$\Gamma$-equivariant principal $G$-bundles $u \colon G \times_K E_0 \xrightarrow{\cong} E$.  
If we take $f \colon E \to G/K$ to be the composite of $u^{-1}$ with the obvious
projection $G \times_K E_0 \to G/K$, then $p_0$ is $p_f$.
\\[2mm]~\ref{lem:criterion_for_reduction:homotopy_invariance} 
Fix a $\Gamma \times G$-homotopy $h \colon E \times [0,1] \to G/K$ with $h_k = f_k$ for $k = 0,1$.  
Let $q \colon E \times [0,1] \to B \times [0,1]$ be the $\Gamma$-equivariant principal
$G$-bundle given by $p \times \id_{[0,1]}$. Then we obtain a $\Gamma$-equivariant principal
$K$-bundle $q_h \colon (E \times [0,1])_h \to B \times [0,1]$ whose restriction to $B \times \{k\}$ 
is $p_{f_k}$ for $k = 0,1$. Now the claim follows from
Theorem~\ref{the:homotopy_invariance}.
\end{proof}

\begin{definition}[Almost connected group] \label{def:almost_connected} 
  Given a group $G$,
  let $G^0$ be its \emph{component of the identity} and define the \emph{component group}
  $\overline{G}$ by $\overline{G} = G/G^0$.  We call $G$ \emph{almost connected}
   if its component group $\overline{G}$ is compact.
\end{definition}

The next result is due to Abels~\cite[Corollary~4.14]{Abels(1978)}.  
\begin{theorem}[Almost connected groups] \label{the:almost_connected_groups} 
  Let $G$ be a  locally compact Hausdorff topological group.  Suppose that $G$ is almost connected.

  Then  $G$ contains a maximal compact subgroup $K$ which is unique up to conjugation, and the
  $H$-fixed point set $(G/K)^H$ is contractible for every compact subgroup $H \subseteq   K$.
\end{theorem}

\begin{theorem}[Existence of a $K$-reduction]
  \label{the:reduction_to_maximal_compact_subgroup}
  Let $G$ and $\Gamma$ be locally compact second countable topological groups. Suppose
  that $G$ is almost connected. Let $K \subseteq G$ be a maximal compact subgroup.  Let
  $\calr$ be a family of local representation for $(\Gamma,G)$ such that for each element
  $(H,\alpha)$ in $\calr$ the subgroup $H \subseteq \Gamma$ is compact. Let $\calr_K$ be
  the family of local representations of $(\Gamma,K)$ which is given by 
  $\calr_K =  \{(H,\alpha) \mid \alpha(H) \subseteq K, (H,\alpha) \in \calr\}$.

  Then every $\Gamma$-equivariant principal $G$-bundle $p \colon E \to B$ with $\calr(p)
  \subseteq \calr$ has a (preferred) $K$-reduction $p_0 \colon E_0 \to B$ which is unique
  up to isomorphism of $\Gamma$-equivariant principal $K$-bundles and satisfies $\calr(p_0)
  \subseteq \calr_K$.
\end{theorem}
\begin{proof}
  Condition $(H)$  is always satisfied under the conditions of 
  Lemma~\ref{the:reduction_to_maximal_compact_subgroup} because of
  Theorem~\ref{the:criterion_for_(H)}. It suffices to prove the claim for the universal
  $\Gamma$-equivariant principal $G$-bundle because of 
  Theorem~\ref{the:Classifying_space_for_gamma-equivariant_principal_G-bundles}, 
  since $K$-reductions are compatible with pullbacks. 

  Let $p \colon E(\Gamma,G,\calr) \to B(\Gamma,G,\calr)$ be the universal $\Gamma$-equivariant
   principal $G$-bundle with respect to the family $\calr$, see 
  Theorem~\ref{the:Classifying_space_for_gamma-equivariant_principal_G-bundles}. 
  Recall that  $E(\Gamma,G,\calr)$ is a $\Gamma$-$G$-$CW$-complex whose isotropy groups belong
  to the family of subgroups $\calf(\calr)$ associated to $\calr$. For any element in $K(H,\alpha) \in \calf$
  the image of $K(H,\alpha)$ under the projection $\Gamma \times G \to G$ is compact and 
  hence a closed subgroup  of $G$. Lemma~\ref{lem:quotient_of_CW-complexes} implies that 
  $E(\Gamma,G,\calr)/\Gamma$ is a $G$-$CW$-complex.
  Because of Theorem~\ref{the:almost_connected_groups}  the $G$-$CW$-complex $G/K$ is 
  the classifying space for proper $G$-actions. Hence there is up to $G$-homotopy precisely one map
  $E(\Gamma,G,\calr)/\Gamma \to G/K$. We conclude that there is up to $\Gamma\times G$-homotopy precisely
  one $\Gamma \times G$-map $f \colon  E(\Gamma,G,\calr) \to G/K$. From 
  Lemma~\ref{lem:criterion_for_reduction}   we obtain a $K$-reduction $p_f$ of $p$ which is unique 
  up to isomorphism of  $\Gamma$-equivariant principal $K$-bundles.   
  One easily checks $\calr(p_f) \subseteq \calr_K$.
\end{proof}

\begin{example}[Equivariant vector bundles and Riemannian metrics]
In the case $G = GL_n(\IR)$, we have as maximal compact subgroup $O(n)$
and Theorem~\ref{the:reduction_to_maximal_compact_subgroup}  implies the well-known the statement
that any equivariant vector bundle $\xi$ over a proper $\Gamma$-$CW$-complex can be equipped with 
a $\Gamma$-invariant Riemannian metric and that two equivariant vector bundles
over a proper $\Gamma$-$CW$-complex with $\Gamma$-invariant Riemannian metrics
admit an isomorphism respecting the $\Gamma$-invariant Riemannian  metrics if and only
the equivariant vector bundles are isomorphic (after forgetting the invariant Riemannian metrics).
\end{example}


\section{On the homotopy type of the classifying space}
\label{sec:On_the_homotopy_type_of_the_classifying_space}

In this section we want to establish for $H \subseteq \Gamma$ a weak homotopy equivalence
\begin{eqnarray*}
\bigsqcup_{[\alpha] \in \hom_\calr(H,G)/G} BC_G(\alpha)  & \simeq & B(\Gamma,G,\calr)^H,
\end{eqnarray*}
where $[\alpha]$ runs over the $G$-conjugacy classes of homomorphisms $\alpha \colon H \to G$
with $(H,\alpha) \in \calr$. This will follow from Theorem~\ref{the:Fixed_point_sets_of_B(Gamma,G,calr)}.

Let $\hom(H,G)$ be the space of homomorphisms of topological groups $H \to G$, endowed with
the subspace topology from $\hom(H,G) \subseteq \map(H;G)$,
see Subsection~\ref{subsec:Space_of_homomorphisms}. Denote by $\hom(H,G)/G$ the
quotient space under the conjugation action of $G$, i.e., the left $G$-action sending
$(g,\alpha) \in G \times \hom(H,G)$ to $c_g \circ \alpha$, where $c_g \colon G \to G$
sends $g'$ to $gg'g^{-1}$, see   Subsection~\ref{subsec:Space_of_homomorphisms}. 
Recall that the centralizer of $\alpha \in \hom(H,G)$ is
\[
C_G(\alpha) := \{g \in G \mid g\alpha(h)g^{-1} =\alpha(h) \; \text{for all}\; h \in H\}.
\]
For $\calr$ a family of local representations for $(\Gamma,G)$, and $(H,\alpha)$ in
$\calr$ define
\[
\hom_{\calr}(H,G) := \{ \alpha \in \hom(H,G)| (H,\alpha) \in \calr \}
\]
and note that $\hom_{\calr}(H,G)$ is closed under the conjugation of $G$.

\begin{theorem}[Fixed point sets of $B(\Gamma,G,\calr)$]
  \label{the:Fixed_point_sets_of_B(Gamma,G,calr)}
  Let $\calr$ be a family of local representations for $(\Gamma,G)$ satisfying Condition
  (H) introduced in Definition~\ref{def:property_(H)}.  Consider an element $(H,\alpha)$
  in $\calr$.
  \begin{enumerate}

  \item \label{the:Fixed_point_sets_of_B(Gamma,G,calr):pi_0} We obtain a bijection
    \[\hom_\calr (H,G)/G \xrightarrow{\cong} \pi_0\bigl(B(\Gamma,G,\calr)^H\bigr);
    \]
  \item \label{the:Fixed_point_sets_of_B(Gamma,G,calr):BC_Galpha} Given $(H,\alpha)$ in
    $\calr$, let $B(\Gamma,G,\calr)^H_{\alpha}$ be the path component of
    $B(\Gamma,G,\calr)^H$ that corresponds under the bijection of
    assertion~\ref{the:Fixed_point_sets_of_B(Gamma,G,calr):pi_0} to the class of $\alpha$
    in $\hom_\calr (H,G)/G$.

    Then there exists a weak homotopy equivalence
    \[
    BC_G(\alpha) \xrightarrow{\simeq} B(\Gamma,G,\calr)^H_{\alpha}.
    \]
  \end{enumerate}
\end{theorem}
\begin{proof}~\ref{the:Fixed_point_sets_of_B(Gamma,G,calr):pi_0} We obtain from
  Theorem~\ref{the:Classifying_space_for_gamma-equivariant_principal_G-bundles} bijections
  \[
  \pi_0\bigl(B(\Gamma,G,\calr)^H\bigr) \xrightarrow{\cong}
  [\Gamma/H,B(\Gamma,G,\calr)]^{\Gamma} \xrightarrow{\cong}
  \Bundle_{\Gamma,G,\calr}(\Gamma/H).
  \]
  Given $\alpha \in \hom_\calr (H,G)$, we obtain a $\Gamma$-equivariant principal
  $G$-bundle
  \[
  \xi_{\alpha} \colon \Gamma \times_{\alpha} G \to \Gamma/H, \quad (\gamma,g) \mapsto
  \gamma H,
  \]
  where $\Gamma \times_{\alpha} G$ is the quotient of $\Gamma \times G$ under the left
  $H$-action given by $h \cdot (\gamma,g) = (\gamma\cdot h^{-1}, \alpha(h) \cdot g)$.  We
  conclude from Lemma~\ref{lem:equivariant_principal_G-bundles_over_Gamma/H_times_Z}~%
\ref{lem:equivariant_principal_G-bundles_over_Gamma/H_times_Z:given_by_rho}
  and~\ref{lem:equivariant_principal_G-bundles_over_Gamma/H_times_Z:uniqueness} that any
  $\Gamma$-equivariant principal $G$-bundle $q \colon E \to\Gamma/H$ with $\calr(q) \subset
  \calr$ is isomorphic to $\xi_{\alpha}$ for some $\alpha \in \hom_\calr (H,G)$, and for
  two elements $\alpha,\beta \in \hom_\calr(H,G)$ the $\Gamma$-equivariant principal
  $G$-bundles $\xi_{\alpha}$ and $\xi_{\beta}$ are isomorphic if and only if the classes
  of $\alpha$ and $\beta$ in $\hom_\calr (H,G)/G$ agree. Hence we obtain a bijection
  \[\hom_\calr (H,G)/G \xrightarrow{\cong} \Bundle_{\Gamma,G,\calr}(\Gamma/H), \quad
  [\alpha] \mapsto [\xi_{\alpha}].
  \]
  \\[2mm]~\ref{the:Fixed_point_sets_of_B(Gamma,G,calr):BC_Galpha} Let $p \colon
  E(\Gamma,G,\calr) \to B(\Gamma,G,\calr)$ be the canonical projection which is a
  principal $G$-bundle.  We want to show that  $p$ induces a principal $C_G(\alpha)$-bundle
  \[
  p_{(H,\alpha)} \colon E(\Gamma,G,\calr)^{K(H,\alpha)} \to B(\Gamma,G,\calr)^H_{\alpha}.
  \]

  Abbreviate $E = E(\Gamma,G,\calr)$. Consider $(H,\alpha) \in \calr$. Let $E^{\langle H \rangle}$
  be the subspace of $E$ consisting of those elements $e \in E$ such that for each 
  $h \in H$ there exists $g \in G$ with $h \cdot e = e \cdot g$.  For each $e \in E$ define
  $\rho_e \in \hom(H,G)$ by requiring $h \cdot e = e \cdot \rho_e(h)$.  Thus we can define
  a map of sets
  \[
  \rho \colon E^{\langle H \rangle} \to \hom(H,G).
  \]
  Next we show that $\rho$ is continuous.  We apply Theorem~\ref{the:local_objects} at the
  point $b = p(e)$ and obtain a commutative diagram
  \[
  \xymatrix{ \Gamma \times_{\Gamma_b} (T \times G)\ar[r]^-{f}_-{\cong} \ar[d]_{q} &
    p^{-1}(V) \ar[d]^{p|_{p^{-1}(V)}}
    \\
    \Gamma \times _{\Gamma_b} T \ar[r]^-{u}_-{\cong} & V }
  \]
  It suffices to show that the composite of $\rho$ with the map $f^{\langle H \rangle}
  \colon \bigl(\Gamma \times_{\Gamma_b} (T \times G)\bigr)^{\langle H \rangle} \to
  E^{\langle H \rangle}$ is continuous.

  Put $\Gamma^{\langle H \rangle} = \{\gamma \in \Gamma \mid \gamma^{-1} \cdot H \cdot
  \gamma \subseteq \Gamma_b\}$.  
  Consider the commutative diagram
  \[
  \xymatrix{\Gamma^{\langle H \rangle} \ar[r]^-{\nu'} \ar@{^{(}->}[d] & \hom(H,\Gamma_b)
    \ar@{^{(}->}[d]
    \\
    \Gamma \ar[r]_-{\nu} & \hom(H,\Gamma) }
  \]
  where $\nu'$ and $\nu$ respectively are given by conjugating the inclusion homomorphism
  $i\colon H \to \Gamma_b$ with $\gamma \in \Gamma^{\langle H \rangle}$ and $\gamma \in \Gamma$
  respectively, the left vertical arrow is the inclusion of subgroups and the right vertical
  arrow is the injection induced by the inclusion $\Gamma_b \to \Gamma$. Since
  $\map(H,\Gamma_b)$ is the preimage of the constant map under $\map(H,\Gamma) \to \map(H,\Gamma/\Gamma_b)$,
  we conclude from Subsection~\ref{lem:hom(H,G)_is_closed_in_map(H,G)} that the right vertical arrow is the inclusion
  of a closed subspace. The conjugation map $\Gamma \times \hom(H,\Gamma) \to \hom(H,\Gamma)$ 
  is continuous as explained in  Subsection~\ref{lem:hom(H,G)_is_closed_in_map(H,G)},
  We conclude that $\nu$ and hence $\nu'$ are  continuous.  Define a map of sets
  \[
  \mu \colon \Gamma^{\langle H \rangle} \times G \to \hom(H,G)
  \]
  by sending $(\gamma, g)$ to the composite $H \xrightarrow{c_{\gamma} \circ i} \Gamma_b
  \xrightarrow{\rho_e} G \xrightarrow{c_g} G$, where $i \colon H \to \Gamma_b$ is the
  inclusion.  Since $\nu'$ is continuous, $\mu$ is continuous. The map $\mu$
  factorizes through the obvious projection $\Gamma^{\langle H \rangle} \times G \to
  \bigl(\Gamma \times_{\Gamma_b} G\bigr)^{\langle H \rangle}$, which is an identification,
  to a continuous map $\overline{\mu} \colon \bigl(\Gamma \times_{\Gamma_b}
  G\bigr)^{\langle H \rangle} \to \hom(H,G)$ making the following diagram commutative
  \[
  \xymatrix{\Gamma^{\langle H \rangle}  \times G \ar[d] \ar[rd]^{\mu} & \\
    \bigl(\Gamma \times_{\Gamma_b} G\bigr)^{\langle H \rangle} \ar[r]^{\overline{\mu}} &
    \hom(H,G)}
  \]

  If $\pr \colon \Gamma \times_{\Gamma_b} (T \times G) \to \Gamma \times _{\Gamma_b} G$ is
  the projection,  the following diagram is commutative
  \[
  \xymatrix{\bigl(\Gamma \times_{\Gamma_b} (T \times G)\big)^{\langle H \rangle}
  \ar[r]^-{f^{\langle H \rangle}} \ar[d]^{\pr^{\langle H \rangle}}
   & E^{\langle H \rangle} \ar[d]^\rho 
   \\
    \bigl(\Gamma \times _{\Gamma_b} G\bigr)^{\langle H \rangle} \ar[r]^{\overline{\mu}} 
    & 
   \hom(H,G)
    }
   \] 
   Hence  $\overline{\mu} \circ \pr^{\langle H \rangle} = \rho \circ f^{\langle H \rangle}$.  This
  implies that $ \rho \circ f^{\langle H \rangle}$ and hence $\rho \colon E^{\langle H
    \rangle} \to \hom(H,G)$ is continuous. One easily checks that $\rho$ is a $G$-map.

  Let $E^{\langle H,\alpha \rangle}$ be the preimage under $\rho$ of the orbit $G \cdot
  \alpha \subseteq \hom(H,G)$ with respect to the action of $G$ on $\hom(H,G)$ given by
  composing with conjugation automorphisms.  Hence $\rho$ induces a $G$-map
  \[
  \rho_{\alpha}' \colon E^{\langle H,\alpha \rangle} \to G \cdot \alpha.
  \]
  Since $\calr$ satisfies Condition (H), see Definition~\ref{def:property_(H)}, the $G$-map
  $\iota_{\alpha} \colon G/C_G(\alpha) \to G \cdot \alpha, g \mapsto c_g \circ \alpha$ is
  a homeomorphism.  Define the $G$-map
  \[
  \rho_{\alpha} = \iota_{\alpha}^{-1} \circ \rho_{\alpha}'  \colon E^{\langle H,\alpha \rangle} 
  \to G/C_G(\alpha).
  \]
  One easily checks that the preimage of $1\cdot C_G(\alpha) \in G/C_G(\alpha)$ under
  $\rho_{\alpha}$ is $E^{K(H,\alpha)}$.  Since $(G,C_G(\alpha))$ satisfies Condition (S),
  it implies that the canonical $G$-map
  \[
  G \times_{C_G(\alpha)} E^{K(H,\alpha)} \to E^{\langle H,\alpha \rangle}
  \]
  is a $G$-homeomorphism. One easily checks that $E^{\langle H,\alpha \rangle}$ is the
  preimage of $B(\Gamma,G,\calr)^H_{\alpha}$ under $p \colon E(\Gamma,G,\calr) \to
  B(\Gamma,G,\calr)$.  Hence $p$ induces a $C_G(\alpha)$-map 
  \[
  p_{(H,\alpha)} \colon   E(\Gamma,G,\calr)^{K(H,\alpha)} \to B(\Gamma,G,\calr)^H_{\alpha}
  \]
   for which the following diagram  of $G$-spaces commutes
  \[
  \xymatrix@!C = 10em{G \times_{C_G(\alpha)} E^{K(H,\alpha)} \ar[rr]^{\cong} \ar[rd]_-{G \times_{G_{C_G(\alpha)}} p_{(H,\alpha)}}
  & &  E^{\langle H,\alpha \rangle} \ar[ld]^{\quad p|_{p^{-1}(B(\Gamma,G,\calr)^H_{\alpha})}}
  \\
  & B(\Gamma,G,\calr)^H_{\alpha} &
  }
  \]
  Since the right vertical arrow is a principal $G$-bundle, Lemma~\ref{lem:criterion_for_reduction} 
  implies that 
  $p_{(H,\alpha)} \colon E(\Gamma,G,\calr)^{K(H,\alpha)} \to B(\Gamma,G,\calr)^H_{\alpha}$ 
  is a principal $C_G(\alpha)$-bundle.

  Let $X$ be a $CW$-complex and $f \colon X \to B(\Gamma,G,\calr)^H_{\alpha}$ be a weak
  homotopy equivalence.  Let $q \colon f^* E(\Gamma,G,\calr)^{K(H,\alpha)} \to X$ be the
  pullback of the principal $C_G(\alpha)$-bundle $p_{(H,\alpha)}$. Every principal
  $G$-bundle is a fibration. We conclude from the long exact sequence of homotopy groups
  associated to $p_{(H,\alpha)}$ and $q$ and the Five-Lemma that the induced map $f^*
  E(\Gamma,G,\calr)^{K(H,\alpha)} \to E(\Gamma,G,\calr)^{K(H,\alpha)}$ is a weak homotopy
  equivalence. Since $E(\Gamma,G,\calr)^{K(H,\alpha)}$ is weakly contractible, the same is
  true for $f^* E(\Gamma,G,\calr)^{K(H,\alpha)}$. Since $q$ is a principal $G$-bundle over
  the $CW$-complex $X$ with weakly contractible total space, it is a model for the
  universal principal $C_G(\alpha)$-bundle and $X$ is a model for $BC_G(\alpha)$.
\end{proof}


\section{Examples}
\label{sec:examples}

\subsection{Some special cases for $G$, $\Gamma$, and $\calr$}
\label{subsec:Some_special_cases_for_G_Gamma_calr}

\begin{example}[Trivial $\Gamma$] \label{exa:trivial_Gamma} Suppose $\Gamma$ is
  trivial. Then there is only one family $\calr$ of local representations consisting of
  $\{1\} \to G$, and the universal $\Gamma$-equivariant principal $G$-bundle with respect
  to the family of local representations $\calr$ is the same as the universal principal
  $G$-bundle $EG \to BG$.
\end{example}

\begin{example}[Trivial $G$] \label{exa:trivial_G} Suppose $G$ is trivial. Then a family
  of local representations is the same as a family $\calf$ of subgroups of $\Gamma$, and
  the universal $\Gamma$-equivariant principal $G$-bundle with respect to the family of
  local representations $\calr = \calf$ is the identity $\EGF{\Gamma}{\calf} \to
  \EGF{\Gamma}{\calf}$.
\end{example}

\begin{example}[Trivial local representations]
  \label{exa:trivial_local_representations}
  Let $\calf$ be a family of subgroups of $\Gamma$. Let $\caltr(\calf)$ be the system of
  local representations given by pairs $(H,\rho_H)$, where $H$ belongs to $\calf$ and
  $\rho_H$ is the trivial representation. Then a $\Gamma$-equivariant principal $G$-bundle
  $q \colon E \to B$ satisfies $\calr(p) \subseteq \caltr(\calf)$ if and only if all
  isotropy groups of $B$ belong to $\calf$, the induced map $q/\Gamma \colon E/\Gamma \to
  B/\Gamma$ is a principal $G$-bundle, and $q$ is the pullback of $q/\Gamma$ with the
  projection $\pr \colon B \to B/\Gamma$. This follows from
  Theorem~\ref{the:local_objects}.

  A model for the universal $\Gamma$-equivariant principal $G$-bundle associated to
  $\caltr(\calf)$ is
  \[
  \id \times p \colon \EGF{\Gamma}{\calf} \times EG \to \EGF{\Gamma}{\calf} \times BG,
  \]
  where $p \colon EG \to BG$ is the universal principal $G$-bundle, the left $\Gamma$- and
  right $G$-actions are given on the total space by $\gamma \cdot (x,e) \cdot g = (\gamma
  \cdot x, e\cdot g)$ and on the base space by $\gamma \cdot (x,b) \cdot g = (\gamma \cdot
  x, b)$.
\end{example}

\begin{example}[$\calr(\calf)$] \label{exa:proper_spaces} 
  Consider a family $\calf$ of
  subgroups of $\Gamma$.  Define the associated family of local representations
  \[
  \calr(\calf) := \{(H,\alpha) \mid H \in \calf\; \text{and} \; \alpha \colon H \to G \;
  \text{any group homomorphism}\},
  \]
  in other words, $H$ runs through elements in $\calf$ and $\alpha$ through all possible
  group homomorphisms.

  Then a $\Gamma$-equivariant principle $G$-bundle $p \colon E \to B$ over the
  $\Gamma$-$CW$-complex $B$ satisfies $\calr(p) \subseteq \calr(\calf)$ if and only if all
  isotropy groups of $B$ belong to $\calf$.  If we suppose that $\calr(\calf)$ satisfies
  Condition (H) introduced in Definition~\ref{def:property_(H)}, then in this situation
  the universal $\Gamma$-equivariant principal $G$-bundle with respect to $\calr(\calf)$
  classifies $\Gamma$-equivariant principal $G$-bundles over $\Gamma$-$CW$-complexes whose
  isotropy groups belong to $\calf$.
  
  If we choose $\calf$ to be the family of compact subgroups and $\calr(\calf)$ satisfies
  Condition (H), these are precisely the $\Gamma$-equivariant principal $G$-bundles over
  proper $\Gamma$-$CW$-complexes.
\end{example}

\begin{example}[$n$-dimensional complex vector bundles with Hermitian metric over proper $\Gamma$-$CW$-complexes]
  \label{exa:Gamma_equivariant_U(n)_bundles} 
  Let $\Gamma$  be a topological group. Assume that all compact subgroups of $\Gamma$ are Lie groups and
  $\Gamma$ is completely regular, e.g., $\Gamma$ is a Kac-Moody group, or that $\Gamma$ is 
   locally compact, second countable and has finite
    covering dimension, e.g., is a  Lie groups.  Let $G = \calu(n)$ be the
  Lie group of unitary automorphisms of $\mathbb{C}^n$. Consider the family of   representations
  \[
  \calr := \{(H,\alpha) \mid H \; \text{is compact} \; \text{and} \; \alpha \colon H \to G \;
  \text{any group homomorphism}\},
  \]
   which is the family $\calr(\calCOM)$ of Example~\ref{exa:proper_spaces}  associated to the family $\calCOM$ of
  compact subgroups. 

  In this case $ \hom(H,\calu(n))$ is isomorphic to the space of unitary representations of $H$ on 
  $\mathbb{C}^n$, and $\hom (H,\calu(n))/\calu(n)$ is isomorphic to the set of isomorphism classes of
  $n$-dimensional unitary $H$-representations. Denote by $V_1,V_2, \ldots,V_k, \ldots$ the
  irreducible unitary representations of $H$ and  $d_i:= \dim_{\mathbb{C}}V_i$. The 
  set of isomorphism classes unitary $n$-dimensional $H$-representations can be
  parametrized with the set of partitions of $n$ using the dimensions $d_i$, i.e.,
  \[ \hom(H,\calu(n))/\calu(n) \cong \{(n_1, n_2, \ldots, n_k,\ldots) \mid n_1d_1+ \cdots +n_kd_k+ \cdots =n, n_i\geq 0\}.
  \]

  If the homomorphism $\alpha\colon H \to \calu(n)$ induces the representation $V=
  \bigoplus_{i}V_i^{\oplus n_i}$, then the isotropy group
  \[
  C_{\calu(n)}(\alpha) \cong \prod_{i} \calu(n_i)
  \]
  and therefore we have that
  \[
  B(\Gamma,\calu(n),\calr)^H \cong \bigsqcup_{ \{(n_1, n_2, \ldots, n_k,\ldots) \mid n_1d_1+ \cdots +n_kd_k+ \cdots =n, n_i\geq 0\}}\; \prod_i B\calu(n_i).
  \]
\end{example}


\subsection{Compact abelian Lie group $G$}
\label{subsec:Compact_abelian_Lie_group_G}

  Let $\Gamma$  be a topological group. Assume that all compact subgroups of $\Gamma$ are Lie groups and
  $\Gamma$ is completely regular, e.g., $\Gamma$ is a Kac-Moody group, or that $\Gamma$ is 
   locally compact, second countable and has finite
    covering dimension, e.g., is a  Lie groups. Let $G$ be a compact abelian Lie group.

  Consider the family of   representations
  \[
  \calr := \{(H,\alpha) \mid H \; \text{is compact} \; \text{and} \; \alpha \colon H \to G \;
  \text{any group homomorphism}\},
  \]
   which is the family $\calr(\calCOM)$ of Example~\ref{exa:proper_spaces}  associated to the family $\calCOM$ of
  compact subgroups.
  We  have 
  \[
  \hom_\calr(H,G) = \hom(H,G)
  \]
  and therefore we obtain from Theorem~\ref{the:Fixed_point_sets_of_B(Gamma,G,calr)} 
  a weak homotopy equivalence
  \[
  \bigsqcup_{\hom(H,G)} BG \simeq B(\Gamma,G,\calr)^H.
  \]
  Whenever $G=S^1$ we have that $\hom(H,S^1)\cong H^2(BH,\mathbb{Z})$ and $BS^1 \simeq
  K(\mathbb{Z},2)$. Therefore we obtain a  weak homotopy equivalence
  \[
  \bigsqcup_{H^2(BH,\mathbb{Z})} K(\mathbb{Z},2) \simeq  B(\Gamma,S^1,\calr)^H.
  \]

Now make the stronger assumption that $\Gamma$ is a Lie group.
Then for every compact subgroup $H \subseteq \Gamma$
the homogeneous space $\Gamma/H$ is 
a smooth manifold and hence a $CW$-complex. 
This implies that $E \Gamma \times_\Gamma B(\Gamma,G,\calr)$ is a $CW$-complex.
Because of Theorem~\ref{the:criterion_for_(H)} and
Theorem~\ref{the:Classifying_space_for_gamma-equivariant_principal_G-bundles}
we have the universal $\Gamma$-equivariant $G$-bundle with respect to $\calr$
\[
p \colon E(\Gamma,G,\calr) \to B(\Gamma,G,\calr)
\]
Applying the homotopy quotient with respect to the group $\Gamma$ we obtain a principal $G$-bundle
over a $CW$-complex 
\[
G \to E \Gamma \times_\Gamma E(\Gamma,G,\calr) \to E \Gamma \times_\Gamma
B(\Gamma,G,\calr)
\]
which can be classified by a map
\[
E \Gamma \times_\Gamma B(\Gamma,G,\calr) \to BG.
\]
 This map induces an adjoint map
\[
\psi\colon B(\Gamma,G,\calr) \to \map(E \Gamma, BG)
\]
which is $\Gamma$-equivariant. In general the map $\psi$ is not a $\Gamma$-equivariant
homotopy equivalence, but under the specific choices of $\Gamma$, $G$ and $\calr$ above indeed
they are.

\begin{theorem} \label{the:compact_abelian_G} 
   Let $\Gamma$ be a Lie group and let $G$
  be a compact abelian Lie group.  Put $\calr = \calr(\calCOM)$ for $\calCOM$ the family
  of compact subgroups of $\Gamma$.  Let $X$ be a proper $\Gamma$-CW complex.

     Then the map
    \[
    \map(\id_X,\psi) \colon \map (X, B(\Gamma,G,\calr)) \to \map  (X,\map(E\Gamma, BG)), 
    \quad f \mapsto \psi \circ f
    \]
    is a weak $\Gamma$-homotopy equivalence. 

    In particular we obtain bijections
    \[
     \Bundle_{\Gamma,G,\calr}(X) \cong  [X, B(\Gamma,G,\calr))]^{\Gamma}  \xrightarrow{\cong}  [X,\map(E\Gamma, BG)]^{\Gamma} = [E\Gamma \times_{\Gamma} X,BG],
     \]
     and, when $G=S^1$
    \[
    \Bundle_{\Gamma,S^1,\calr}(X) \cong H^2(E\Gamma \times_\Gamma X, \mathbb{Z}).
    \]
  \end{theorem}
    \begin{proof}
    For $H$ a compact subgroup of $\Gamma$, we claim that the induced map
    \[
    \psi^H\colon B(\Gamma,G,\calr)^H \to \map(E\Gamma, BG)^H \simeq \map(BH, BG)
    \]
    is a weak homotopy equivalence; its proof is based on the proofs and results
    of~\cite[Theorem~2 and Proposition4]{Lashof-May-Segal(1983)}.  From
    Theorem~\ref{the:Fixed_point_sets_of_B(Gamma,G,calr)} we know that
    \[
    B(\Gamma,G,\calr)^H \cong \bigsqcup_{\alpha \in \hom(H,G)} B(\Gamma,G,\calr)^H_\alpha,
    \]
    and since $G$ is abelian we know that 
    \[
    E(\Gamma, G, \calr)^{K(H,\alpha)} \to B(\Gamma,G,\calr)^H_\alpha
    \]
    is a $G$-principal bundle. Therefore we have a commutative square of principal $G$-bundles
    \[\xymatrix{E(\Gamma, G, \calr)^{K(H,\alpha)} \ar[r] \ar[d] & EG \ar[d] \\
    B(\Gamma,G,\calr)^H_\alpha \ar[r]^{\simeq} & BG
    }
    \]
    where the bottom arrow is a homotopy equivalence. This square induces another square
    of associated principal $G$-bundles
    \begin{align}\xymatrix{E \Gamma \times_{K(H,\alpha)} E(\Gamma, G, \calr)^{K(H,\alpha)} \ar[r] \ar[d] 
    & E \Gamma \times_{H,\alpha}EG \ar[d]  \ar[r]
    & EG \ar[d]
   \\
    E \Gamma/H \times B(\Gamma,G,\calr)^H_\alpha \ar[r]^{\simeq} 
    & E \Gamma/ H \times BG \ar[r] 
    & BG
    } \label{fig:classifying_map_for_associated_bundle_of_B_alpha}
    \end{align} 
    where on the upper left corner the group $K(H,\alpha)$ acts on $E \Gamma$ via the
    canonical isomorphism $H \cong K(H,\alpha)$, on the upper middle term the group $H$
    acts on $EG$ via the homomorphism $\alpha\colon  H \to G$ and the horizontal maps on the
    right hand side are classifying maps.  In~\cite[Proof of Theorem~2, page~173]{Lashof-May-Segal(1983)} 
    it is shown that the adjoint map of the lower horizontal
    maps of diagram~\eqref{fig:classifying_map_for_associated_bundle_of_B_alpha}
    \[
    BG \to \map(E \Gamma, EG)^H \simeq \map(BH,BG)
    \]  
    is a weak equivalence on basepoint components; therefore the map
    \[
    \psi^H|_{B(\Gamma,G,\calr)^H_\alpha}\colon B(\Gamma,G,\calr)^H_\alpha \to \map(E\Gamma, BG)^H \simeq \map(BH, BG)
    \]
    is also a weak equivalence of basepoint components. In~\cite[Proposition~4]{Lashof-May-Segal(1983)} 
    it is shown that the map
    \[
    B\colon  \hom(H,G) \to [BH,BG] 
    \]
    given by the classifying space functor is an isomorphism, and since the bundle 
   $E   \Gamma \times_{H,\alpha}EG$ is constructed through the action defined by the
    homomorphism $\alpha$, then we conclude that the map
    \[
    \psi^H\colon B(\Gamma,G,\calr)^H \to \map(E\Gamma, BG)^H \simeq \map(BH, BG)
    \]
    is indeed a weak homotopy equivalence.
    
    Hence for every proper $\Gamma$-$CW$-complex $Y$ the induced map
    \begin{eqnarray}
    & \psi_* \colon [Y,B(\Gamma,G,\calr)]^{\Gamma} \to [Y,\map(E \Gamma, BG)]^{\Gamma}, \quad [f] \mapsto [\psi \circ f]& 
    \label{psi_ast}
  \end{eqnarray}
  is bijective, see~\cite[Proposition~2.3 on page~35]{Lueck(1989)}. 

    In order to show that $\map(\id_X,\psi)$ is a weak $\Gamma$-homotopy equivalence, it suffices to show
    for any $\Gamma$-$CW$-complex $Z$ that the induced map
    \begin{multline*}
    \map(\id_X,\psi)_*  \colon [Z,\map (X, B(\Gamma,G,\calr))]^{\Gamma} \to [Z,\map  (X,\map(E\Gamma, BG))]^{\Gamma}, 
    \\ 
    [f] \mapsto [\map(\id_X,\psi) \circ f]
    \end{multline*}
    is bijective, see~\cite[Proposition~2.3 on page~35]{Lueck(1989)}. 
   Because of the adjunctions appearing in 
   Subsections~\ref{subsec:Basic_feature_of_the_category_of_compactly_generated_spaces}
  and the identification $\pi_0(\map(A,B)^{\Gamma}) = [A,B]^{\Gamma}$ for $\Gamma$-spaces $A$ and $B$,
   this is equivalent to showing that adjoint map 
    \[
   \psi_* \colon [X \times Z,B(\Gamma,G,\calr)]^{\Gamma} \to [X \times Z,\map(E\Gamma, BG)]^{\Gamma} ,
    \quad [f] \mapsto [\psi \circ f]
  \]
   is bijective, where $\Gamma$ acts diagonally on $X \times Z$. Since $\Gamma$ is a Lie group and $X$ is a proper
   $\Gamma$-$CW$-complex, $X \times Y$ is $\Gamma$-homotopy equivalent to a proper $\Gamma$-$CW$-complex $Y$.
     Hence  $\map(\id_X,\psi)$ is a weak $\Gamma$-homotopy equivalence because of the bijectivity of~\eqref{psi_ast}.

    The other claims follow using Theorem~\ref{the:Fixed_point_sets_of_B(Gamma,G,calr)}.
  \end{proof}


\section{The case $G=\calp \calu(\calh)$ the projective unitary group. }
\label{Example:The_case_G_is_PU(H)_the_projective_unitary_group}

Twisted versions of K-theory may be defined via a specific type of
projective unitary bundles. The key point is that the space $\Fred(\calh)$ of Fredholm
operators on a separable Hilbert space $\calh$ endowed with the norm topology, which
itself has the homotopy type of $\mathbb{Z} \times BU$~\cite{Atiyah(1969b),Jaenich(1965)},
carries a conjugation action by the group $\calp \calu(\calh)$ of projective unitary
operators. Thus, to a pair $(X,P)$ of a CW-complex
$X$ together with a principal $\calp \calu(\calh)$-bundle
\[
\calp \calu(\calh) \to P \to X
\]
over $X$, one can associate the \emph{twisted K-theory groups} $K^{-i}(X,P)$ 
(see~\cite{Atiyah-Segal(2004)}) defined as the homotopy groups
\[
K^{-i}(X,P) := \pi_i\bigl( \Gamma(P \times_{\calp \calu(\calh)} \Fred(\calh))\bigr)
\]
of the space of sections of the associated $\Fred(\calh)$-bundle
\[
\Fred(\calh) \to P \times_{\calp \calu(\calh)} \Fred(\calh) \to X.
\]

The equivariant version of the previous construction requires equivariant projective
unitary bundles of a certain kind, and in order to construct their universal and
classifying space we need to show that the group $\calp \calu(\calh)$ satisfies
items~\ref{def:property_(H):component},~\ref{def:property_(H):S_for_G}
and~\ref{def:property_(H):homeo} of Condition (H) introduced in
Definition~\ref{def:property_(H)}. In what follows, we will show that $\calp \calu(\calh)$
satisfies items~\ref{def:property_(H):component},~\ref{def:property_(H):S_for_G}
and~\ref{def:property_(H):homeo} of Condition (H), whenever we consider homomorphisms
$\alpha \colon H \to \calp \calu(\calh)$ from finite groups $H$.


\subsection{Existence of local cross  sections for $H$ finite}
\label{subsec:Existence_of_local_cross_sections_for_H_finite}

Let $\calu(\calh)$ and $\calp \calu(\calh)$ be respectively the unitary and projective
unitary groups of a separable Hilbert space $\calh$. The group $\calu(\calh)$ is defined
as
\[
\calu(\calh):= \{U \in \calb(\calh) \mid U U^*=U^*U=1 \}
\]
where $\calb(\calh)$ denotes the space of bounded operators on $\calh$, its center
$Z(\calu(\calh))$ is $S^1$, and $\calp \calu(\calh)$ is the quotient $ \calu(\calh) /S^1$.
We endow $ \calu(\calh)$ with the norm topology, i.e., a sub-base for the topology is
given by the sets
\[
B_\epsilon(T) :=\{ S \in \calu(\calh) \mid \| S-T \| < \epsilon \}
\]
where \[
\|T \|:= \sup \{ \|Tx\|  \mid x \in \calh \text{ such that } \|x\|  \leq 1  \}.
\]

Endow $\calp \calu(\calh)$ with the quotient topology and note this topology can be
recovered with the metric defined by the distance between the $S^1$-orbits, i.e., for $T,U
\in \calp \calu(\calh)$ define
\[
(T,U)_{\min}:= \min \{ \| \widetilde{T} - \widetilde{U} \| 
\mid \widetilde{T} , \widetilde{U} \in \calu(\calh) \text{ lifts of } T,U \text{ respectively} \}.
\]

With these topologies the groups $ \calu(\calh)$ and $\calp \calu(\calh)$ become a topological groups.
The short exact sequence of topological groups
\[
1 \to S^1 \to \calu(\calh) \stackrel{p}{\to} \calp \calu(\calh) \to 1
\]
is  a $S^1$-principal bundle (cf.~\cite{Simms(1970)}).

The group $\calu(\calh)$ endowed with the norm topology is moreover a Banach Lie 
group\footnote{A reference for the foundations of Banach manifolds may be found on~\cite{Lang(1972)}, and 
a reference on the properties the properties of Banach Lie groups and their Banach Lie algebras may 
be found on~\cite{de_la_Harpe(1972), Neeb(2006)} and the references therein.}; 
namely, $\calu(\calh)$ is a Banach manifold
whose structural maps are maps of Banach manifolds, see for instance~\cite[Example~V.1.6, page~391]{Neeb(2006)}. It
can be modeled locally by the vector space $ \cala$ of skew-adjoint operators
\[
\cala := \{ L \in \calb(\calh) \mid L +L^*=0 \};
\]
via the exponential map
\[ \exp\colon  \cala \to \calu(\calh), \quad L \mapsto \exp(L);
\]
in this way we could think of the skew-adjoint operators as the tangent space of
$\calu(\calh)$ at the identity, $T_1\calu(\calh) = \cala$.  

This in particular implies that a base of open sets around $1 \in \calu(\calh)$ may be
obtained by the image of the exponential map of a base of open sets around $0 \in \cala$. 
Denoting by $i \mathbb{R}$ the operators of the form $r\sqrt{-1} \cdot \id_\calh$
for $r \in \mathbb{R}$, we obtain the diagram
\begin{align} \label{diagram_tangent_space_U(H)}
  \xymatrix{0 \ar[r] & i \mathbb{R} \ar[r] \ar[d]^\exp & \cala \ar[r]^\pi \ar[d]^\exp & \cala/i\mathbb{R} \ar[r] &0 
   \\
    1 \ar[r] & S^1 \ar[r] & \calu(\calh) \ar[r]^p & \calp \calu (\calh) \ar[r] & 1 \\
  }\end{align} 
which in particular permits us to model locally $\calp \calu(\calh)$ through any section
\[
\xymatrix{0 \ar[r] & i \mathbb{R} \ar[r]  & \cala \ar[r]_\pi  & \cala/i\mathbb{R}  \ar@/_1pc/[l]_\sigma \ar[r] &0 
}
\]
 by the composition 
 \[
 p \circ \exp \circ \sigma \colon  \cala / i \mathbb{R} \to \calp \calu (\calh)
 \]
 making the following diagram commutative
\[
\xymatrix{\cala \ar[r]_\pi \ar[d]^\exp & \cala/i\mathbb{R} \ar@/_1pc/[l]_\sigma \ar[d]^{p \circ \exp \circ \sigma}   
\\
 \calu(\calh) \ar[r]^p & \calp \calu (\calh).  
}
\]

Now let $H$ be a finite group. Recall that $\hom(H,\calp \calu(\calh))$ obtains the
subspace topology from its embedding into $\map(H,\calp \calu(\calh))$, see
Subsection~\ref{subsec:Space_of_homomorphisms}; since $\calp \calu(\calh)$ is metric, this
topology can also be defined with the supremums metric of~\eqref{supremums_metric}.

Consider the conjugation action
\begin{align*}
\calp \calu(\calh) \times \hom(H,\calp \calu(\calh)) & \to \hom(H,\calp \calu(\calh)) \\
(g, \alpha) & \mapsto g \alpha g^{-1} 
\end{align*}
and for $\alpha \in  \hom(H,\calp \calu(\calh)) $ denote by 
\[
\calp \calu(\calh) \cdot \alpha := \{ g \alpha g^{-1} \mid g \in \calp \calu(\calh) \}
\]
 the orbit of $\alpha$ under the conjugation action. We claim

\begin{theorem} \label{the:local_cross_sections_for_Hom(H,PU(H))}
For any finite group $H$ and any $\alpha \in  \hom(H,\calp \calu(\calh)) $, the projection
\[
\pr \colon \calp\calu(\calh) \to \calp\calu(H)/C_{\calp\calu(\calh))}(\alpha)
\]
is a principal $C_{\calp\calu(\calh)}(\alpha)$-bundle and the canonical map
\[\iota_{\alpha} \colon  \calp\calu(H)/C_{\calp\calu(\calh)}(\alpha) \xrightarrow{\cong}  \calp\calu(H) \cdot \alpha
\]
is a homeomorphism.
\end{theorem}

Its proof needs some preparation. 
For $\alpha \colon H \to \calp \calu(\calh)$ denote by 
\[
\widetilde{H}:=\alpha^*\calu(\calh)
\]
the central $S^1$-extension of $H$ defined by the 
pullback of $\alpha$ making the following diagram commutative
\begin{align} \label{diagram_tilde(H)_wolfgang}
\xymatrix{1  \ar[r] & S^1 \ar[d]^= \ar[r] & \widetilde{H} \ar[d]^{\widetilde{\alpha}} \ar[r]^\gamma & H \ar[r] \ar[d]^\alpha & 1\\
1 \ar[r] & S^1 \ar[r] & \calu (\calh) \ar[r] & \calp \calu(\calh) \ar[r] & 1,
}\end{align}

Denote the space of homomorphisms from $\widetilde{H}$ to $\calu(\calh)$ on which the 
kernel of the map $\widetilde{H} \stackrel{\gamma}{\to} H$ acts by multiplication
\[
\hom_{S^1}(\widetilde{H},\calu(\calh)) 
:= \{ f \in \hom(\widetilde{H},\calu(\calh)) \mid f(x)=x \;\text{for all}\;  x \in \ker(\gamma) \},
\]
and endow it with subspace topology 
$\hom_{S^1}(\widetilde{H},\calu(\calh)) \subseteq \hom(\widetilde{H},\calu(\calh))$.
(This is automatically compactly generated by the next result.)

\begin{lemma}
The space $\hom_{S^1}(\widetilde{H},\calu(\calh))$ is closed and open in  $ \hom(\widetilde{H},\calu(\calh))$.
\end{lemma}

\begin{proof}
Consider the restriction map
\begin{align*}
A:\hom(\widetilde{H},\calu(\calh)) & \to \hom(S^1,\calu(\calh))\\
f  &\mapsto  f|_{\ker(\gamma)}
\end{align*}
and note that $\hom_{S^1}(\widetilde{H},\calu(\calh))= A^{-1}(\rho)$ where
$\rho(\lambda)(x) =\lambda \cdot x$ for any $x \in \calh$.

Take any other $\beta \in \hom(S^1,\calu(\calh))$ and since it defines a representation of
$S^1$ different from $\rho$, there must exists an element $x \in \calh$ of norm $1$
and an integer $k \neq 1$ such that $\beta(\lambda)x= \lambda^k x$. We have then
\[
\|\rho, \beta \|_{\sup} \geq \sup_{\lambda \in S^1} \| \rho(\lambda)x - \beta(\lambda)x \| =\sup_{\lambda \in S^1} | 1 - \lambda^{k-1} | >1
\]
and therefore we can conclude that $\rho$ is an isolated point in
$\hom(S^1,\calu(\calh))$. Hence $A^{-1}(\rho)$ is closed and open.
\end{proof}

Define also the space of homomorphisms from $H$ to $\calp \calu(\calh)$ which induce
isomorphic central $S^1$-extensions of $H$
\[
\hom(H,\calp \calu(\calh))_{\widetilde{H}}  := \{ \beta  \colon H \to\calp \calu(\calh) 
\mid \beta^*\calu(\calh)\cong \widetilde{H} \;\text{as}\;  \text{central}\; S^1\text{-extensions of}\; H\}
\]
and endow it with the subspace topology of $\hom(H,\calp \calu(\calh))$. 
(This is automatically compactly generated by the next result.)

\begin{lemma} \label{lem:hom(H,PU(H))_wilde(H)_open_closed}
The space $\hom(H,\calp \calu(\calh))_{\widetilde{H}}$ is closed and open in  $ \hom(H,\calp \calu(\calh))$.
\end{lemma}

\begin{proof}
  The isomorphic classes of central extensions of a finite group are in one to one
  correspondence with the group $H^2(H,S^1)$.  This cohomology group can be calculated by
  the continuous cohomology~\cite[Example 1.4]{Segal(1970cohtopgr)} as the quotient of
  the topological groups
\[
\frac{\ker\bigl(\map(H^2, S^1) \xrightarrow{\delta} \map(H^3,S^1)\bigr)}{\im\bigl(\map(H,S^1) 
\xrightarrow{\delta} \map(H^2,S^1)\bigr)}
\]
where $\delta$ denotes the standard group cohomology differential. For a homomorphism
$\alpha \in \hom(H,\calp \calu(\calh))$ define the 2-cocycle $\widetilde{c}_\alpha \in
\map(H^2,S^1)$
\[
\widetilde{c}_\alpha (g,h) := \widetilde{\alpha(g)}\widetilde{\alpha(h)}\widetilde{\alpha(gh)}^{-1}
\]
where $ \widetilde{\alpha(g)}\in \calu(\calh)$ denotes fixed lifts of $\alpha(g)$ for all $g \in H$. If 
\[
\overline{c}_\alpha (g,h) := \overline{\alpha(g)} \ \overline{\alpha(h)} \ \overline{\alpha(gh)}^{-1}
\]
denotes the 2-cocycle defined by a different choice of lifts, then, since
$\overline{\alpha(g)} \widetilde{\alpha(g)}^{-1}$ is in the center of $\calu(\calh)$ for
all $g \in H$, the map $e \in \map(H,S^1)$
\[
e(g):= \overline{\alpha(g)} \widetilde{\alpha(g)}^{-1}
\]
satisfies the equation $\delta(e) \cdot \widetilde{c}_\alpha=  \overline{c}_\alpha$. Therefore we have the equality
cohomology classes $[\widetilde{c}_\alpha]=[\overline{c}_\alpha]$. Thus we can define a map
\begin{align*}
 \rho:\hom(H,\calp \calu(\calh)) & \to H^2(H,S^1)\\
 \alpha & \mapsto [\widetilde{c}_\alpha]
\end{align*}
Next we prove that it is continuous. Since $H$ is finite and $\calu(\calh) \to \calp
\calu(\calh)$ is a principal $S^1$-bundle, we can define a small neighborhood $V$ of
$\alpha$ in $ \hom(H,\calp \calu(\calh))$ such that for every $h$ there exists a
continuous map $V \to \calu, \; \beta \mapsto \widetilde{\beta(h)}$ whose value for 
$\beta = \alpha$ is the given lift $\widetilde{\alpha(h)}$ and whose composite with the
projection $\calu(H) \to \calp\calu(H)$ is the evaluation map $\beta \mapsto \beta(h)$.
Now we can define a continuous map
\begin{align*}
V & \to \map(H^2,S^1)_0:= \ker(\map(H^2, S^1) \stackrel{\delta}{\to} \map(H^3,S^1))\\
\beta & \mapsto \widetilde{c}_\beta:= \widetilde{\beta(g)}\widetilde{\beta(h)}\widetilde{\beta(gh)}^{-1}.
\end{align*}
Since the quotient map 
\[
\map(H^2,S^1)_0 \to \map(H^2,S^1)_0 / Im(\delta) = H^2(H,S^1)
\]
is continuous, and $H^2(H,S^1)$ is finite, we conclude that the map $\rho$ is continuous.

Any 2-cocycle $\widetilde{c}_\alpha$ defines a $S^1$-central extension
$H\times_{\widetilde{c}_\alpha} S^1$ of $H$, and it is classical result in group theory
that $\widetilde{H}$ and $H\times_{\widetilde{c}_\alpha} S^1$ are isomorphic as central
extensions of $H$.

Therefore $\hom(H,\calp \calu(\calh))_{\widetilde{H}}=\rho^{-1}([\widetilde{c}_\alpha])$
and hence it is open and closed in $ \hom(H,\calp \calu(\calh))$.

\end{proof}

Before we proceed with the study of the canonical map $ \phi\colon \hom_{S^1}(\widetilde{H},\calu(\calh)) \to \hom(H,\calp \calu(\calh))_{\widetilde{H}}$, we will make use of Banach Lie group structure of $\calu(\calh)$ to prove the following lemma.

\begin{lemma} \label{lem:U(H)_n_local_homeo_PU(H)_n} For $n >1$ let 
  $\calu(\calh)_{n}:=  \{U \in \calu(\calh) \colon U^{n}=1\}$ and $\calp\calu(\calh)_{n}:= \{U \in   \calp\calu(\calh) \mid  U^{n}=1 \}$. 
  Then the projection map
  \[
  p|_{\calu(\calh)_n}\colon  \calu(\calh)_n \to \calp \calu(\calh)_n
  \]
  is a local homeomorphism.
\end{lemma}
\begin{proof}
The subspace $\calu(\calh)_{n}$ is a Banach submanifold of $\calu(\calh)$ since it is
the inverse image $f^{-1}(\{1\}) =\calu(\calh)_{n} $ for the function
\[f\colon  \calu(\calh) \to \calu(\calh), \quad U \mapsto U^n,
\]
that is analytical. Take $U \in \calu(\calh)_{n}$ and consider the isomorphism of tangent spaces
\[
R_{U^{-1}} \colon  T_U \calu(\calh) \stackrel{\cong}{\to}T_1 \calu(\calh), \quad  L \mapsto L U^{-1}.
\]
Since the tangent space of $U \in \calu(\calh)_{n}$ at $U$ is 
\[
T_U \calu(\calh)_n = \{ L \in \calb(\calh) \mid LU^*+UL^*=0 \text{ and } \sum_{k=0}^{n-1} U^k L U^{n-k-1}=0 \},
\]
its image under $R_{U^{-1}}$ becomes
\[
R_{U^{-1}} T_U \calu(\calh)_n = \{ A + A^*=0 \text{ and } \sum_{k=0}^{n-1} U^k A U^{n-k}=0 \}.
\]
The vector space $i \mathbb{R}$ of diagram~\eqref{diagram_tangent_space_U(H)} is not
included in $R_{U^{-1}} T_U \calu(\calh)_n$ since for $A= r\sqrt{-1} $ with $r \neq 0$ the
sum $\sum_{k=0}^{n-1} U^k A U^{n-k}= nr \sqrt{-1} \neq 0$. Therefore the vector space
$R_{U^{-1}} T_U \calu(\calh)_n$ maps isomorphically to $\pi(R_{U^{-1}} T_U
\calu(\calh)_n)$ under the map $\pi$ of diagram~\eqref{diagram_tangent_space_U(H)}.
By the inverse function theorem for Banach manifolds~\cite[Theorem~1~\S I.5, page~13]{Lang(1972)} the map $p$ restricted to
$\calu(\calh)_n \cdot U^{-1}$ has a local inverse around $1$, and hence the map
$p|_{\calu(\calh)_n}\colon  \calu(\calh)_n \to \calp \calu(\calh)_n$ has a local inverse around
a neighborhood of $U$ in $\calu(\calh)_n$. 
\end{proof}

\begin{proposition} \label{pro:hom(c,U(H))_to_hom(C,PU(H))_principal_bundle}
The canonical map
\begin{align*}
\phi\colon \hom_{S^1}(\widetilde{H},\calu(\calh)) & \to \hom(H,\calp \calu(\calh))_{\widetilde{H}}\\
f \colon \widetilde{H}\to \calu(\calh) & \mapsto f/S^1 \colon \widetilde{H}/S^1 \to \calu(\calh)/S^1
 \end{align*}
is a principal $\hom(H,S^1)$-bundle,  and since $\hom(H,S^1)$ is finite, it is also a local homeomorphism.
\end{proposition}
\begin{proof} 
  The map $\phi\colon \hom_{S^1}(\widetilde{H},\calu(\calh))  \to \hom(H,\calp   \calu(\calh))_{\widetilde{H}}$ 
  is continuous because of   Lemma~\ref{lem:quotient_map_is_continuous}. One easily checks that it  is surjective.

Define the $\hom(H,S^1)$-action on  $\hom_{S^1}(\widetilde{H},\calu(\calh))$ by the map
\[
\hom(H,S^1) \times \hom_{S^1}(\widetilde{H},\calu(\calh))  \to \hom_{S^1}(\widetilde{H},\calu(\calh)),
\quad (r,f)  \mapsto r \cdot f,
\]
where $(r \cdot f)(\widetilde{h}) := r(\gamma(\widetilde{h})) \cdot f(\widetilde{h})$ for
all $\widetilde{h} \in \widetilde{H}$.  One easily checks that the action is free and that
$\phi$ and $\phi(f)=\phi(r\cdot f)$ holds for all $r \in \hom(H,S^1)$ and 
$f \in \hom_{S^1}(\widetilde{H},\calu(\calh))$. Hence $\phi$ induces a surjective map
\[
\overline{\phi} \colon \hom_{S^1}(\widetilde{H},\calu(\calh))/\hom(H,S^1)  
\to \hom_{S^1}(\widetilde{H},\calu(\calh))
\]
Next we show that $\overline{\phi}$ is injective and hence $\phi$ is locally bijective. 
Consider $f,g \in\hom_{S^1}(\widetilde{H},\calu(\calh))$ with $\phi(f)=\phi(g)$. For all $h \in H$,
take any lift $\widetilde{h} \in \widetilde{H}$ such that $\gamma(\widetilde{h})=h$ and
define the element $r(h)\in S^1$ by the equation $f(\widetilde{h})= r(h) g(\widetilde{h})$.
Note that $r(h)$ does not depend on the choice of lift and therefore we obtain a map 
$r \colon H \to S^1$. For $h,k \in H$ with respective lifts $\widetilde{h}, \widetilde{k} \in
\widetilde{H}$ one has
\[ r(hk)  g(\widetilde{h} \widetilde{k}) = f(\widetilde{h} \widetilde{k})
= r(h) g(\widetilde{h} ) r(k) g(\widetilde{k})  = r(h)r(k) g(\widetilde{h} \widetilde{k}),
\]
and therefore we see that $r \in \hom(H,S^1)$ and $r \cdot f = g$.

We are left with showing that $\phi$ is a local homeomorphism.

For $h \in H$ denote by $|h|$ its order and by $\langle h \rangle$ the cyclic group it
generates. Since all central $S^1$-extensions of cyclic groups are trivializable,  the
restriction of $\widetilde{H}$ to $\langle h \rangle$ must be isomorphic to $S^1 \times
\langle h \rangle$; therefore we can choose a lift $\widetilde{h} \in \widetilde{H}$ for
each $h \in H$ such that $\widetilde{h}$ generates a cyclic group of order ${|h|}$. 

 Define the restriction maps
\begin{align*}
\Psi\colon  \hom_{S^1}(\widetilde{H},\calu(\calh)) 
& \to \prod_{h \in H} \hom(\langle \widetilde{h} \rangle, \calu(\calh)),
&
\widetilde{\alpha}  \mapsto \prod_{h \in H} \widetilde{\alpha}|_{\langle \widetilde{h} \rangle};
\\
\psi \colon \hom(H, \calp\calu(\calh))_{\widetilde{H}}  
& \to \prod_{h \in H} \hom(\langle h \rangle, \calp\calu(\calh)),
& 
\alpha \mapsto \prod_{h \in H}\alpha |_{\langle h \rangle},
\end{align*}
and note that both maps $\Psi$ and $\psi$ are injective. We claim
moreover that the map $\Psi$ induces a homeomorphism onto its
image. The proof is as follows: for $\epsilon >0$ and
$\widetilde{\alpha} \in \hom_{S^1}(\widetilde{H},\calu(\calh))$,
consider the open ball of radius $\epsilon$ defined by the supremum
metric of~\eqref{supremums_metric}
 \[
B_\epsilon(\widetilde{\alpha}):=\{ f \in  \hom_{S^1}(\widetilde{H},\calu(\calh)) \mid 
\|\widetilde{\alpha}(g) - f(g)\| < \epsilon \;\text{for all} \; g \in \widetilde{H} \},
\]
and note that this open ball can also be defined as 
\[
B_\epsilon(\widetilde{\alpha}):=\{ f \in  \hom_{S^1}(\widetilde{H},\calu(\calh)) \mid 
\|\widetilde{\alpha}(\widetilde{h}) - f(\widetilde{h})\| < \epsilon \;\text{for all}\;  h \in H \},
\]
since for all $\lambda \in S^1$ we have that 
\[
\|\widetilde{\alpha}(\widetilde{h} ) - f( \widetilde{h})\|  
= \|\widetilde{\alpha}(\lambda \cdot \widetilde{h} ) - f(\lambda \cdot \widetilde{h})\|.
\]
For the restricted homomorphisms $\widetilde{\alpha}|_{\langle \widetilde{h} \rangle}$ we
can also consider the open balls of radius $\epsilon$
   \[
B_\epsilon(\widetilde{\alpha}|_{\langle \widetilde{h} \rangle})
:=\{f \in  \hom(\langle \widetilde{h} \rangle, \calu(\calh)) \mid
 \|\widetilde{\alpha}( g) - f(g)\| < \epsilon \; \text{for all} \; g \in \langle \widetilde{h} \rangle \},
\]
and therefore we get the following equality of sets
\[
\Psi(B_\epsilon(\widetilde{\alpha})) =
\left( \prod_{h \in H} B_\epsilon(\widetilde{\alpha}|_{\langle \widetilde{h} \rangle}) \right) 
\cap \Psi\bigl(\hom_{S^1}(\widetilde{H},\calu(\calh))\bigr).
\]
This implies that the map $\Psi$ induces an open map onto its image, and since it is injective and
continuous, it induces a homeomorphism onto  its image.


Denoting by $\phi_h \colon \hom( \langle \widetilde h \rangle, \calu(\calh)) \to \hom( \langle
h \rangle, \calp\calu(\calh))$ the canonical map defined by 
$\phi_h(\gamma)(h)= \gamma(\widetilde{h})$,
we obtain the following commutative diagram
\begin{align} \label{diagram:maps_Psi_and_psi}
\xymatrix{\hom_{S^1}(\widetilde{H},\calu(\calh)) \ar[r]^-\Psi \ar[d]_\phi 
&  \prod_{h \in H} \hom( \langle \widetilde{h} \rangle, \calu(\calh)) \ar[d]^{\prod_h \phi_h} 
\\
\hom(H, \calp\calu(\calh))_{\widetilde{H}} \ar[r]_-\psi 
& \prod_{h \in H} \hom( \langle h \rangle, \calp\calu(\calh)).
}
\end{align}
Since there are canonical homeomorphisms 
\[
\hom( \langle \widetilde{h} \rangle, \calu(\calh)) 
\xrightarrow{\cong} \calu(\calh))_{|h|} \quad f \mapsto f(\widetilde{h}),
\]
by Lemma~\ref{lem:U(H)_n_local_homeo_PU(H)_n} we know that there must exist 
$\epsilon$ such that the maps
\[
\phi_h \colon B_\epsilon(\widetilde{\alpha}|_{\langle \widetilde{h} \rangle}) 
\to \phi_h(B_\epsilon(\widetilde{\alpha}|_{\langle \widetilde{h} \rangle}))
\]
are homeomorphisms for all $h \in H$; denote by 
\[
\tau_h: \phi_h(B_\epsilon(\widetilde{\alpha}|_{\langle \widetilde{h} \rangle}))
 \to B_\epsilon(\widetilde{\alpha}|_{\langle \widetilde{h} \rangle})
\] 
these inverse maps. Restricting to the open subset $B_\epsilon(\widetilde{\alpha})$ of the top left corner 
of the diagram~\eqref{diagram:maps_Psi_and_psi}, we obtain the following diagram
\begin{align*}\xymatrix@!C=7em{B_\epsilon(\widetilde{\alpha})  \ar[d]_\phi 
& \ar[l]^-\cong_-{(\Psi|_{B_\epsilon(\widetilde{\alpha})})^{-1}} \Psi(B_\epsilon(\widetilde{\alpha})) \\
\phi(B_\epsilon(\widetilde{\alpha})) \ar[r]_-\psi 
& \psi(\phi(B_\epsilon(\widetilde{\alpha}))) \ar[u]^\cong_{(\prod_h \tau_h)|_{\psi(\phi(B_\epsilon(\widetilde{\alpha})))}}
}\end{align*}
where the right hand side vertical arrow is the homeomorphism 
that the maps $\tau_h$ induce once restricted to the open set
\[
 \psi(\phi(B_\epsilon(\widetilde{\alpha}))) =\left(\prod_h \phi_h(B_\epsilon(\widetilde{\alpha}|_{\langle \widetilde{h} \rangle})) \right) 
\cap  \bigl(\prod_h \phi_h\bigr)  \Psi(B_\epsilon(\widetilde{\alpha})),\]
and the upper horizontal arrow is the inverse of $\Psi$ restricted to $B_\epsilon(\widetilde{\alpha})$.

Since the map $\psi$ is injective, we can define the map
\[
{(\Psi|_{B_\epsilon(\widetilde{\alpha})})^{-1}} \circ \bigl(\prod_h \tau_h\bigr)|_{\psi(\phi(B_\epsilon(\widetilde{\alpha})))} 
\circ \psi \colon  \phi(B_\epsilon(\widetilde{\alpha}))  \to B_\epsilon(\widetilde{\alpha})
\]
which clearly is the inverse map of $\phi$ once it is restricted to $B_\epsilon(\widetilde{\alpha})$.
This proves that $\phi$ is a local homeomorphism.
\end{proof}

The conjugation action of $\calu(\calh)$ on
$\hom_{S^1}(\widetilde{H},\calu(\calh))$ factors through the projection $\calu(\calh) \to \calp\calu(\calh)$
to an action
\[
\calp\calu(\calh)) \times \hom_{S^1}(\widetilde{H},\calu(\calh)) \to \hom_{S^1}(\widetilde{H},\calu(\calh)),
\]
since the conjugation action restricted to the center of $\calu(\calh)$ is trivial.

\begin{lemma} \label{lem:map_hom(C,U(H))_to_hom(C,PU(H))_wolfgang} 
  Let $\alpha \colon H \to \calp   \calu(\calh)$ and 
  $\widetilde{\alpha} \colon \widetilde{H} \to \calu(\calh)$ be as in
  diagram~\eqref{diagram_tilde(H)_wolfgang}. Then the existence of a local cross section for the
  conjugation map $\calu(\calh) \to \calu(\calh) \cdot \widetilde{\alpha}$ implies the
  existence of a local cross section of the conjugation map 
  $\calp \calu(\calh) \to \calp   \calu(\calh) \cdot \alpha$.
\end{lemma}

\begin{proof}
This follows from the commutativity of the diagram
\[\xymatrix{
\calu (\calh) \ar[r] \ar[d] & \calp \calu (\calh) \ar[d] \ar[r]^=& \calp \calu (\calh) \ar[d] \\
\calu (\calh) \cdot \widetilde{\alpha} \ar[r]^= & \calp \calu(\calh) \cdot \widetilde{\alpha} \ar[r]^\phi &  \calp \calu(\calh) \cdot \alpha,
}
\]
where the vertical arrows are defined by conjugation on $\widetilde{\alpha}$ and $ \alpha$
respectively, and the conclusion of
Proposition~\ref{pro:hom(c,U(H))_to_hom(C,PU(H))_principal_bundle} 
that the map $\phi$ is a local homeomorphism.
\end{proof}

\begin{lemma} \label{lem:Hom(tilde(H),U(H))_has_local_cross_sections_wolfgang}
For  $\widetilde{\alpha} \in \hom_{S^1}(\widetilde{H},\calu(\calh))$ the map induced by the conjugation action
\[
\calu(\calh) \to \calu(\calh) \cdot \widetilde{\alpha}, \quad g \mapsto  g \widetilde{\alpha} g^{-1}
\]
has a local cross section.
\end{lemma}

\begin{proof}
For $\beta, \gamma  \in \hom_{S^1}(\widetilde{H},\calu(\calh))$ let
\[
T_{\beta,\gamma} := \int_{\widetilde{H}} \beta(\widetilde{h}) \circ \gamma(\widetilde{h}^{-1})  \, d\widetilde{h}
\quad \in \calu(\calh)
\]
be the intertwining operator between $\beta$ and $\gamma$ 
appearing in~\cite[VII.1.Proposition 21]{Gaal(1973)} where $d\widetilde{h}$ is the normalized 
Haar measure on $\widetilde{H}$. We have 
\begin{eqnarray*}
\beta \circ T_{\beta,\gamma} & = &  T_{\beta,\gamma} \circ \gamma;
\\
T_{\beta,\gamma}^* & = & T_{\gamma,\beta},
\end{eqnarray*} 
since for all $a \in \widetilde{H}$ one has
\begin{eqnarray*}
T_{\beta,\gamma}  \circ \gamma(a) 
& = & 
\int_{\widetilde{H}} \beta(\widetilde{h}) \circ \gamma(\widetilde{h}^{-1}a) \, d\widetilde{h} 
\\
& = &
\int_{\widetilde{H}}\beta(a\widetilde{h}) \circ \gamma(\widetilde{h}^{-1}) \,d\widetilde{h}
\\
& = &
 \beta(a)  \circ T_{\beta.\gamma},
\end{eqnarray*}
and
\begin{eqnarray*}
T_{\beta,\gamma}^* 
& = &
\int_{\widetilde{H}} \left(\beta(\widetilde{h}) \circ \gamma(\widetilde{h}^{-1})\right)^*  \, d\widetilde{h}
\\
& = & 
\int_{\widetilde{H}} \gamma(\widetilde{h}^{-1})^* \cdot \beta(\widetilde{h})^*  \, d\widetilde{h}
\\
& = & 
\int_{\widetilde{H}} \gamma(\widetilde{h}^{-1})^{-1}\cdot \beta(\widetilde{h})^{-1}  \, d\widetilde{h}
\\
& = & 
\int_{\widetilde{H}} \gamma(\widetilde{h}) \cdot \beta(\widetilde{h}^{-1})  \, d\widetilde{h}
\\
& = & 
T_{\gamma,\beta}.
\end{eqnarray*}
Notice that since $\beta$ and $\gamma$ belong to $\hom_{S^1}(\widetilde{H},\calu(\calh))$, 
then the intertwiner $T_{\beta,\gamma}$ can also be defined by the finite sum
\begin{align} \label{finite_sum_definition_of_T_wolfgang}
T_{\beta,\gamma} := \frac{1}{|H|} \cdot \sum_{h \in H}  \beta(\widetilde{h})\circ \gamma(\widetilde{h}^{-1})  
\end{align}
where $\widetilde{h}$ is any fixed choice of lift of $h$ in $\widetilde{H}$. 
We can choose an open neighborhood of $\widetilde{\alpha}$
\[
V := \{ \beta \in  \hom_{S^1}(\widetilde{H},\calu(\calh)) 
\mid \| \alpha(\widetilde{h})-\beta(\widetilde{h}) \| < \frac{1}{|H|} \text{ for all } h \in H\}.
\]
The triangle  inequality implies that for all $\beta$ in $V$ we have $\| 1 - T_{\beta,\widetilde{\alpha}}\| < 1$
and therefore the operator $T_{\beta,\widetilde{\alpha}}$ is invertible. Hence we get  a map
\[
\tau  \colon V \to GL(\calh) , \quad \beta \mapsto T_{\beta,\widetilde{\alpha}}, 
\]
whose continuity follows from the finite sum definition of $T_{\beta,\widetilde{\alpha}}$ 
of~\eqref{finite_sum_definition_of_T_wolfgang}.  It satisfies for every $\beta \in V$ 
\begin{eqnarray*}
\tau(\beta)  \circ \widetilde{\alpha}\circ \tau(\beta)^{-1}  
& = & 
\beta
\\
\tau(\beta)^* \circ \beta  & = & \widetilde{\alpha}  \circ \tau(\beta)^*.
\end{eqnarray*}
Composing the map $\tau$ with the retraction 
\[
\rho \colon GL(\calh) \to \calu(\calh), \quad T \mapsto T \sqrt{(T^*T)}^{-1}
\]
defined in~\cite[Chapter~4]{Kuiper(1965)}, we obtain a continuous map 
\[
\sigma:= \rho \circ \tau  \colon V \to \calu(\calh)
\]
from a open neighborhood $V$ of $\widetilde{\alpha}$ in $\hom_{S^1}(\widetilde{H},\calu(\calh))$
to $\calu(\calh)$ with the desired property
\[
\sigma(\beta)  \circ \widetilde{\alpha} \circ\sigma(\beta))^{-1} = \beta.
\]
Since the action of $\calu(\calh)$ is transitive on $\calu(\calh) \cdot \widetilde{\alpha}$, the translations of the previously 
defined local cross section define local cross sections around any point in the orbit space.
\end{proof}

\begin{corollary} \label{cor:hom_S1(tilde(H),U(H))_homeo_to_disjoint_union_of_orbits} 
  The space $\hom_{S^1}(\widetilde{H},\calu(\calh))$ is homeomorphic to the disjoint union of
  its orbits under the conjugation action of $\calu(\calh)$. Hence each orbit
  $\calu(\calh) \cdot \widetilde{\alpha}$ is open and closed in
  $\hom_{S^1}(\widetilde{H},\calu(\calh))$.
\end{corollary}
\begin{proof} From the proof of
  Lemma~\ref{lem:Hom(tilde(H),U(H))_has_local_cross_sections_wolfgang} we know that for
  $\alpha,\beta \in \hom_{S^1}(\widetilde{H},\calu(\calh))$ lying in different orbits
  there must exist $h \in H$ such that for all lifts $\widetilde{h}$
  \[
  \| \alpha(\widetilde{h}) - \beta(\widetilde{h})\| \geq \frac{1}{|H|}
  \]
  and therefore the distance between $\alpha$ and $\beta$ is greater or equal than
  $\frac{1}{|H|}$. This implies that different orbits are separated by at least a distance
  of $\frac{1}{|H|}$ and therefore $\hom_{S^1}(\widetilde{H},\calu(\calh))$ is
  homeomorphic to the disjoint union of its orbits under the conjugation action of
  $\calu(\calh)$.
  \end{proof}

\begin{proof}[Proof of Theorem~\ref{the:local_cross_sections_for_Hom(H,PU(H))}]
  Lemma~\ref{lem:map_hom(C,U(H))_to_hom(C,PU(H))_wolfgang} and
  Lemma~\ref{lem:Hom(tilde(H),U(H))_has_local_cross_sections_wolfgang} imply the existence of
  local cross sections for the map
  \[
  \calp \calu(\calh) \to \calp \calu(\calh) \cdot \alpha, \quad g \mapsto g \alpha g^{-1}.
  \]
  This map is the composite of the projection 
  $\pr \colon \calp\calu(\calh) \to \calp\calu(H)/C_{\calp\calu(\calh))}(\alpha)$
  with the bijective continuous map 
  $\iota_{\alpha} \colon  \calp\calu(H)/C_{\calp\calu(\calh))}(\alpha) \to  \calp\calu(H) \cdot \alpha$.
  This implies that $\iota_{\alpha}$ has local cross sections and hence is a homeomorphism. We conclude that 
  $\pr \colon \calp\calu(\calh) \to \calp\calu(H)/C_{\calp\calu(\calh))}(\alpha)$ has local cross sections and hence
  is a principal $C_{\calp\calu(\calh))}(\alpha)$-bundle by Remark~\ref{rem:(S)_and_principal}. 
  This finishes the proof of Theorem~\ref{the:local_cross_sections_for_Hom(H,PU(H))}.
\end{proof}

\begin{proposition} \label{pro:three_items_of_Condition(H)_for_finite_H_and_PU(H)_wolfgang} Let $H$
  be a finite group and $\alpha \in \hom(H, \calp \calu(\calh))$. Then
  items~\ref{def:property_(H):component},~\ref{def:property_(H):S_for_G}
  and~\ref{def:property_(H):homeo} of Condition (H) are satisfied:
  \begin{itemize}
   
  \item[\ref{def:property_(H):component}] The path component of $\alpha$ in $\hom(H,\calp
    \calu(\calh))$ is contained in the orbit $\calp \calu(\calh) \cdot \alpha$;
 
   \item[\ref{def:property_(H):S_for_G}] The canonical map $\calp \calu(\calh) \to \calp
    \calu(\calh)/C_{\calp \calu(\calh)}(\alpha)$ is a $C_{\calp
      \calu(\calh)}(\alpha)$-principal bundle;
  
   \item[\ref{def:property_(H):homeo}] The canonical map
    \[\iota_{\alpha} \colon \calp \calu(\calh)/C_{\calp \calu(\calh)}(\alpha) \to
    \hom(H,\calp \calu(\calh)), \quad gC_{\calp \calu(\calh)}(\alpha) \mapsto g \alpha
    g^{-1}
    \]
    is a homeomorphism onto its image.
  \end{itemize}
\end{proposition}
\begin{proof}
Consider the commutative diagram
\[\xymatrix{
\calu(\calh) \cdot \widetilde{\alpha} \ar@{^{(}->}[r] \ar[d]^\phi 
& \hom_{S^1}(\widetilde{H},\calu (\calh)) \ar@{^{(}->}[r] \ar[d]^\phi
& \hom(\widetilde{H},\calu (\calh)) \ar[d] 
\\
\calp \calu(\calh) \cdot {\alpha} \ar@{^{(}->}[r]
&  \hom(H,\calp \calu (\calh))_{\widetilde{H}} \ar@{^{(}->}[r] 
& \hom(H,\calp \calu (\calh).
} \]
By Lemma~\ref{lem:hom(H,PU(H))_wilde(H)_open_closed} we know that $\hom(H,\calp \calu
(\calh))_{\widetilde{H}}$ is open and closed in $\hom(H,\calp \calu (\calh)$, by
Proposition~\ref{pro:hom(c,U(H))_to_hom(C,PU(H))_principal_bundle} we know that the map
$\phi$ is a principal $\hom(H,S^1)$-bundle and by
Corollary~\ref{cor:hom_S1(tilde(H),U(H))_homeo_to_disjoint_union_of_orbits} we know that
$\calu(\calh) \cdot \widetilde{\alpha}$ is both open and closed in
$\hom_{S^1}(\widetilde{H},\calu(\calh))$. This implies that $\calp \calu(\calh) \cdot
{\alpha}$ is open and closed in $\hom(H,\calp \calu (\calh)$; this
proves~\ref{def:property_(H):component}.

The other two conditions~\ref{def:property_(H):S_for_G} and~\ref{def:property_(H):S_for_Gamma}  have already been proved in
Theorem~\ref{the:local_cross_sections_for_Hom(H,PU(H))}.
\end{proof}


\subsection{Almost free $\Gamma$-equivariant stable projective unitary bundles.}

Let $\Gamma$ be a Hausdorff topological group and consider the family of all finite
subgroups $\calf \cali \caln(\Gamma)$ of $\Gamma$. Let $\calr=\calr(\calf \cali
\caln(\Gamma))$ be the associated family of local representations
\[
\calr=\{ (H,\alpha) \mid H \in \calf \cali \caln(\Gamma) \text{ and } 
\alpha \colon H \to \calp \calu(\calh) \text{ any group homomorphism} \}.
\]
A $\Gamma$-$CW$-complex $X$ is called \emph{almost free} if all its isotropy groups are finite.

\begin{theorem}[Universal $\Gamma$-equivariant projective unitary 
 bundle for almost free $\Gamma$-$CW$-complexes]
\label{the:universal_bundle_projective_unitary_almost_free_wolfgang}
Let $\Gamma$ be a Hausdorff topological group. Then the bundle
  \[
  \calp \calu(\calh) \to E(\Gamma,\calp\calu(\calh),\calr) \stackrel{p}{\to}
  B(\Gamma,\calp\calu(\calh),\calr)
  \]
  is a universal $\Gamma$-equivariant projective unitary bundle for almost free   $\Gamma$-$CW$-complexes
\end{theorem}
\begin{proof}
  The result follows from
  Theorem~\ref{the:Classifying_space_for_gamma-equivariant_principal_G-bundles} since
  items~\ref{def:property_(H):component},~\ref{def:property_(H):S_for_G}
  and~\ref{def:property_(H):homeo} of Condition (H) were proved in
  Proposition~\ref{pro:three_items_of_Condition(H)_for_finite_H_and_PU(H)_wolfgang} and
  item~\ref{def:property_(H):S_for_Gamma} follows from the fact that $\Gamma$ is Hausdorff
  and the subgroups $ H \in \calf \cali \caln(\Gamma)$ are finite.
\end{proof}

Certainly one can associate a $\Fred(\calh)$-bundle to any $\Gamma$-equivariant projective
unitary bundle and one can take the homotopy groups of its $\Gamma$-equivariant
sections. This procedure applied to some $\Gamma$-equivariant projective unitary bundles
may fail to produce the expected twisted equivariant K-theory groups, let us see why: 
for $H$ finite subgroup of $\Gamma$ and $\alpha \colon H \to \calp\calu(\calh)$ any group
homomorphism, we can construct the $\Gamma$-equivariant principal
$\calp\calu(\calh)$-bundle
\[
\Gamma \times_\alpha \calp\calu(\calh) \to \Gamma/H 
\]
where $\Gamma \times_\alpha \calp\calu(\calh)$ is the quotient of $\Gamma \times
\calp\calu(\calh)$ under the left $H$-action given by $h\cdot (\gamma, g) = (\gamma
h^{-1},\alpha(h)g)$. The associated bundle
\[
\left(\Gamma \times_\alpha \calp\calu(\calh) \right) \times_{\calp\calu(\calh)} \Fred(\calh) \cong \Gamma \times_\alpha  \Fred(\calh)
\]
is equivalent to the quotient of $ \Gamma \times \Fred(\calh)$ under the left $H$-action
given by $h \cdot (\gamma, F) = (\gamma h^{-1},\alpha(h)F \alpha(h)^{-1})$. Therefore the
space of $\Gamma$-equivariant sections of the bundle
 \[
 \Gamma \times_\alpha  \Fred(\calh) \to \Gamma/H
 \]
is homeomorphic to the space of $\alpha$-invariant operators
\[
\Fred(\calh)^\alpha := \{F \in \Fred(\calh) \mid \alpha(h)F \alpha(h)^{-1}=F \text{ for all } h \in H \}.
\]

Now, the index map
\begin{align*}
\Index\colon \Fred(\calh)^\alpha \to R(\widetilde{H}) & & F \mapsto \ker(F) - \coker(F)
\end{align*}
maps an $\alpha$-invariant operator to an element of the Grothendieck ring of
representations of $\widetilde{H}:=\alpha^* \calu(\calh)$, group which was defined in
diagram~\eqref{diagram_tilde(H)_wolfgang}.

In order for the homotopy groups of $\Fred(\calh)^\alpha$ to represent the local
coefficients for twisted equivariant K-theory, we know that we need to impose two
conditions on the homomorphism $\alpha\colon  H \to \calp \calu(\calh)$:
\begin{itemize}

\item The image of the index map is a subgroup of $R(\widetilde{H})$; this in order for
  the connected components $\pi_0( \Fred(\calh)^\alpha)$ to be a group, and

\item All the representations of $\widetilde{H}$, on which $S^1=\ker(\widetilde{H} \to H)$
  acts by multiplication, must appear on the image of the index map; this in order to
  obtain a theory compatible with restriction of representations.

\end{itemize}
If a homomorphism $\alpha$ satisfies these two conditions, then the index map
\[
\Index \colon \pi_0(\Fred(\calh)^\alpha) \stackrel{\cong}{\to} R_\alpha(H),
\]
 induces an isomorphism between the path components of $\Fred(\calh)^\alpha$ and the
Grothendieck group
\[
R_\alpha(H) := \{ V \in R(\widetilde{H}) \mid S^1=\ker(\widetilde{H} \to H)\text{ acts by multiplication on } V \}
\]
of so called $\alpha$-twisted $H$-representations 

Homomorphisms that satisfy these two
conditions will be called {\it{stable}}, cf.~\cite[Section~6]{Atiyah-Segal(2004)}.  

\begin{definition}[Stable homomorphism] \label{def:stable_homomorphism_wolfgang} A homomorphism
  $\alpha \colon H \to \calp \calu(\calh)$ is \emph{stable} if all irreducible
  representations of $\widetilde{H}$ on which $S^1=\ker(\widetilde{H} \to H)$ acts by
  multiplication appear infinitely number of times.
\end{definition}

Define the set $\cals$ of stable local representations for almost free $\Gamma$-actions
\[
\cals=\{ (H,\alpha) \mid H \in \calf \cali \caln(\Gamma) \text{ and } 
\alpha \colon H \to \calp \calu(\calh) \;\text{is a stable homomorphism} \}.
\]
Note that $\cals$ is indeed a family of local representations since conjugation of stable
homomorphisms is also stable and the restriction of a stable homomorphism to a subgroup is
stable; this last statement follows from the Frobenius Reciprocity Theorem since all
irreducible $\alpha|_K$-twisted representations of $K$ subgroup of $H$ may be obtained
from the $\alpha$-twisted representations of $H$.

Now let $B$ be an almost free $\Gamma$-$CW$-complex.  Then a $\Gamma$-equivariant principle
$\calp\calu(\calh)$-bundle $p \colon E \to B$ whose family of local representations is
contained in $\cals$ is the same as a $\Gamma$-equivariant stable projective unitary
bundle over $B$ in the sense
of~\cite[Definition~2.2]{Barcenas-Espinoza-Joachim-Uribe(2012)}. Note that the existence of the
local data appearing in~\cite[Definition~2.2]{Barcenas-Espinoza-Joachim-Uribe(2012)} is
automatically satisfied by Theorem~\ref{the:local_objects}.

We obtain from Theorem~\ref{the:Classifying_space_for_gamma-equivariant_principal_G-bundles} 
a universal $\Gamma$-equivariant stable projective unitary bundle $p \colon
E(\Gamma,\calp\calu(\calh),\cals) \to B(\Gamma,\calp\calu(\calh),\cals)$ such that for any
almost free $\Gamma$-$CW$-complex $X$ the pullback construction yields a bijection from
$\bigl[X,B(\Gamma,\calp\calu(\calh),\cals)\bigr]^{\Gamma}$ to the set of isomorphism
classes of $\Gamma$-equivariant stable projective unitary bundles over $X$.  This
generalizes~\cite[Theorem~3.4]{Barcenas-Espinoza-Joachim-Uribe(2012)}, where $\Gamma$ is
assumed to be discrete.  Summarizing:

\begin{theorem}[Universal $\Gamma$-equivariant stable projective bundle]
 \label{the:classification_of_equivariant_stable_projective_unitary_bundles} 
  Let $\Gamma$ be a Hausdorff topological group. 
 Then the bundle
  \[
  \calp \calu(\calh) \to E(\Gamma,\calp\calu(\calh),\cals) \stackrel{p}{\to}
  B(\Gamma,\calp\calu(\calh),\cals)
  \]
  is a universal $\Gamma$-equivariant stable projective unitary bundle for almost free
  $\Gamma$-$CW$-complexes.
\end{theorem}
For $X$ a proper $\Gamma$-$CW$-complex,  the set of 
the isomorphism classes of projective unitary $\Gamma$-equivariant stable
bundles over $X$ has been denoted by  ${\rm{Bun}}^\Gamma_{st}(X, \calp\calu(\calh))$ 
in~\cite[Definition~2.2]{Barcenas-Espinoza-Joachim-Uribe(2012)}; we have
then the canonical isomorphisms of sets
\[
\\{\rm{Bun}}^\Gamma_{st}(X, \calp \calu(\calh)) =  \Bundle_{\Gamma,\calp
  \calu(\calh),\cals}(X)
\stackrel{\cong}{\to}\bigl[X,B(\Gamma,\calp\calu(\calh),\cals)\bigr]^{\Gamma}.
\]

Therefore the $\Gamma$-equivariant stable projective unitary bundle
 \[
  \calp \calu(\calh) \to E(\Gamma,\calp\calu(\calh),\cals) \stackrel{p}{\to}
  B(\Gamma,\calp\calu(\calh),\cals)
  \]
  is the universal twist in equivariant K-theory for almost free $\Gamma$-CW-complexes. In
  particular, for a fixed $\Gamma$-equivariant map $f\colon X \to   B(\Gamma,\calp\calu(\calh),\cals)$ 
  with $X$ almost free $\Gamma$-CW-complex, the twisted
  equivariant K-theory groups of the pair $(X,f)$ can be defined as
\[
K_\Gamma^{-i}(X,f) := \pi_i\bigl({\rm{Sections}}(f^*E(\Gamma,\calp\calu(\calh),\cals) \times_{\calp\calu(\calh)} \Fred(\calh))^\Gamma\bigr),
\]
namely, as the homotopy groups of the space of $\Gamma$-equivariant sections of the
associated $\Fred(\calh)$-bundle; see~\cite[Appendix~A]{Barcenas-Espinoza-Joachim-Uribe(2012)} 
for further properties of the twisted equivariant K-theory groups defined in this way.

Now, in the case that $\Gamma$ is furthermore a Lie group we can show that there is an
isomorphism between the set $\Bundle^\Gamma_{st}(X, \calp \calu(\calh))$ and $H^3( E
\Gamma \times_\Gamma X, \mathbb{Z})$.

\begin{theorem} \label{the:Bundles_and_H_upper_three}
   Let $\Gamma$ be a Lie group and $X$ be an almost free $\Gamma$-CW complex.

     Then the map
    \[
    \map(\id_X,\psi) \colon \map (X, B(\Gamma,\calp\calu(\calh),\cals)) \to \map  (X,\map(E\Gamma, B\calp\calu(\calh))), 
    \quad f \mapsto \psi \circ f
    \]
    is a weak $\Gamma$-homotopy equivalence. 

    In particular we obtain bijections
    \begin{multline*}
     \Bundle_{\Gamma,\calp\calu(\calh,\cals}(X) \cong  [X, B(\Gamma, \calp\calu(\calh) ,\cals)]^{\Gamma}  
     \xrightarrow{\cong}  [X,\map(E\Gamma, B\calp\calu(\calh))]^{\Gamma}
     \\ 
          [E\Gamma \times_{\Gamma} X,B\calp\calu(\calh)] 
     \cong
     H^3(E\Gamma \times_\Gamma X, \mathbb{Z}).
   \end{multline*}
    \end{theorem}
    \begin{proof}
    For $H$ a finite group of $\Gamma$ we just need to show that the induced map
    \[
    \psi^H\colon B(\Gamma, \calp\calu(\calh) ,\cals)^H 
    \to \map(E\Gamma, B\calp\calu(\calh))^H \simeq \map(BH, B\calp\calu(\calh))
    \]
    is a weak homotopy equivalence; the rest of the proof is formal and it is equivalent 
    to the one that appears in the proof of Theorem~\ref{the:compact_abelian_G}.
    
    Let $\hom_\cals(H,\calp \calu(\calh))$ be the set of stable homomorphisms. 
    In~\cite[Proposition~1.5]{Barcenas-Espinoza-Joachim-Uribe(2012)}    it is shown that the map
    \[
   \hom_\cals(H,\calp \calu(\calh))/\calp \calu(\calh) \stackrel{\cong}{\to} \Ext(H,S^1),
    \quad     \alpha \mapsto \widetilde{H}=\alpha^*\calu(\calh)
   \]
    is an isomorphism of sets, where $\Ext(H,S^1)$ denotes the set of isomorphism 
    classes of central $S^1$-extensions of $H$.  Now, we have the isomorphism of sets 
    \[
    \Ext(H,S^1)\cong H^3(BH,\mathbb{Z}) \cong [BH,B\calp \calu(\calh)]
    \]
    since $\calp \calu(\calh)$ is $K(\mathbb{Z},2)$-space by 
     Kuiper's Theorem~\cite{Kuiper(1965)}. Since by Theorem~\ref{the:Fixed_point_sets_of_B(Gamma,G,calr)}  we know that 
    \[
    \hom_\cals(H,\calp \calu(\calh))/\calp \calu(\calh) \stackrel{\cong}{\to} \pi_0(B(\Gamma, \calp\calu(\calh) ,\cals)^H),
    \]
    we obtain the desired isomorphism at the level of connected components 
    \[
    \pi_0(\psi^H)\colon \pi_0(B(\Gamma, \calp\calu(\calh) ,\cals)^H) 
    \xrightarrow{\cong} \pi_0(\map(E\Gamma, B\calp\calu(\calh))^H).
    \]
    
    Take now $\alpha \in  \hom_\cals(H,\calp \calu(\calh))$. By Theorem~\ref{the:Fixed_point_sets_of_B(Gamma,G,calr)} 
     we have a weak homotopy equivalence
    \[
    BC_{ \calp\calu(\calh)}(\alpha) \stackrel{\simeq}{\to}B(\Gamma, \calp\calu(\calh) ,\cals)^H_\alpha;
    \]
    by~\cite[Theorem~1.8]{Barcenas-Espinoza-Joachim-Uribe(2012)} we know that the homotopy groups 
    of $C_{ \calp\calu(\calh)}(\alpha) $ are
    \[
    \pi_j(C_{ \calp\calu(\calh)}(\alpha)) = \left\{ \begin{matrix}
    \hom(H,S^1) & \text{for} & j=0; \\
    \mathbb{Z} & \text{for} & j=2; \\
    0 && \text{otherwise.} 
      \end{matrix} \right.
    \]
    By a simple calculation the homotopy groups for $j >0$ of $\map(E\Gamma, B\calp\calu(\calh))^H$ are
    \[
    \pi_j(\map(E\Gamma, B\calp\calu(\calh))^H) = \left\{ \begin{matrix}
    H^2(BH,\mathbb{Z})=\hom(H,S^1)& \text{for} & j=1; \\
    H^0(BH,\mathbb{Z})=\mathbb{Z} & \text{for} & j=3; \\
    0 & &\text{otherwise.}
      \end{matrix} \right.
    \]
    Therefore we see that the homotopy groups of $B(\Gamma, \calp\calu(\calh) ,\cals)^H$ 
    and $\map(E\Gamma, B\calp\calu(\calh))^H$
    are isomorphic and concentrated in degrees $j = 0$ and $j=2$.
    
    Now, in order to show that indeed $\psi^H$ induces the desired isomorphism on homotopy
    groups we need to resort to the proof of
    Theorem~\ref{the:Fixed_point_sets_of_B(Gamma,G,calr)}. For $E=E(\Gamma,
    \calp\calu(\calh) ,\cals)$ we know that
    \[
    E|_{B(\Gamma, \calp\calu(\calh) ,\cals)^H_\alpha} \cong \calp\calu(\calh) \times_{C_{\calp\calu(\calh)}(\alpha)} E^{K(H,\alpha)}
    \]
    with $E^{K(H,\alpha)} \to B(\Gamma, \calp\calu(\calh) ,\cals)^H_\alpha$ a principal $C_{ \calp\calu(\calh)}(\alpha)$-bundle. 
     Therefore the map that defines $\psi^H$
    \[
    E\Gamma/H \times B(\Gamma, \calp\calu(\calh) ,\cals)^H_\alpha \to B \calp\calu(\calh)
    \]
    is weakly homotopically equivalent to the map obtained by applying the classifying space functor 
    \[
    BH \times BC_{ \calp\calu(\calh)}(\alpha) \to B \calp\calu(\calh)
    \]
    to the group homomorphism 
    \[
    \phi \colon H \times C_{\calp\calu(\calh)}(\alpha) \to \calp\calu(\calh). \quad  (h,g) \mapsto \alpha(h) g.
    \]
    
    Since all the $\widetilde{H}$ representations defined by $\alpha$ appear infinitely
    number of times, then it follows that the inclusion of groups 
     $C_{\calp\calu(\calh)}(\alpha) \subset \calp\calu(\calh)$ induces an isomorphism on the
    second homotopy groups $\pi_2$; this implies that $\psi^H$ induces an isomorphism on
    the third homotopy groups $\pi_3$. It remains to show that $\psi^H$ induces an
    isomorphism at the level of fundamental groups.
    
   The map $\phi$ permits to define a homeomorphism
  \[
   C_{\calp\calu(\calh)}(\alpha) \to \hom_\alpha(\mathbb{Z} \times H ,  \calp\calu(\calh)), \quad
  g \mapsto \left[(n,h) \mapsto g^n\alpha(h) \right].
  \]
   where
   \[
  \hom_\alpha(\mathbb{Z} \times H ,  \calp\calu(\calh)) := \{\beta \in \hom_\alpha(\mathbb{Z} \times H ,  \calp\calu(\calh)) \mid
  \beta(0,h)=\alpha(h) \}.
 \]
  Moreover we can define a group homomorphism (see~\cite[Lemma~1.6]{Barcenas-Espinoza-Joachim-Uribe(2012)})
  \[
  C_{\calp\calu(\calh)}(\alpha)  \to \hom(H,S^1),
  \quad   g  \mapsto (h \mapsto \widetilde{\alpha}(\widetilde{h}) \circ  \overline{g} 
  \circ \widetilde{\alpha}(\widetilde{h})^{-1} \circ \overline{g}^{-1}),
  \]
  where $\widetilde{h}$ is any lift of $h$ in $\widetilde{H}$ and $\overline{g}$ is any lift  
  of $g$ on $\calu (\calh)$, such that
  it becomes an isomorphism of groups on path connected components 
  \[
  \pi_0(C_{ \calp\calu(\calh)}(\alpha) ) \xrightarrow{\cong} \hom(H,S^1).
  \]
  Therefore we obtain the following commutative diagram
  \[\xymatrix{
  \pi_0(C_{\calp\calu(\calh)}(\alpha)) \ar[r]^-{\cong}  \ar[d]^-{\cong} 
   & \pi_0( \hom_\alpha(\mathbb{Z} \times H ,  \calp\calu(\calh))) \ar[r]
   &
  \pi_0(\map_{B\alpha}(B \mathbb{Z} \times BH, B \calp\calu(\calh))) \ar[d]^{\cong} 
   \\
  \hom(H,S^1) \ar[r]^\cong 
   & 
   [BH,BS^1] \ar[r]^-{\cong} 
   & 
   [BH, \Omega B  \calp\calu(\calh) ]
   }  
  \]
  where
  \[
  \map_{B\alpha}(B \mathbb{Z} \times BH, B \calp\calu(\calh))
   := \{ f\colon  B \mathbb{Z} \times BH \to B \calp\calu(\calh) \colon f(*,y)= B\alpha(y)  \},  
  \]
  the right vertical arrow is given by the adjoint map since $B\mathbb{Z} \simeq S^1$ and $\Omega$ denotes the based loop space functor, 
  and the bottom right map is defined via the homotopy equivalence
  $BS^1 \stackrel{\simeq}{\to} \Omega B \calp\calu(\calh)$ which
  induces the homotopy equivalence $BS^1 \to \calp\calu(\calh)
  $. Hence we conclude that at the level of fundamental groups the map
  $\phi$ induces an isomorphism
  \[
  \pi_1(BC_{ \calp\calu(\calh)}(\alpha)) \stackrel{\cong}{\to} \pi_1(\map(BH,B  \calp\calu(\calh)))
  \]
  and therefore $\psi^H$ induces the desired isomorphism at the level
  of fundamental groups.
\end{proof}


\section{Appendix A:Compactly generated spaces}
\label{sec:Appendix_A:Compactly_generated_spaces}

We briefly recall some basics about compactly generated spaces. More information and
proofs can be found in~\cite{Steenrod(1967)}. A topological space $X$ is \emph{compactly
  generated} if it is a Hausdorff space and a set $A \subseteq X$ is closed if and only if
for any compact subset $C \subset X$ the intersection $C \cap A$ is a closed subspace of
$C$.

Every locally compact space, and every space satisfying the first axiom of countability,
e.g., a metrizable space, is compactly generated. If $p \colon X \to Y$ is an
identification of topological spaces and $X$ is compactly generated and $Y$ is Hausdorff,
then $Y$ is compactly generated.  A closed subset of a compactly generated space is again
compactly generated. For open subsets one has to be careful as it is explained 
in Subsection~\ref{subsec:open_subsets}.


\subsection{Open subsets}
\label{subsec:open_subsets}

Recall that a topological space $X$ is called \emph{regular} if for any point $x \in X$
and closed set $A \subseteq X$ there exists open subsets $U$ and $V$ with $x \in U$, $A
\subseteq V$ and $U \cap V = \emptyset$.  A space is called \emph{locally compact} if
every $x \in X$ possesses a compact neighborhood.  Equivalently, for every $x \in X$ and
open neighborhood $U$ there exist an open neighborhood $V$ of $x$ such that the closure of
$V$ in $X$ is compact and contained in $U$, 
see~\cite[Lemma~8.2 in Section~3-8 on page~185]{Munkres(1975)}.
Recall the Definition~\ref{def:quasi-regular} saying that an open subset $U \subseteq X$ is called \emph{quasi-regular} 
if for any $x \in U$ there exists an open neighborhood $V_x$ whose closure in $X$ is contained in $U$.

\begin{lemma} \label{lem:properties_of_quasi-regular_open_sets} 

\begin{enumerate}

\item \label{lem:properties_of_quasi-regular_open_sets:open_subset} 
Let  $B$ be a (compactly generated) space. 
A quasi-regular open subset $U \subseteq B$ equipped with the subspace topology 
is compactly generated;

\item \label{lem:properties_of_quasi-regular_open_sets:preimage} 
Let $f \colon X \to Y$ be a (continuous) map between (not necessarily compactly generated) 
spaces. If $V \subseteq Y$ is a quasi-regular open subset, then $f^{-1}(V) \subseteq X$ is a
quasi-regular open subset.

\item \label{lem:properties_of_quasi-regular_open_sets:finite_intersection} 
The intersection of finitely many quasi-regular open subsets is again a quasi-regular open subset;

\item \label{lem:properties_of_quasi-regular_open_sets:regular} 
A space is regular if and only if every open subset is quasi-regular;

\item \label{lem:properties_of_quasi-regular_open_sets:equivalence:locally_compact_or_metrizable} 
Any locally compact space, any metrizable space and every $CW$-complex are regular;

\item \label{lem:properties_of_quasi-regular_open_sets:Gamma-CW} 
Every $\Gamma$-invariant open subset of a $\Gamma$-$CW$-complex is quasi-regular and,
equipped with the subspace topology, compactly generated.
\end{enumerate}
\end{lemma}
\begin{proof}~\ref{lem:properties_of_quasi-regular_open_sets:open_subset} 
See~\cite[page~135]{Steenrod(1967)}.
\\[2mm]~\ref{lem:properties_of_quasi-regular_open_sets:preimage}
Consider a point $x \in f^{-1}(V)$. Choose an open set $W$ of $Y$ such that
$f(x) \in W$ and the closure of $W$ in $B$ is contained in $V$.
Then $f^{-1}(W)$ is an open subset of $X$ which contains $x$ 
and whose closure in $X$ is contained in $f^{-1}(V)$.
\\[2mm]~\ref{lem:properties_of_quasi-regular_open_sets:finite_intersection} 
Let $U_1$, $U_2$, \ldots,  $U_r$ be quasi-regular open subsets.
Consider $x \in U := \bigcap_{i=1}^r U_i$. Choose for every $i = 1,2 \ldots, r$ an open subset
$V_i$ with $x \in V_i$ such that the closure $\overline{V_i}$ of $V_i$ in $B$ is contained in $U_i$.
Put $V: = \bigcap_{i=1}^r V_i$. Then $x \in V$ and $\overline{V} \subseteq \cap_{i=1}^r \overline{V_i} \subseteq U$.
Hence $U$ is a quasi-regular open subset.
\\[2mm]~\ref{lem:properties_of_quasi-regular_open_sets:regular} 
See~\cite[Lemma~2.1 in Section~4-2 on page~196]{Munkres(1975)}.
\\[2mm]~\ref{lem:properties_of_quasi-regular_open_sets:equivalence:locally_compact_or_metrizable} 
This is obvious for locally compact spaces.
Metrizable spaces are treated in~\cite[Theorem~2.3 in Section~4-2 on page~198]{Munkres(1975)}. 
Every $CW$-complex is paracompact, see~\cite{Miyazaki(1952)}, and hence in particular regular, 
see~\cite[Theorem~4.1 in Section~6-4 on page~255]{Munkres(1975)}.
\\[2mm]~\ref{lem:properties_of_quasi-regular_open_sets:Gamma-CW} 
The projection $X \to X/\Gamma$ is open and $X/\Gamma$ is a $CW$-complex
by Lemma~\ref{lem:quotient_of_CW-complexes}.
Now apply assertions~\ref{lem:properties_of_quasi-regular_open_sets:open_subset},~%
\ref{lem:properties_of_quasi-regular_open_sets:regular} 
and~\ref{lem:properties_of_quasi-regular_open_sets:equivalence:locally_compact_or_metrizable} 
\end{proof}


\subsection{The retraction functor $k$}
\label{subsec:The_retraction_functor}

 There is a construction which assigns to a topological Hausdorff space $X$ a new topological
space $k(X)$ such that $X$ and $k(X)$ have the same underlying sets, $k(X)$ is compactly
generated, $X$ and $k(X)$ have the same compact subsets, the identity $k(X) \to
X$ is continuous and is a homeomorphism if and only if $X$ is compactly generated.
Namely,  define the new topology on $k(X)$ by declaring a subset $A \subseteq X$ to be 
closed if and only if for every compact subset of $X$ the intersection $A \cap C$ 
is a closed subset of $C$.


\subsection{Mapping spaces, product spaces and subspaces}
\label{subsec:mapping_spaces_product_spaces_and_subspaces}

Given two compactly generated spaces $X$ and $Y$, denote by $\map(X,Y)_{k.o.}$ the set of
maps $X \to Y$ with the compact-open-topology, i.e., a subbasis for the
compact-open-topology is given by the sets $W(C,U) = \{f \colon X \to Y \mid f(C)
\subseteq U\}$, where $C$ runs through the compact subsets of $X$ and $U$ runs though the
open subsets of $Y$.  Notice that $\map(X,Y)_{k.o.}$ is not compactly generated in general.
We denote by $\map(X,Y)$ the topological space given by $k(\map(X,Y)_{k.o.})$.

If $X$ and $Y$ are compactly generated spaces, then $X \times Y$ stands for $k(X \times_p X)$, 
where $ X \times_p Y$ is the topological space with respect to the ``classical''
product topology. 

If $A \subseteq X$ is a subset of a compactly generated space, the subspace topology means
that we take $k(A_{st})$ for $A_{st}$ the topology space given by the ``classical''
subspace topology on $A$.

Roughly speaking, all the usual constructions of topologies are made compactly generated by
passing from $Y$ to $k(Y)$ in order to stay within the category of compactly generated
spaces.


\subsection{Basic features of the category of compactly generated spaces}
\label{subsec:Basic_feature_of_the_category_of_compactly_generated_spaces}

The category of compactly generated spaces has the following convenient features:

\begin{itemize}

\item A map $f \colon X \to Y$ of compactly generated spaces is continuous if and only if
  its restriction $f|_C \colon C \to Y$ to any compact subset $C \subseteq X$ is
  continuous;

\item The product of two identifications is again an identification;

\item If $X$ is locally compact and $Y$ compactly generated, then $X \times Y$ and $X
  \times_p Y$ are the same topological spaces;

\item The product of a $\Gamma_1$-$CW$-complex and of a $\Gamma_2$-$CW$-complex is a
  $\Gamma_1 \times \Gamma_2$-$CW$-complex;

\item If $X$, $Y$, and $Z$ are compactly generated spaces, then the obvious maps
  \begin{eqnarray*}
    \map(X,\map(Y,Z)) & \xrightarrow{\cong} &\map(X \times Y,Z);
    \\
    \map(X, Y \times Z) & \xrightarrow{\cong} & \map(X,Y) \times \map(X,Z).
  \end{eqnarray*}
  are homeomorphisms and the map given by composition
  \[
  \map(X,Y) \times \map(Y;Z) \to \map(X,Z)
  \]
  is continuous;

\item Given a pushout in the category of compactly generated spaces, its product with a
  compactly generated space is again a pushout in the category of compactly generated
  spaces.
\end{itemize}


\subsection{Space of homomorphisms}
\label{subsec:Space_of_homomorphisms}

Let $H$ and $G$ be (compactly generated) topological groups.
Let $\hom(H,G)$ be the set of group homomorphisms from $H$ to $G$.
It is obviously a subset of $\map(H,G)$.  The proof of the next result is a
typical formal proof using the convenient basic properties of the category of compactly
generated spaces.

\begin{lemma} \label{lem:hom(H,G)_is_closed_in_map(H,G)}
The subset $\hom(H,G)$ of $\map(H,G)$ is closed.
\end{lemma}
\begin{proof}
Consider the map 
\[
u \colon H \times H \times \map(H,G) \to G, \quad (h_1,h_2,\alpha) \mapsto \alpha (h_1) \cdot \alpha (h_2) \cdot \alpha (h_1h_2)^{-1}.
\]
It is continuous since it  can be written as a composition of maps
each of which is obviously continuous by the basic features presented in
Subsection~\ref{subsec:Basic_feature_of_the_category_of_compactly_generated_spaces}.
\begin{multline*}
H \times H \times \map(H,G) \xrightarrow{u_1} H \times \map(H,G) \times H \times \map(H,G) \times H \times \map(H,G) 
\\
\xrightarrow{u_2}  G \times G \times G \xrightarrow{u_3} G
\end{multline*}
where 
\begin{eqnarray*}
u_1(h_1,h_1,\alpha )  
& = & 
(h_1,\alpha ,h_2,\alpha ,h_1h_2,\alpha );
\\
u_2(h_1,\alpha _1,h_2,\alpha _2,h_3,\alpha _3) 
& = & 
(\alpha _1(h_1), \alpha _2(h_2), \alpha _3(h_3));
\\
\alpha_3(g_1,g_2,g_3) 
& = & 
g_1g_2g_3^{-1}.
\end{eqnarray*}
Its adjoint is the continuous map
\[
v \colon \map(H,G) \to \map(H \times H,G), \quad \alpha \mapsto v(\alpha)
\]
given by $v(\alpha )(h_1,h_2) := \alpha (h_1) \cdot \alpha (h_2) \cdot \alpha (h_1h_2)^{-1}$.
The evaluation map
\[
\ev \colon \map(H,G) \to  G, \quad \alpha  \mapsto \alpha (1)
\]
is continuous by Subsection~\ref{subsec:Basic_feature_of_the_category_of_compactly_generated_spaces}.
Since $\hom(H,G)$ is the intersection of $v^{-1}(\{1\})$ and $\ev^{-1}(\{1\})$ and $\{1\} \subseteq G$ is closed,
$\hom(H,G)$ is a closed subset of $\map(H,G)$.
\end{proof}

We will equip $\hom(H;G)$  with the subspace topology coming from $\map(H,G)$.
Notice that it is automatically compactly generated so that $\hom(H;G)$ agrees with
$k(\hom(H;G))$.

We leave the formal proof  to the reader that the conjugation action
\[
G \times \hom(H,G) \to \hom(H,G), \quad (g,\alpha) \mapsto c_g \circ \alpha
\]
is continuous, where $c_g \colon G \to G$ sends $g'$ to $gg'g^{-1}$.

We also include the following results

\begin{lemma} \label{lem:identifications_induce_closed_embeddings}
Let $p \colon X \to Y$ be an identification and $Z$ be a space. Then  the induced map
\[
p^* \colon \map(Y,Z) \to \map(X,Z), \quad f \mapsto f \circ p
\]
is a closed embedding, i.e., has closed image and the induced map
$q \colon \map(Y,Z) \to \im(p^*)$ is a homeomorphism.
\end{lemma}
\begin{proof}
Obviously $p^*$ is injective and continuous. 

Next we  show that the image of $p^*$ is closed. Consider $f \colon X \to Z$ which is not in the image of $p^*$.
Then  there exists $x_0$ and $x_1$ in $X$ with $f(x_0) \not= f(x_1)$ and $p(x_0) = p(x_1)$. 
Since the map 
\[\ev \colon \map(X,Z) \mapsto Z \times Z, \quad f \mapsto \bigl(f(x_0), f(x_1)\bigr)
\]
is continuous, the preimage of the open set $\{(z_0,z_1) \in Z \times Z \mid z_0 \not= z_1\}$ under $\ev$ is an open set
which contains $f$ and does not meet $\im(p^*)$. Hence $\im(p^*)$ is closed in $\map(X,Z)$.

It remains to show that the inverse of  $q$ 
\[
q^{-1}  \colon \im(p^*)  \to \map(Y,Z) 
\]
is continuous. Consider an open subset $V \subset \map(Y,Z)_{k.o.}$.
By the definition of the compact open topology we can find an index set $I$, for every
element $i \in I$ a finite index set $J_i$, and for every $i \in I$ and $j_i \in J_i$ an open subset
$U_{i,j_i} \subseteq Y$ and a compact subset $C_{i,j_i} \subseteq Z$ such that
\[
V= \bigcup_{i \in I} \bigcap_{j_i \in J_i} W(U_{i,j_i}, C_{i,j_i})
\]
where $W(U_{i,j_i}, C_{i,j_i}) = \{f \in \map(Y,Z) \mid f(U_{i,j_i}) \subseteq C_{i,j_i}\}$. We have
\begin{eqnarray*}
q(V) 
& = & 
p^*\left( \bigcup_{i \in I} \bigcap_{j_i \in J_i} W(U_{i,j_i}, C_{i,j_i})\right)
\\
& = &
\bigcup_{i \in I} \bigcap_{j_i \in J_i} p^*\bigl( W(U_{i,j_i}, C_{i,j_i})\bigr) 
\\
& = &
\bigcup_{i \in I} \bigcap_{j_i \in J_i} \bigl( W\bigl(p^{-1}(U_{i,j_i}), C_{i,j_i}) \cap \im(p^*)\bigr)
\\
& = &
\left(\bigcup_{i \in I} \bigcap_{j_i \in J_i} W\bigl(p^{-1}(U_{i,j_i}), C_{i,j_i}\bigr)\right) \cap \im(p^*)
\end{eqnarray*}
where $W\bigl(p^{-1}(U_{i,j_i}), C_{i,j_i}\bigr) = \{f \in \map(X,Z) \mid f(p^{-1}(U_{i,j_i})) \subseteq C_{i,j_i}\}$.
Since $\bigcup_{i \in I} \bigcap_{j_i \in J_i} W\bigl(p^{-1}(U_{i,j_i}), C_{i,j_i}\bigr)$ is an open subset of
$\map(Y,Z)_{k.o.}$, the subset $q(V)$ of  $\im(p^*)_{k.o.} $ is open where the topology on $\im(p^*)_{k.o.} $ is the subspace topology
with respect to the inclusion $\im(p^*) \to \map(X,Z)_{k.o.}$. Hence the map $q^{-1}_{k.o.} \colon \im(p^*)_{k.o.} \to \map(Y,Z)_{k.o.}$ is continuous.
This implies that the map 
\[
k(q^{-1}_{k.o.}) \colon k\bigl(\im(p^*)_{k.o.}\bigr) \to k\bigl(\map(Y,Z)_{k.o.}\bigr) = \map(Y,Z)
\]
 is continuous.
The identity induces a continuous map $i_0 \colon \map(X,Z) = k(\map(X,Z)_{k.o.}\bigr) \to  \map(X,Z)_{k.o.}$ and hence
by restriction a continuous map $i_1 \colon \im(p^*) \to \im(p^*)_{k,o.}$. If we apply $k$ to it and use the fact
that $\im(p^*)$ is compactly generated, we obtain a continuous map
\[
i_1 \colon  \im(p^*) \to k(\im(p^*)_{k.o.}).
\]
Since  $q^{-1}  \colon \im(p^*)  \to \map(Y,Z)$ is the composite of the two continuous maps
$k(q^{-1}_{k.o.})$ and $i_1$ above, it is continuous itself. This finishes the proof
of Lemma~\ref{lem:identifications_induce_closed_embeddings}.
\end{proof}

\begin{lemma} \label{lem:quotient_map_is_continuous}
For every group $G$ and $G$-spaces $Y$ and $Z$ the map 
\[
Q_{Y,Z} \colon \map_G(Y,Z) \to \map_G(Y/G,Z/G), \quad f \mapsto f/G
\]
is continuous.
\end{lemma}
\begin{proof}
Let $p_Y\colon Y \to Y/G$ and $p_Z \colon Z \to Z/G$ be the projections. 
The map 
\[
(p_Z)_* \colon \map(Y,Z) \to \map(Y,Z/G), \quad f \mapsto p_Z \circ f
\]
is continuous and hence induces a continuous map
\[
p_Z \colon \map_G(Y, Z)\to \map_G(Y,Z/G), \quad f \mapsto p_Z \circ f.
\]
Therefore  the claim follows if we can prove that the map
\[
Q_{Y,Z/G} \colon \map_G(Y,Z/G) \to \map(Y/G,Z/G), \quad f \mapsto f/G
\]
is continuous.  Its inverse is the map
\[
p_Y^* \colon \map(Y/G,Z/G)  \to \map_G(Y,Z/G) \quad f \mapsto f \circ p_Y
\]
which is a homeomorphism by Lemma~\ref{lem:identifications_induce_closed_embeddings}.
This finishes the proof of Lemma~\ref{lem:quotient_map_is_continuous}.
\end{proof}


\subsection{Further features of the category of compactly generated spaces}
\label{subsec:Further_feature_of_the_category_of_compactly_generated_spaces}

\begin{lemma} \label{lem_pushouts-pullbacks}
Consider a commutative square of $\Gamma$-spaces
\[
\xymatrix{B_0 \ar[r]^{i_1} \ar[d]^{i_2} & B_1 \ar[d]^{j_1}
\\
B_2 \ar[r]^{j_2} & B
}
\]
Then 
\begin{enumerate}

\item \label{lem_pushouts-pullbacks:equivariant_versus_non-equivariant}
The square is a $\Gamma$-pushout if and only if it is pushout after forgetting the group action.

\item \label{lem_pushouts-pullbacks:pullbacks}
If the square above is a $\Gamma$-pushout and $f \colon E \to B$ is a $\Gamma$-map,
then the square obtained by the pullback construction
\[
\xymatrix{E_0 \ar[r]^{\overline{i_1}} \ar[d]^{\overline{i_2}} & E_1 \ar[d]^{\overline{j_1}}
\\
E_2 \ar[r]^{\overline{j_2}} & E
}
\]
is a $\Gamma$-pushout;

\item \label{lem_pushouts-pullbacks:quotients} 
If the given square is a $\Gamma$-pushout, the square
\[
\xymatrix{B_0/\Gamma \ar[r]^{i_1/\Gamma} \ar[d]^{i_2/\Gamma} & B_1/\Gamma \ar[d]^{j_1/\Gamma}
\\
B_2/\Gamma \ar[r]^{j_2/\Gamma} & B/\Gamma
}
\]
is a pushout of spaces.
\end{enumerate}
\end{lemma}
\begin{proof}~\ref{lem_pushouts-pullbacks:equivariant_versus_non-equivariant}
Suppose that the square is a (non-equivariant) pushout. We want to show that it is
a $\Gamma$-pushout. Consider $\Gamma$-maps
$f_k \colon B_k \to Y$ for $k = 0,1,2$ satisfying $f_k \circ i_k = f_0$ for $k = 1,2$.
Then there is precisely one map $f \colon B \to X$ with $f \circ j_k = f_k$ for $k = 1,2$.
It remains to show that $f$ is $\Gamma$-equivariant. This follows from the fact that the product
of $\Gamma$ with the square above is again a (non-equivariant) pushout. This follows formally from the adjunctions
of mapping spaces and the universal property of pushouts.

Now suppose  that the square is a $\Gamma$-pushout. Then it is a pushout after forgetting the group action,
since there is a bijection
\[
a \colon \map(B,X) \xrightarrow{\cong} \map_\Gamma(B, \map(\Gamma,X)), \quad f \mapsto a(f),
\]
where $a(f)$ sends $b \in B$ to the map $\Gamma \to X,\; \gamma \mapsto f(\gamma \cdot b)$,
$X$ is a (non-equivariant) space and $\map(\Gamma,X)$ becomes a $\Gamma$-space 
$\gamma \cdot f(h) := f(h\gamma)$.
\\[2mm]~\ref{lem_pushouts-pullbacks:pullbacks}
Because of assertion~\ref{lem_pushouts-pullbacks:equivariant_versus_non-equivariant},
we can assume without loss of generality that $\Gamma$ is trivial. The map
$j_1 \amalg j_2 \colon B_1 \amalg B_2 \to B$ is an identification since the given square is a pushout.
Then $\id_E \times j_1 \amalg \id_E \times j_2 \colon E \times B_1 \amalg E \times B_2 \to E \times B$
is an identification, see~\cite[Theorem~4.4]{Steenrod(1967)}.  Consider $E$ as a 
closed subspace of $E \times B$ by identifying $e$ with $(e,f(e))$.
Then the restriction of $\id_E \times j_1 \amalg \id_E \times j_2$ to 
the preimage of $E$ is again an identification
which can be identified with $\overline{j_1} \amalg \overline{j_2} \colon E_1 \amalg E_2 \to E$.
Obviously the square obtained by the pullback construction is a pushout of sets. Hence it is a pushout 
of spaces.
\\[2mm]~\ref{lem_pushouts-pullbacks:quotients} This follows from the universal properties
of a pushout and the projection maps $B \to B/\Gamma$.
\end{proof}

\begin{lemma}\label{lem:filtrations}
Let $B_0 \subseteq B_1 \subseteq B_2 \subseteq \cdots \subseteq B$
be a filtration of the $\Gamma$-space $B$ by closed $\Gamma$-invariant subspaces.

\begin{enumerate}

\item \label{lem:filtrations:equivariant_versu_non_equivariant} 
We have $B = \colim_{n \to \infty} B_n$ in the category of $\Gamma$-spaces if and only if
we have after  forgetting the $\Gamma$-actions $B = \colim_{n \to \infty} B_n$ in the category of spaces;

\item \label{lem:filtrations:preimages}
Suppose that we have $B = \colim_{n \to \infty} B_n$ in the category of $\Gamma$-spaces.
Let $f \colon E \to B$ be a $\Gamma$-map. Then we obtain a filtration 
\[
E_0 \subseteq E_1 \subseteq E_2 \subseteq \cdots \subseteq E
\]
of $E$ by closed $\Gamma$-invariant subspaces
$E_n = f^{-1}(B_n)$ with the property that $E = \colim_{n \to \infty} E_n$ holds in the category of $\Gamma$-spaces;

\item \label{lem:filtrations:quotients} 
Suppose that $B = \colim_{n \to \infty} B_n$ holds in the category of $\Gamma$-spaces.
Then 
\[
B_0/\Gamma \subseteq B_1/\Gamma \subseteq B_2/\Gamma \subseteq \cdots \subseteq B/\Gamma
\]
is a filtration of the space $B\Gamma$ by closed subspaces with the property that 
$B/\Gamma = \colim_{n \to \infty} B_n/\Gamma$ holds  in the category of spaces.

\end{enumerate}
\end{lemma}
\begin{proof}~\ref{lem:filtrations:equivariant_versu_non_equivariant}
The proof is similar to the one of 
Lemma~\ref{lem_pushouts-pullbacks}~\ref{lem_pushouts-pullbacks:equivariant_versus_non-equivariant}, 
using the fact that after forgetting the group actions we obtain a filtration by closed subspaces
$\Gamma\times B_0 \subseteq \Gamma\times B_1 \subseteq \Gamma\times B_2 \subseteq \cdots \subseteq \Gamma\times B$
such that $\Gamma\times B = \colim_{n \to \infty} \Gamma\times B_n$ holds in the category of $\Gamma$-spaces,
see~\cite[Theorem~10.3]{Steenrod(1967)}.
\\[2mm]~\ref{lem:filtrations:preimages}
Let $A \subseteq E$ be a subset. Suppose that $A \cap E_n$ is closed. We have to show that $A \subseteq E$ is closed,
or, equivalently, that for every compact subset $C \subset E$ the space $C \cap A$ is a closed subspace of $C$.
Since $f(C)$ is compact, there exists a natural number $n$ with $f(C) \subseteq B_n$. This implies
$C \subseteq E_n$, see~\cite[Lemma~9.3]{Steenrod(1967)}. 
Now the claim follows from $A \cap C = A \cap (C \cap E_n) = (A \cap E_n ) \cap C$
since $A \cap E_n$ is closed in $E_n$ and hence $(A \cap E_n ) \cap C$ is closed in $C$.
\\[2mm]~\ref{lem:filtrations:quotients} 
This follows from the universal properties of $\colim_{n \to \infty}$ and of the quotient maps
$B_k \to B_k/\Gamma$ and $B \to B/\Gamma_k$.
\end{proof}

\begin{lemma} \label{lem:quotient_of_CW-complexes} If $X$ is a
  $G$-$CW$-complex. Let $N \subseteq G$ a normal subgroup and $Q$ be
  the topological group $G/N$ and $\pr \colon G \to Q$ be the projection. Suppose for any $x \in X$ 
  that $\pr(G_x)$ is closed in $Q$.  (This assumption is automatically satisfied if $N$
  is compact or $Q$ is trivial.) 

  Then $X/N$ is a $Q$-$CW$-complex.
\end{lemma}
\begin{proof} Let $\pr \colon G\to Q$ be the projection. Consider a subgroup $H \subseteq G$
such that $\pr(H)$ is closed in $Q$.  Then we obtain a $Q$-homeomorphism 
$N\backslash \bigl(G/H\bigr) \xrightarrow{\cong} Q/\pr(H)$.
Now the claim follows from
Lemma~\ref{lem_pushouts-pullbacks}~\ref{lem_pushouts-pullbacks:equivariant_versus_non-equivariant}
and~\ref{lem_pushouts-pullbacks:quotients} 
and Lemma~\ref{lem:filtrations}~\ref{lem:filtrations:equivariant_versu_non_equivariant} 
and~\ref{lem:filtrations:quotients}.
\end{proof}


  \section{Appendix B: Some properties of locally compact groups}
  \label{subsection:Appendix_B:Some_properties_of_locally_compact_groups_old}

Let $H$ and $G$ be topological groups. 
For $\alpha$ in $\hom(H,G)$, let $C_G(\alpha)$ be the centralizer of $\alpha$, i.e.,
$C_G(\alpha) =\{g \in G \mid c_g \circ \alpha = \alpha \}$,
and denote by $G \cdot \alpha$ the orbit of $\alpha$ under the $G$-action. Then the map
\[
\iota_\alpha \colon G / C_G(\alpha) \to G \cdot \alpha, \ \ \ g C_G(\alpha) \mapsto c_g \circ
\alpha
\]
is bijective, continuous and $G$-equivariant. The next two theorems are the main results of this section.

\begin{theorem} \label{the:connected_components_and_conjugation}
  Let $H$ be a compact group, let $G$ a locally compact second countable group and $\alpha \in
  \hom(H,G)$. 

  Then the connected  component $\mathcal{C}_\alpha$ of $\alpha$ in $\hom(H,G)$ is
  contained in the orbit of $\alpha$ under $G$, i.e. $\mathcal{C}_\alpha \subset G\cdot
  \alpha$.
\end{theorem}

\begin{theorem} 
  \label{the:G/C_G(alpha)_homeo_G.alpha} 
  Let $G$ be a second countable,
  locally compact group which  is almost connected. Let $H$  be a compact group.

 Then $  \iota_\alpha \colon G/C_G(\alpha) \to G \cdot \alpha$ is a $G$-homeomorphism.
\end{theorem}

Their respective proofs need some preparation.
For a compact space $X$ and a complete metric space $(Y,d)$ we can equip
the set of continuous maps from $X$ to $Y$ with the supremums metric
\begin{align} \label{supremums_metric}
d_{\sup}(\alpha,\beta) := \sup\{d(\alpha(x),\beta(x)) \mid x \in X\},
\end{align}
and obtain a complete metric space, 
see~\cite[Theorem~1.4 in Section~7-1 on page~267]{Munkres(1975)}. 
Moreover, the topology induced
by the supremums metric agrees with the compact-open topology
by~\cite[Theorem~4.6 in Section~7-4 on page~283 and Theorem~5.1 in Section~7.5 on page~286]{Munkres(1975)}. 
In particular $\map(X,Y) = \map(X,Y)_{c.o.}$. 

The following result is due to Birkhoff~\cite{Birkhoff(1936)} and Kakutani~\cite{Kakutani(1936)}.

\begin{theorem}  \label{the:Birkhoff-Kakutani}
  A Hausdorff topological group is metrizable if and only if it is first countable. In
  this case the metric can be taken to be left invariant.
\end{theorem}

Theorem~\ref{the:Birkhoff-Kakutani} implies that for a locally compact first countable group $G$ we can choose a left invariant
metric $d_G$ inducing the given topology such that $G$ with this metric is complete.
Note that completeness follows from local compactness. Namely,  we then can find an $\epsilon > 0$ such that
the closed ball $\overline{B}_{\epsilon}(1)$ around $1$ is compact and hence the closed ball $\overline{B}_{\epsilon}(g)$
around any $g \in G$ is compact which implies that every Cauchy sequence contains a subsequence
contained in $B_{\epsilon}(g)$ for some $g \in G$ and hence contains a convergent subsequence,
by~\cite[Theorem~7.4 in Section~3-7 on page~181]{Munkres(1975)}, and finally we can
apply~\cite[Lemma~1.1 in Section~7-1 on page~264]{Munkres(1975)}

Recall from Subsection~\ref{subsec:Space_of_homomorphisms} that we have equipped
$\hom(H,G)$ with the subspace topology of $\map(H,G)$, that 
$\hom(H,G) \subseteq \map(H,G)$ is closed and that the conjugation action
\[
G \times \hom(H,G)   \to \hom(H,G ), \quad  (g,\alpha) \mapsto c_g \circ \alpha
\] 
is continuous. Hence this topology on $\hom(H,G)$ agrees with the topology coming from the
compact-open topology on $\map(H,G)$ restricted to $\hom(H;G)$ as well with the topology
associated to the supremums metric restricted to $\hom(H,G)$, and $\hom(H,G)$ is a
complete metric space with the supremums metric.

The next Theorem is taken from~\cite[Theorem~I]{Lee-Wu(1970)} and obviously implies
Theorem~\ref{the:connected_components_and_conjugation}.

\begin{theorem} \label{the:Lee-Wu_connected_components_Hom(H,G)} 
  Let $H$ be a compact
  group and $G$ a locally compact second countable group. Let $\mathcal{C} \subset \hom(H,G)$ be a
  connected  component of the space $\hom(H,G)$. Then if $\theta$ and $\phi$ are in
  $\mathcal{C}$, then there exists $g \in G$ with $\theta = c_g \circ \phi$.
\end{theorem}

The following result is taken from Lee-Wu~\cite[Theorem~II]{Lee-Wu(1970)}. Notice
that the condition of being almost connected is necessary, see~\cite[Example, page 412]{Lee-Wu(1970)}.

\begin{theorem} \label{the:Theorem-II_Lee-Wu(1970)} Let $L$ be a locally compact group
  which is almost connected. Let $F$ be a compact subgroup of $L$ and denote by 
 $i \colon  F \to L$ the inclusion. If $\{x_\lambda \mid \lambda \in \Lambda \}$ is a net in $L$ such
  that the homomorphisms $c_{x_\lambda} \circ i$ converge to an element $\theta \in
  \hom(F,L)$, then there exists an element $y \in L$ such that $\theta = c_y \circ i$.
\end{theorem}

\begin{theorem} \label{the:G_alpha_closed_in_Hom(H,G)} 
 Let $G$ be a locally compact group
  which is almost connected.  Let $H$ be a compact group and
  $\alpha \in \hom(H,G)$.  Then the orbit  $G \cdot \alpha$ of $\alpha$ under the conjugation 
  action of $G$ is closed in $\hom(H,G)$.
\end{theorem}
 \begin{proof}
  Suppose we  have a net $\{g_\lambda \in G \mid \lambda \in \Lambda \}$ of elements of $G$ such that the
  homomorphisms $c_{g_\lambda} \circ \alpha$ converge to an element 
  $\phi \in   \hom(H,G)$. We want to show that $\phi$ belongs to $G \cdot \alpha$.

  Let us take $L = H \times G$ and $K(H,\alpha) :=\{(h, \alpha(h) \mid h \in H \}$ the
  associated compact subgroup of $H \times G$ defined by $\alpha$. We have then that $L$
  is locally compact and $L$ is almost connected. Denote by 
  $i \colon   K(H,\alpha) \to H \times G$ the inclusion and take the net 
  $\{(1,g_\lambda) \in H \times  G \mid \lambda \in \Lambda \}$ of elements in $H \times G$ 
  induced by the elements $g_\lambda \in G$. Note that since the homomorphisms $c_{g_\lambda} \circ \alpha$
  converge to $\phi \in \hom(H,G)$, the homomorphisms
  \[
  c_{(1,g_\lambda)} \circ i \in \hom(K(H,\alpha),H \times G)
  \]
  converge to the homomorphism $\theta_\phi \in \hom(K(H,\alpha),H \times G)$ defined by
  \[
  \theta_\phi(h,\alpha(h)) := (h, \phi(h)).
  \]
   Applying Theorem~\ref{the:Theorem-II_Lee-Wu(1970)} we know that there exists a pair
  $(y,g) \in H \times G$ such that $\theta_\phi = c_{(y,g)} \circ i$; this implies that
  for all $h \in H$
  \[
  (h, \phi(h))= ( yhy^{-1}, g \alpha(h) g^{-1})
  \]
  and therefore we can conclude that $\phi = c_g \circ \alpha$. Hence $\phi$ belongs to the
  orbit $G \cdot \alpha$.
\end{proof}

Finally  we give the  proof Theorem~\ref{the:G/C_G(alpha)_homeo_G.alpha}.
\begin{proof}[Proof of Theorem~\ref{the:G/C_G(alpha)_homeo_G.alpha}]

  Since $G$ is  second countable, it is a Lindel\"of space, i.e., 
  every open covering contains a countable subcovering, 
  see~\cite[Theorem~1.3  in Section~4-1 on page~191]{Munkres(1975)}.

  By Theorem~\ref{the:G_alpha_closed_in_Hom(H,G)} we know that the orbit $G \cdot \alpha$
  is closed in $\hom(H,G)$. We have already explained
 that  $\hom(H,G)$ is a complete  metric space. We conclude that the orbit $G
  \cdot \alpha$ is a complete metric space and hence by the Baire
Category Theorem, see~\cite[Theorem~7.2 in Section~7.7 on page~294]{Munkres(1975)},
also a Baire space, i.e., every countable union of closed sets each of which has empty interior has itself empty interior.

  Next we show that the map $ \iota_\alpha$ is an open map. Since we know that it is a
  bijective continuous $G$-map, this will imply that $\iota_\alpha$ is indeed a
  $G$-homeomorphism. Since the image and source of $\iota_\alpha$ are transitive $G$-spaces,
  it suffices to show for every  open set $U$ of $G$ containing $1_G \in G$
  that there exists an open neighborhood of $\alpha$ which is contained in $\iota_{\alpha}(U)$.

  Since $G$ is locally compact and the map $\mu \colon G \times G \to G, \; (g_1,g_2)  \mapsto g_1^{-1} g_2$
  is continuous, we can find a non-empty open set $V \subset U$ such that its closure $K=\overline{V}$ is compact and
  satisfies 
  \[
  K^{-1} K := \{k_1^{-1} \cdot k_2 \mid k_1 , k_2 \in K\}  \subset U;
  \] any nonempty open set $V \subset U$ such that $\overline{V} \times \overline{V} \subset \mu^{-1}(U) \cap U \times U$ satisfy
  this condition.

  Since $G$ is a Lindel\"of space, it can be covered by countably many $G$-translates of the open
  subset $V$ and hence also by countably many $G$-translates of the subset $K$. 
  This implies that $G/C_G(\alpha)$ can be covered by countably many $G$-translates of
  $KC_G(\alpha)/C_G(\alpha)$. Since each 
  \[
  \iota_\alpha(g KC_G(\alpha)/C_G(\alpha))= gK\cdot \alpha
  \]
  is a closed set in $G \cdot \alpha$ and $G \cdot \alpha$ is a Baire space,
  one of the $G$-translates of $KC_G(\alpha)/C_G(\alpha)$ must have non empty
  interior. This implies that $K \cdot \alpha$ has non empty interior,
  and therefore the set $K^{-1} K \cdot \alpha $ contains an open neighborhood of $\alpha$.  
  Since $K^{-1}K \alpha$ is contained in $i_{\alpha}(U)$, the claim follows.
\end{proof}




\end{document}